\documentclass{article}

\usepackage{blindtext}
\usepackage[a4paper, total={6in, 8in}]{geometry}
\usepackage{biblatex}
\usepackage[utf8]{inputenc}

\usepackage{csquotes}
\usepackage{authblk}
\usepackage[english]{babel}
\usepackage{amsthm}
\usepackage{amssymb}
\usepackage{amsmath}
\usepackage{amscd}
\usepackage{hyperref}
\usepackage{MnSymbol}
\usepackage{appendix}

\hypersetup{
    colorlinks=true,
    linkcolor=red,
    filecolor=magenta,      
    urlcolor=cyan,
    pdftitle={Ultralimit and Poisson boundary},
    pdfpagemode=FullScreen,
    }

\newtheorem{theorem}{Theorem}[subsection]
\newtheorem{corollary}[theorem]{Corollary}
\newtheorem{lemma}[theorem]{Lemma}
\newtheorem{definition}[theorem]{Definition}
\newtheorem{example}[theorem]{Example}
\newtheorem{prop}[theorem]{Proposition}
\newtheorem{remark}[theorem]{Remark}
\newtheorem{problem}{Problem}
\newtheorem{notation}[theorem]{Notation}
\newtheorem{claim}[theorem]{Claim}
\newtheorem{thmx}{Theorem}

\newcommand{\bigslant}[2]{{\raisebox{.2em}{$#1$}\left/\raisebox{-.2em}{$#2$}\right.}}

\title{Ultralimits, Amenable actions and Entropy\footnote{The author acknowledge the ISF support made through grant 1483/16.}}
\author{Elad Sayag\footnote{elad.sayag2003@gmail.com}}
\date{}
\affil[ ]{School of Mathematical Sciences, Tel-Aviv University, Israel}

\begin{document}

\maketitle
\begin{abstract}
In this paper we show that the minimal value of Furstenberg entropy (along all measures, not restricting to stationary ones) for any amenable action is the same as for the action of the group on itself. Using the boundary amenability result of \cite{ADAMS1994765}, this allows us to compute the minimal value of the entropy over all the measure classes in the boundary of the free group. Similar results are proved for the action of a hyperbolic group on its Gromov boundary.\\
Our main tool is an ultralimit realization of the Poisson boundary of a time dependent matrix-valued random walk on the group. \\
This extends and refines the results and tools of \cite{sayag2022entropy}.
\end{abstract}

\tableofcontents

\section{Introduction}

In this paper we continue the study initiated in \cite{sayag2022entropy} on the {problem of entropy minimizing}. We will relate this problem in the case of amenable actions (see Theorem \ref{Theorem entropy for amenable intro}) and boundary actions (see Theorems \ref{Theorem Entropy for hyperbolic groups intro}, \ref{Theorem entropy for lattices intro}). In particular, we calculate the minimal entropy number for the action of the free group on its boundary (see Theorem \ref{Theorem entropy for free groups intro}). We will give a new perspective on amenable actions by providing an ultralimit construction for them, starting with the action of the group on itself (see Theorem \ref{Theorem ultralimit ameanble action intro}). 
A secondary goal of this paper is to develop further the tool of ultralimits from a measure theoretic point of view. Building on \cite{conley_kechris_tucker-drob_2013}, \cite{MR390154}, \cite{MR2964622} and dealing with unbounded functions. We will also relate it to the construction of ultralimits in \cite{sayag2022entropy}.

\subsection{The problem of minimizing entropy}

Let \(G\) be a discrete countable group. Throughout this introduction \(\lambda\) will be a finitely supported, generating probability measure on \(G\) and \(f\) will be a convex function on \((0,\infty)\) with \(f(1)=0\).\\
In \cite{sayag2022entropy}, we introduced the notion of Furstenberg \((\lambda,f)\)-entropy. Namely, for a \(G\)-space \(X\) and a quasi-invariant probability measure \(\nu\) on \(X\), we define
\[h_{\lambda,f}(X,\nu):=\sum_{g\in G}\lambda(g)D_{f}(g\nu||\nu)=\sum_{g\in G}\lambda(g)\intop_{X} f(\frac{dg\nu}{d\nu})\:d\nu\]
Our main object of interest is the \textbf{minimal entropy number} defined for a \(G\)-space \(X\) as follows:
\[I_{\lambda,f}(X)=\inf_{\nu\in M(X)} h_{\lambda,f}(X,\nu)\]
Here, by a \(G\)-space we mean that \(X\) is a Borel space equipped with a fixed measure class which is \(G\)-invariant, \(M(X)\) denotes the collection of probability measures on \(X\) of this measure class. \\
One has the following topological variant - let \(X\) be a topological \(G\)-space and let \(M_{top}(X)\) be the collection of all Borel measures. Define 
\[I_{\lambda,f}^{top}(X)=\inf_{\nu\in M_{top}(X)} h_{\lambda,f}(X,\nu)\]
We note that although \(\nu\) is not assumed to be quasi-invariant one can easily adapt the definition for this case (see subsection \ref{subsection Recollection of Entropy}).\\
In \cite{sayag2022entropy}, we proved that when \(G\) is amenable one has \(I_{\lambda,f}(G)=0\). In the present paper we extend this result to amenable actions in the sense of Zimmer (e.g. \cite{zimmer1978amenable}):
\begin{thmx}[See Theorem \ref{Theorem entropy and amenability}]\label{Theorem entropy for amenable intro}
Let \(G\) be a discrete countable group and let \(S\) be an amenable \(G\)-space. Then for any \(\lambda,f\) as above we have:  
\[I_{\lambda,f}(S)=I_{\lambda,f}(G)\]
\end{thmx}
Our proof of Theorem \ref{Theorem entropy for amenable intro} is based on the ideas of \cite{sayag2022entropy}, an ultralimit realization of Poisson boundaries of time-dependent matrix-valued random walks, as well as Elliott-Giordano realization of amenable actions as such Poisson boundaries (e.g. \cite{Amenable}).\\
Theorem \ref{Theorem entropy for amenable intro} combined with the work of Adams on boundary amenability of the Gromov boundary of hyperbolic groups (e.g. \cite{ADAMS1994765}) implies the following:
\begin{thmx}[See Theorem \ref{Theorem Entropy for hyperbolic groups}]\label{Theorem Entropy for hyperbolic groups intro}
Let \(G\) be a hyperbolic group and \(\partial G\) be its Gromov boundary. Then for any \(\lambda,f\) as above we have: 
\[I^{top}_{\lambda,f}(\partial G)=I_{\lambda,f}(G)\]
\end{thmx}
Similarly we will get:
\begin{thmx}[See Theorem \ref{Theorem Entropy for lattices}]\label{Theorem entropy for lattices intro}
Let \(G\) be a discrete subgroup in a semi-simple Lie group \(\mathbf{G}\). Let \(\mathbf{B}\leq \mathbf{G}\) be a Borel subgroup and consider the flag space \(X = \mathbf{G}/\mathbf{B}\), then for any \(\lambda, f\) as above we have:
\[I^{top}_{\lambda,f}(X)=I_{\lambda,f}(G)\]
\end{thmx}

\(\)\\
We now focus on the free group \(F_{d}\) on \(d\) generators \(\{a_1,\dots,a_{d}\}\).\\
Denote by \(\Delta_d\) the set of generating symmetric probability measures supported on \(\{a_1^{\pm1},\dots,a_{d}^{\pm1}\}\). Given a \(\mu\in\Delta_{d}\) we denote by \(\nu_{\mu}\) the \(\mu\)-harmonic measure on the boundary \(\partial F_{d}\).\\
Suppose additionally that \(f\) is strictly convex and smooth. In \cite[Definition 7.4]{sayag2022entropy} we introduced a bijection \(T:\Delta_d\to\Delta_d\).\\
Let \(\lambda\in\Delta_{d}\) and \(\mu=T^{-1}(\lambda)\),
in \cite{sayag2022entropy} we proved \(I_{\lambda,f}(F_{d})=h_{\lambda,f}(\partial F_{d},\nu_{\mu})\) and we showed that the measure \(\nu_{\mu}\) on the boundary \(\partial F_{d}\) minimizes \((\lambda,f)\)-entropy \underline{in its measure class}.\\
A natural question left open in \cite{sayag2022entropy} is whether \(\nu_{\mu}\) minimizes the \((\lambda,f)\)-entropy among \underline{all measures} on the boundary, that is, whether \(I_{\lambda,f}^{top}(\partial F_{d})=h_{\lambda,f}(\nu_{\mu})\). One of the main goals of this paper was to prove this equality:
\begin{thmx}[See Theorem \ref{Theorem entropy for boundary of free groups}]\label{Theorem entropy for free groups intro}
In the notations above, for any \(\lambda\in\Delta_d\) we have:
\[I_{\lambda,f}^{top}(\partial F_{d})=h_{\lambda,f}(\partial F_d,\nu_{\mu})\]
where \(\mu=T^{-1}(\lambda)\).
\end{thmx}

\subsection{Ultralimits}

\subsubsection{An ultralimit construction for amenable actions}

Amenability of \(G\)-spaces was first studied by Zimmer, see \cite{zimmer1978amenable},\cite{Zimmer1984ErgodicTA}. Since then this property was generalized to various settings (see \cite{anantharaman2000amenable}) and studied extensively.
In order to prove Theorem \ref{Theorem entropy for amenable intro}, we provide a new perspective on this notion. For this we use the construction of ultralimit of uniformly bounded quasi-invariant (BQI) \(G\)-spaces from \cite[Definition 4.17]{sayag2022entropy}. Based on this, our main tool is an ultralimit construction for (ergodic) amenable actions:
\begin{thmx}[See Theorem \ref{Theorem Amenable actions and ultralimit}]\label{Theorem ultralimit ameanble action intro}
Let \(G\) be a discrete countable group, and \(S\) is an ergodic amenable \(G\)-space. Then there is a BQI measure \(\nu\) on \(S\) (in the measure class) and a sequence of uniformly BQI measures \((\lambda_n)_{n\geq 0}\) on \(G\times\mathbb{N}\) such that for any non-principle ultrafilter \(\mathcal{U}\) on the natural numbers we have an isomorphism between the Radon-Nikodym factors:
\[(S,\nu)_{RN}\cong \mathcal{U}\lim\:_{RN}\: (G\times\mathbb{N},\lambda_n)\]
\end{thmx}

\begin{remark}
In \cite{weakcont} it was proved that unitary twisted \(L^{2}\)-representations induced from amenable actions are weakly contained in the regular representation. This was our motivation to speculate Theorem \ref{Theorem ultralimit ameanble action intro}.
\end{remark}
To prove Theorem \ref{Theorem ultralimit ameanble action intro}, we first use the result of \cite{Amenable} that represents such ergodic amenable action as generalized Poisson boundary. That is, Poisson boundary of a matrix-valued time dependent random walk on \(G\). Such a Poisson boundary is attached to a sequence \(\boldsymbol{\sigma}=(\sigma^{(n)})\) consisting of stochastic matrices of measures on \(G\). For the classical Furstenberg-Poisson boundary \(\mathcal{B}(G,\mu)\) we introduced an ultralimit construction \cite[Theorem E]{sayag2022entropy}. In Theorem \ref{Theorem ultralimit construction to Poisson boundary of a time dependent matrix-valued random walk} we generalize this construction to the general case.\\
In \cite{sayag2022entropy}, the verification that the construction coincides with the Furstenberg-Poisson boundary was based on a classical maximality property of the Furstenberg-Poisson boundary among \(\mu\)-stationary systems. We provide an analogue of the classical theory (of Furstenberg-Glasner \cite{FurstenbergGlasner}) to this general case in Appendix \ref{appendix Poisson boundary of a time dependent matrix-valued random walk}.

\subsubsection{Ultralimits in measure theoretic setting}
The tool of ultralimit in the setting of BQI \(G\)-spaces was introduced in \cite{sayag2022entropy} using the tool of ultralimit of \(C^{*}\)-algebras. We remark that one could also have used Von-Neumann algebras (equipped with a state) as in \cite{ando2014ultraproducts}.\\ 
Those approach has the disadvantage that they deal mostly with bounded functions. However, considering actions \(G\curvearrowright X\) it is very common to consider unbounded yet integrable functions or even quasi-invariant measure which are not bounded quasi-invariant.\\
This motivates us to develop the construction of ultralimit from a measure theoretic point of view. The construction of the ultralimit measure space is due to Loeb (see \cite{MR390154}) and we recall it briefly.\\ 
We will prove a basic theorem about interchanging the order of integral with ultralimits. This is an analogue of Lebesgue-Vitali's Theorem about uniform integrability.\\
To state and prove this result, we introduce a filtration on \(L^{1}\) of any probability space by Banach spaces of 'uniformly-integrable-functions'. The basic properties of those spaces will be established in Appendix \ref{appendix uniform integrability}, and we will use this formalism for our Vitali's type theorem as well as showing that this filtration behaves well under ultralimits (see Proposition \ref{proposition compartion in C rho of ultralimit}). 
\(\)\\
This is used in the deduction of Theorem \ref{Theorem entropy for amenable intro} from Theorem \ref{Theorem ultralimit ameanble action intro} -- we need to be able to deal with all the measures in the measure class of \(S\), that may not be in the same bounded measure class as \(\nu\).

\section{The method of ultralimits}\label{section ultralimit of probability spaces}
In this section we will study the  ultralimit of probability spaces and quasi-invariant \(G\)-spaces.\\
We begin by reviewing the ultralimit construction for probability spaces in subsection \ref{subsection Preliminaries- the construction of ultralimits of probability spaces}.\\
Next, in subsection \ref{subsection Ultralimit and integration for uniformly integrable functions} we prove a  Lebesgue-Vitali type theorem about interchanging of integration and ultralimits (see Theorem \ref{Theorem ultralimit and integrals}). For this we appeal to Appendix \ref{appendix uniform integrability} where we introduced a filtration of \(L^{1}\) by spaces uniform integrability.\\
In subsection \ref{subsection Comparison with ultralimit of C*-algebras} we compare the ultralimit of probability spaces and of Banach spaces attached to these probability spaces.\\
In subsection \ref{section Equivariant ultralimit}
we consider the equivariant version for this construction.

\subsection{The construction of ultralimit of probability spaces}\label{subsection Preliminaries- the construction of ultralimits of probability spaces}
In this section we construct the ultralimit of probability spaces.\\
We follow \cite{conley_kechris_tucker-drob_2013} which is based on \cite{MR2964622} and the work of Loeb \cite{MR390154}.\\
Let \(\mathcal{I}\) be a set, let \(\mathcal{U}\) be an ultrafilter on \(\mathcal{I}\). We let \(\mathfrak{X_i}=(X_i,\Sigma_i,\nu_i)\) be a collection of probability spaces parameterized by \(i\in\mathcal{I}\).\\
In this section we will construct a probability space \(\mathcal{U}\lim \mathfrak{X}_i=(X,\Sigma,\nu)\), the \emph{ultralimit} of \(\mathfrak{X_i}\).\\
\(\)\\
\underline{The underlying set \(X\)}: define \(X=\bigslant{\prod_{i\in\mathcal{I}} X_i}{\mathcal{U}}\), that is, the product \(\prod_{i\in\mathcal{I}} X_i\) divided by the equivalence relation given by: 
\[(x_i)\sim_{\mathcal{U}} (y_i) \iff \big\{i\in\mathcal{I} \big| x_i=y_i\big\}\in\mathcal{U}\]
We denote by \([x_i]\) the equivalence class of \((x_i)\). \\
\underline{The \(\sigma\)-algebra \(\Sigma\)}: for \((A_i)_{i\in\mathcal{I}}\) we define
\[\mathcal{U}\lim A_i=\big\{[x_i] \big| x_i\in A_i \quad \mathcal{U}\text{-a.e.} \big\}=\big\{[x_i] \big| \{i \big| x_i\in A_i\}\in\mathcal{U} \big\}  \]
An equivalent description is 
\[\boldsymbol{1}_{\mathcal{U}\lim A_i}([x_i])=\mathcal{U}\lim \boldsymbol{1}_{A_i}(x_i)\]
Here, in the right hand side we have used \(\mathcal{U}\lim\) for sequences of numbers (e.g. \cite[Definition 4.1]{sayag2022entropy}) .\\
The collection \(\mathcal{A}:=\big\{\mathcal{U}\lim A_i \big| A_i \in \Sigma_i\big\}\) forms an algebra on \(X\). We define \(\Sigma:=\sigma(\mathcal{A})\) to be the \(\sigma\)-algebra generated by \(\mathcal{A}\).\\
\underline{The measure \(\nu\)}: our next goal is to define the ultralimit measure.\\
Define the function \(\nu_*:\mathcal{A}\to [0,1]\) by 
\[\nu_*(\mathcal{U}\lim A_i)=\mathcal{U}\lim \nu_i(A_i)\]
Here also, in the right hand side we have used \(\mathcal{U}\lim\) for sequences of numbers.
As \(\:\:\mathcal{U}\lim A_i=\mathcal{U}\lim B_i \iff A_i=B_i\:\:\: \mathcal{U}\textrm{-a.e.}\:\:\) we conclude that \(\nu_{*}\) is well defined. It is easy to see that \(\nu_{*}\) gives rise to a finitely-additive measure of total mass 1 on \(\mathcal{A}\).\\
By Dynkin's Lemma (see \cite[Lemma 1.6]{williams}) an extension of \(\nu_*\) to a probability measure \(\nu:\Sigma\to[0,1]\) is unique if it exists. To show existence, we proceed as follows:\\
We will say that \(N\subset X\) is a \emph{null set} if for any \(\epsilon>0\) we can find \(A\in\mathcal{A}\) with \(N\subset A,\:\nu_{*}(A)<\epsilon\). Define \(\mathcal{N}\) to be the collection of null sets. Note that \(\mathcal{N}\) is closed under finite unions.
\begin{lemma}\label{Lemma defining ultralimit measure}
Let \(A^{(r)}\in\mathcal{A}\) and consider \(v=\lim_{m\to\infty}\nu_*(\bigcup_{r=1}^m A^{(r)})\). Then there is \(A\in\mathcal{A}\) with \(\nu_*(A)=v\) and \(A^{(r)}\subset A\) for all \(r\).
\end{lemma}
\begin{proof}
Write \(A^{(m)}=\mathcal{U}\lim A_i^{(m)}\) and let \(B_i^{(m)}:=\bigcup_{r=1}^{m} A_i^{(r)}\) so that \(B^{(m)}:=\mathcal{U}\lim B_i^{(m)}=\bigcup_{r=1}^{m} A^{(r)}\) and denote \(v_m:=\nu_*(B^{(m)})\) so that \(v_m\uparrow v\). Define 
\[T_m=\big\{i\in\mathcal{I} \big| \: |\nu_i(B_i^{(m)})-v_m|<\frac{1}{m} \big\} \]
by the definition of \(\nu_*\) we have that \(T_m\in \mathcal{U}\).
For any \(i\in \mathcal{I}\) define:
\[m(i)=\sup \big\{m \big| i\in \bigcap_{r=1}^m T_r \big\}\in \mathbb{N}\cup\{+\infty\} \quad,\quad A_i:=B_i^{(m(i))}=\bigcup_{m=1}^{m(i)} A_i^{(m)}\]
We will show that \(A=\mathcal{U}\lim A_i\in\mathcal{A}\) satisfies the desired properties.\\
For any positive integer \(M\) we have
\[ \big\{i\in\mathcal{I} \big|\: m(i)\geq M \big\} = \bigcap_{m=1}^{M} T_m \in\mathcal{U}  \]
Hence we get \(A_i^{(M)}\subset A_i\) for \(\mathcal{U}\)-a.e. \(i\). This implies \(A^{(M)}\subset A\) for any positive integer \(M\). We conclude from this that \(v\leq \nu_{*}(A)\).\\
Moreover, for each fixed \(M\) we have \(\nu(A_i)\leq v_{m(i)}+\frac{1}{m(i)}\leq v+\frac{1}{M}\) for \(\mathcal{U}\)-a.e. \(i\). Thus \(\nu_*(A)\leq v+\frac{1}{M}\) for any positive integer \(M\). We conclude \(\nu_*(A)=v\) and hence the proof of the lemma.
\end{proof}

\begin{corollary}\label{Corollary defining ultralimit measure}
\(\)
\begin{enumerate}
    \item
    \(\mathcal{N}\) is closed under countable union.
    \item 
    \(\overline{\Sigma}:=\big\{B\subset X \big| \exists A\in\mathcal{A}:\: A\triangle B\in\mathcal{N} \big\} \) is a  \(\sigma\)-algebra. 
    \item
    The mapping  \(\nu:\overline{\Sigma}\to[0,1]\) given by \(\nu(B)=\nu_*(A)\), where  \(A\in\mathcal{A}\) is such that \(A\triangle B\in\mathcal{N}\), is well defined. Furthermore, \(\nu\) defines a probability measure extending  \(\nu_*\).

\end{enumerate}
\end{corollary}
\begin{proof}
\begin{enumerate}
    \item
    Let  \(N^{(m)}\in\mathcal{N}\) and consider \(N:=\bigcup_m N_m\). We show \(N\in\mathcal{N}\). Let  \(\epsilon>0\), by assumption we have  \(A^{(m)}\in\mathcal{A}\) such that \(N^{(m)}\subset A^{(m)},\:\nu_*(A^{(m)})\leq \frac{\epsilon}{2^m}\). Applying Lemma \ref{Lemma defining ultralimit measure} yields \(A\in\mathcal{A}\) such that \(A^{(m)}\subset A\) and \(\nu_*(A)=\lim_m \nu_*(\bigcup_{r=1}^{m} A^{(r)})\leq\epsilon\). Thus  \(N\subset A\in\mathcal{A},\:\nu_*(A)\leq\epsilon\).
    \item
     It is obvious that \(\overline{\Sigma}\) is an algebra. In view of item 1, it is enough to show that for  \(A^{(n)}\in\mathcal{A}\) we have \(\bigcup A^{(n)}\in\overline{\Sigma}\).\\
     Using Lemma \ref{Lemma defining ultralimit measure}, we can find \(A\in\mathcal{A}\) such that: \(A\triangle(\bigcup A^{(n)})=A\setminus(\bigcup A^{(n)})\in\mathcal{N}\). Indeed, \(A\setminus(\bigcup A^{(n)})\subset A\setminus(\bigcup_{r=1}^m A^{(r)}))\) and \(\nu_*(A\setminus(\bigcup_{r=1}^m A^{(r)}))=\nu_{*}(A)-\nu_{*}(\bigcup_{r=1}^m A^{(r)}))\stackrel{m\to \infty}{\longrightarrow} 0\).
    \item
    Since \(\mathcal{A}\cap\mathcal{N}=\big\{A\in\mathcal{A} \big| \nu_*(A)=0 \big\}\) we conclude that  \(\nu\) is well defined. To show that it is  \(\sigma\)-additive, using item 1, and the proof of item 2 it is enough to show that if  \((A^{(r)})\) are  \(\nu_*\)-almost disjoint and \(A\) satisfies Lemma \ref{Lemma defining ultralimit measure} then \(\nu_*(A)=\sum_r \nu_*(A^{(r)})\). However this is obvious since \(\nu_*(A)=\lim_{m\to\infty} \nu_*(\bigcup_{r\leq m} A^{(r)})=\lim_{m\to\infty} \sum_{r\leq m} \nu_*(A^{(r)})= \sum_r \nu_*(A^{(r)})\).
\end{enumerate}
\end{proof}

\begin{definition}
Defining  \(\nu=\nu|_{\Sigma}\) and  \(\mathfrak{X}=(X,\Sigma,\nu)\) we have defined the  \(\mathcal{U}\)-\emph{\textbf{ultralimit of the probability spaces}}  \((\mathfrak{X_i})_{i\in\mathcal{I}}\).
\end{definition}

\begin{remark}
\begin{itemize}
    \item 
    The measure space completion of  \(\mathfrak{X}\) is  \((X,\overline{\Sigma},\nu)\)
    \item
    The collection of negligible sets (that is, sets contained in a measurable set of measure \(0\)) in \(\mathfrak{X}\) is the \(\sigma\)-ideal \(\mathcal{N}\).
    \item
    To show uniqueness of an extension for \(\nu_{*}\), one can avoid Dynkin's uniqueness theorem. Clearly item 2 of Corollary \ref{Corollary defining ultralimit measure} implies this uniqueness.
\end{itemize}
\end{remark}

\begin{definition}
Let \(f_i:X_i\to [-\infty,\infty]\) be 
functions on \(\mathfrak{X}_i\). Define
\[(\mathcal{U}\lim_i f_i)([x_i]):=\mathcal{U}\lim_i f_i(x_i)\]
\end{definition}
\begin{claim}
\(\)
\begin{itemize}
    \item 
    If \(f_{i}\) are Borel measurable functions on \(\mathfrak{X}_i\), then \(\mathcal{U}\lim f_i\) is Borel measurable function on \(\mathfrak{X}\).
    \item
    If for \(\mathcal{U}\)-a.e. \(i\) we have \(f_i=g_i\), then \(\mathcal{U}\lim f_i=\mathcal{U}\lim g_i\).
    \item
    If we  have \(f_i=g_i\) \(\nu_i\)-a.e., then \(\mathcal{U}\lim f_i=\mathcal{U}\lim g_i\) \(\nu\)-a.e.
\end{itemize}
\end{claim}
\begin{proof}
The first item follows from:
\begin{multline*}
 \big\{[x_i] \: \big| (\mathcal{U}\lim f_i)([x_i])>a \big\}=
\big\{[x_i]\: \big| \mathcal{U}\lim f_i(x_i) > a \big\}=\\
\bigcup_m \big\{[x_i]\: \big| \{i|f_i(x_i)>a+\frac{1}{m}\}\in\mathcal{U} \big\} =
\bigcup_m\:\mathcal{U}\lim \big\{x_i\: \big| f_i(x_i)>a+\frac{1}{m} \big\}    
\end{multline*}
The second and third items are trivial.
\end{proof}

Let us consider the functoriality of the construction:
\begin{definition}
A \emph{factor map} between probability spaces \(T:(X,\nu)\to(Y,m)\) is a measurable map with \(T_*\nu=m\). We will say that \(Y\) is a factor of \(X\) and that \(X\) is an extension of \(Y\).
\end{definition}

\begin{lemma}\label{Lemma functoriality ultralimit measure spaces}
Let \(\mathfrak{X}_i,\:\mathcal{Y}_i\) be probability spaces and let \(T_i:\mathfrak{X}_i\to\mathcal{Y}_i\) be factors. Then \(T:=\mathcal{U}\lim T_i\) defined by \(T([x_i])=[T_i(x_i)]\) is a factor \(T:\mathcal{U}\lim \mathfrak{X}_i\to\mathcal{U}\lim\mathcal{Y}_i\).
\end{lemma}
\begin{proof}
Note that 
\begin{align*}
    T^{-1}(\mathcal{U}\lim B_i)=\big\{[x_i]\big|T([x_i])\in\mathcal{U}\lim B_i \big\}=
    \big\{[x_i] \big| \{i:T_i(x_i)\in B_i \}\in\mathcal{U} \big\}= \\ 
    =\big\{[x_i]\big|\: \{i:x_i\in T_i^{-1}(B_i)\}\in\mathcal{U}\big\}=\mathcal{U}\lim T_i^{-1}(B_i)
\end{align*}
Thus \(T\) is measurable. To show \(T_*\nu=m\) it is enough to check it on a generating subalgebra:
\begin{align*}
    (T_*\nu)(\mathcal{U}\lim B_i)=\nu(T^{-1}(\mathcal{U}\lim B_i))=\nu(\mathcal{U}\lim T_i^{-1}(B_i))=\mathcal{U}\lim \nu_i(T_i^{-1}(B_i))= \\=\mathcal{U}\lim ((T_i)_*\nu_i)(B_i)=\mathcal{U}\lim m_i(B_i)=m(\mathcal{U}\lim B_i)
\end{align*}
\end{proof}

The last thing we will do in this subsection is to discuss the general relation between ultralimit and integration. We begin with an analog of Fatou's lemma:
\begin{lemma}\label{lemma ultra Fatou}
Suppose \(f_i\) are measurable non-negative functions on \(\mathfrak{X}_i\). Then we have the following inequality (of numbers in \([0,\infty]\)):
\[\intop_{X} \mathcal{U}\lim f_i\:\: d\nu \leq \: \mathcal{U}\lim \intop_{X_i} f_i\: d\nu_i\]
\end{lemma}
\begin{proof}
We may suppose that \(\mathcal{U}\lim \intop_{X_i} f_i \:d\nu_i<\infty\). Denote \(f=\mathcal{U}\lim f_i\).\\
It is enough to show that for any \(\mathfrak{X}\)-simple function \(0\leq\varphi\leq f\) we have \(\intop_X \varphi\:d\nu\leq \mathcal{U}\lim \intop_{X_i} f_i \:d\nu_i\). \\
Write \(\varphi=\sum_{m=1}^N c_m\cdot \boldsymbol{1}_{B^{(m)}}\) where  \(c_m\geq 0 \:,\: B^{(m)}\in \Sigma\) and let \(K:=\sum_{m=1}^N c_m\).\\
By Corollary \ref{Corollary defining ultralimit measure}(2), there are \(A^{(m)}=\mathcal{U}\lim A_i^{(m)}\) with  \(A_i^{(m)}\in \Sigma_i\) so that \(A^{(m)}\triangle B^{(m)}\) are null sets.\\
Define:
\[\psi:=\sum_{m=1}^N c_m\cdot \boldsymbol{1}_A^{(m)}=\mathcal{U}\lim \psi_i \:,\quad \psi_i:=\sum_{m=1}^N c_m\cdot\boldsymbol{1}_{A_i^{(m)}}\]
Note that \(\psi=\varphi\) \(\nu\)-a.e., thus we have \(\psi\leq f\) \(\nu\)-a.e. Hence for any \(m\geq 1\)
\[0=\nu(\big\{\mathcal{U}\lim f_i<\mathcal{U}\lim\psi_i \big\})\geq\nu\big(\mathcal{U}\lim \big\{f_i-\psi_i<-\frac{1}{m} \big\} \big)=\mathcal{U}\lim\nu_i(\{f_i-\psi_i<-\frac{1}{m}\})\]
This means that for all \(\delta,c>0\) we have for \(\mathcal{U}\)-a.e. \(i\): 
\(\nu_i\bigg(\big\{f_i-\psi_i<-c \big\}\bigg)<\delta\). \\
Let \(\epsilon>0\) and take \(\delta=\frac{\epsilon}{2K}, c=\frac{\epsilon}{2}\), then for \(\mathcal{U}\)-a.e. \(i\):
\begin{align*}
    \intop_{X_i} (f_i-\psi_i)_{-} \:d\nu_i\leq\intop_{\{f_i-\psi_i<-\frac{\epsilon}{2}\}} |\psi_i|\:d\nu_i+ \frac{\epsilon}{2} \leq
    +\nu_i\big(\{f_i-\psi_i<-\frac{\epsilon}{2}\}\big)||\psi_i||_\infty+\frac{\epsilon}{2}
    \leq \delta\cdot K +\frac{\epsilon}{2}=\epsilon
\end{align*}
Thus we conclude that for every \(\epsilon>0\) we have for \(\mathcal{U}\)-a.e. \(i\) that \[\intop_{X_i} f_i-\psi_i\: d\nu_i\geq -\epsilon\]
This yields that for any \(\epsilon>0\)
\[ \mathcal{U}\lim_{i} \intop_{X_i} f_i\:d\nu_i\:-\:\sum_{m=1}^N c_m \nu(B^{(m)})=\mathcal{U}\lim_{i} \Bigg(\intop_{X_i}  f_i\:d\nu_i-\sum_{m=1}^N c_m \cdot\mathcal{U}\lim_i\nu_i(A_i^{(m)})\Bigg)\geq -\epsilon \]
As this is true for any \(\epsilon>0\), we conclude
\[\mathcal{U}\lim \intop_{X_i} f_i\:d\nu_i\geq\intop_X \varphi\: d\nu\]
Which is what we wanted to show.
\end{proof}

As in classical analysis, ultralimits and integration do not necessarily commute.
That is, the equality \(\intop_X \mathcal{U}\lim f_i \:\: d\nu=\mathcal{U}\lim \intop_{X_i} f_i\:d\nu_i\) does \textbf{not} always holds. 
Another reflection of this complication is that there can be measures \(\eta_i,\nu_i\) on \(X_i\) such that \(\eta_i\ll\nu_i\) but \(\eta=\mathcal{U}\lim \eta_i\) is not absolutely continuous with respect to \(\nu=\mathcal{U}\lim \nu_i\). 
The next example illustrates this phenomena:
\begin{example}
\begin{enumerate}
    \item
    Take \(\mathcal{I}=\mathbb{N}\) and \(\mathcal{U}\) any non-principle ultra-filter, \(X_i=[0,1]\) with the Lebesgue measure. Let \(f_n=n\cdot\boldsymbol{1}_{[0,\frac{1}{n}]}\) then \(\intop_{X_n} f_n d\nu_n=1\) but \(\:\mathcal{U}\lim f_n=0\) \(\nu\)-a.e.
    \item
    In the setup of the previous example consider the measures \(\eta_n=\frac{1+f_n}{2}\cdot\nu_n\). Then \(\eta=\mathcal{U}\lim\eta_n\) is not absolutely continuous with respect to \(\nu=\mathcal{U}\lim \nu_n\). Indeed the set \(E=\mathcal{U}\lim [0,\frac{1}{n}]\) has measure \(\frac{1}{2}\) for \(\eta\) and measure \(0\) for \(\nu\).
\end{enumerate}
\end{example}
As suggested from ordinary real-analysis, the extra condition required to get equality results is uniform integrability. We study this notion in the next subsection.

\subsection{Ultralimit and integration for uniformly integrable functions}\label{subsection Ultralimit and integration for uniformly integrable functions}

We begin with defining the spaces of uniform integrability. This is a filtration of \(L^{1}\) by Banach spaces for any probability space. For their basic properties see Appendix \ref{appendix uniform integrability}. 

\begin{definition}[see Definition \ref{definition space of majorants}]\label{definition space of majorants main text}
The space of majorants \(\mathbf{M}\) is the set of functions \(\rho:[0,1]\rightarrow[0,1]\) such that:
\begin{enumerate}
    \item \(\rho(0^+)=\rho(0)=0,\: \rho(1)=1\).
    \item
    \(\rho\) is concave: for any \(x,y,t\in[0,1]\) we have \(\rho((1-t)x+ty)\geq (1-t)\rho(x)+t\rho(y)\).
\end{enumerate}
\end{definition}
Any \(\rho\in\mathbf{M}\) is continuous, non-decreasing, sub-additive and satisfies \(\rho(t)\geq t\). The space \(\mathbf{M}\) is closed under composition, \(\max\) and is convex (see Lemma \ref{lemma properties of majorants}).
\begin{definition}\label{definition C rho spaces main text}
Let \(\mathfrak{X}=(X,\Sigma,\nu)\) be a probability space and let \(\rho\in\mathbf{M}\).
\begin{enumerate}
    \item
We define the normed space \(\mathcal{C}_\rho (\mathfrak{X})\) to be the space of functions (mod a.e. 0 functions) in \(L^1(\mathfrak{X})\) such that
\[||f||_\rho=||f||_{\mathcal{C}_\rho (\mathfrak{X})}:=\sup_{A\in\Sigma:\ \nu(A)\neq0}\frac{1}{\rho(\nu(A))}\intop_{A}|f|d\nu<\infty\]
\item
Given a probability measures \(m\) on \((X,\Sigma)\),  we say that \(m\) is \emph{\(\rho\)-absolutely continuous with respect to \(\nu\)} if \(m\) is absolutely continuous with respect to \(\nu\) and \(||\frac{dm}{d\nu}||_\rho\leq1\). We denote this relation by \(m \stackrel{\rho}{\ll}\nu\). 
\end{enumerate}
\end{definition}
\(\mathcal{C}_\rho (\mathfrak{X})\) are Banach spaces that form an exhaustive filtration of \(L^{1}(\mathfrak{X})\) -- see Lemma \ref{lemma properties of C rho spaces}.\\
\(\)\\
We return to the setup of ultralimit of probability spaces, keeping the notations of subsection \ref{subsection Preliminaries- the construction of ultralimits of probability spaces}. We are now ready to formulate and prove our version of Lebesgue-Vitali theorem:
\begin{theorem} \label{Theorem ultralimit and integrals}
Let \(f_i\) are measurable functions on \(\mathfrak{X}_i\). Assume that there is \(\rho\in\mathbf{M}\) with \(f_i\in \mathcal{C}_\rho(\mathfrak{X}_i)\) and \(\sup_i ||f_i||_{\mathcal{C}_\rho(\mathfrak{X}_i)}<\infty\). Then \(\mathcal{U}\lim f_i\in L^1(\mathfrak{X})\) and 
\[\intop_X \mathcal{U}\lim f_i \:\: d\nu=\mathcal{U}\lim \intop_{X_i} f_i\:d\nu_i\]

\end{theorem}
\begin{proof}
We may assume by linearity that \(f_i\geq 0\). Denote \(f=\mathcal{U}\lim f_i\) and \(M=\sup_i ||f_i||_\rho\). As \(\intop_{X_i} f_i\:d\nu_i\leq M\) we get \(\mathcal{U}\lim \intop_{X_i} f_i \:d\nu_i\leq M\) and in particular it is finite. Lemma \ref{lemma ultra Fatou} (Fatou's lemma) implies that  \(f \in L^{1}(\mathfrak{X})\) and  
\[\intop_{X} f\:\:d\nu \leq\: \mathcal{U}\lim \intop_{X_i} f_i \: d\nu_i\]
It remains to prove the inequality in the second direction. Suppose to the contrary that there is \(\epsilon>0\) with:
\[\mathcal{U}\lim\intop_{X_i} f_i\:d\nu_i>3\epsilon+\intop_X f \:d\nu\]
Then there is \(I_0\in\mathcal U\) such that for all \(i\in I_0\) we have:
\[\intop_{X_i} f_i\:d\nu_i>3\epsilon+\intop_X f \:d\nu \]
Let \(\delta>0\) be such that \(\rho(2\delta)<\frac{\epsilon}{2M}\). Let \(\psi\) be an \(\mathfrak{X}\)-simple function with:
\begin{enumerate}
    \item
    \(\psi=\sum_{m=1}^N c_m\boldsymbol{1}_{A^{(m)}}\) where \(A^{(m)}=\mathcal{U}\lim A_i^{(m)},\:c_m\geq0,\: A_i^{(m)}\in \Sigma_i\).
    \item
    \(\nu\)-a.e. \(\psi\leq f\).
    \item
    \(\intop_X (f-\psi)\: d\nu\leq\frac{\epsilon\delta}{2}\).
\end{enumerate}
Define \(\psi_i=\sum_{m=1}^N c_m\boldsymbol{1}_{A^{(m)}_{i}}\) so that \(\psi=\mathcal{U}\lim \psi_{i}\).\\
By items 1,2 
\[\intop_{X} f\: d\nu \geq \intop_{X}\psi \:d\nu=\mathcal{U}\lim_{j} \intop_{X_{j}}\psi_{j} d\nu_{j}\]
Thus there is \(I_{1}\in \mathcal{U}\) so that for any \(j\in I_{1}\):
\[\intop_{X} f\:d\nu \geq \intop_{X_j} \psi_{j}\:d\nu_{j} - \epsilon\]
Hence, for \(i\in I_{0}\cap I_1\) we have:
\[\intop_{X_i} f_i\:d\nu_i>3\epsilon+\intop_{X}f\:d\nu \geq  2\epsilon+ \intop_{X_i}\psi_{i}\:d\nu_{i}\quad\quad(0)\]
On the other hand, by items 2, 3 and Markov's inequality:
\[\mathcal{U}\lim_{j}\: \nu_{j}(\{f_j-\psi_j>\epsilon\})=\nu\big(\mathcal{U}\lim \{f_j-\psi_j>\epsilon\}\big)\leq \nu\Big(\big\{\mathcal{U}\lim_j \big(f_j-\psi_{j}\big)>\frac{\epsilon}{2}\big\} \Big)=\nu\big(\{f-\psi>\frac{\epsilon}{2}\}\big)\leq\delta\]
This gives \(I_2\in\mathcal{U}\) such that for all \(j\in I_2\):
\[\nu_j\bigg(\Big\{ f_j-\psi_{j}>\epsilon \Big\}\bigg)\leq 2\delta \quad\quad (1)\]
But since \(||f_i||_\rho\leq M\) we conclude that for \(i\in I_0\cap I_1\cap I_2\neq\emptyset\):

\begin{align*}
2\epsilon\stackrel{(0)}{<}\intop_{X_i} f_i \: d\nu_i \: - \intop_{X_i}\psi_{i}\:d\nu_i\leq \epsilon+\intop_{\Big\{f_i-\psi_i>\epsilon \Big\}} f_i\: d\nu_i 
\stackrel{(1)}{\leq} \epsilon+||f_i||_\rho \cdot \rho(2\delta)<\epsilon+M\cdot\frac{\epsilon}{2M}=\frac{3\epsilon}{2}
\end{align*}
Which is a contradiction, proving \(\intop_X \mathcal{U}\lim f_i \:\: d\nu=\mathcal{U}\lim \intop_{X_i} f_i\:d\nu_i\).
\end{proof}

\begin{corollary}\label{Corollary 1 from Vitali theorem}
Let \(\rho\in\mathbf{M},\:f_i\in\mathcal{C}_\rho(\mathfrak{X}_i)\) and \(\sup_i ||f_i||_\rho<\infty\).\\
Then  \(\:\mathcal{U}\lim f_i\in\mathcal{C}_\rho(\mathfrak{X})\) and \(||\mathcal{U}\lim f_i||_\rho\leq\mathcal{U}\lim ||f_i||_\rho\).
\end{corollary}
\begin{proof}
By Theorem \ref{Theorem ultralimit and integrals}, \(f=\mathcal{U}\lim f_i\) lies in \(L^1(\mathfrak{X})\). Let \(A\subset X\) be a measurable set and take \(A_i\subset X_i\) measurable so that \(A\triangle \mathcal{U}\lim A_i\) is a null set (this is possible by Corollary \ref{Corollary defining ultralimit measure}(2)). Then \(\nu\)-a.e. \(|f|\cdot\boldsymbol{1}_A=\mathcal{U}\lim|f_i|\cdot\boldsymbol{1}_{A_i}\). Using Theorem \ref{Theorem ultralimit and integrals}:
\[\intop_A |f|\:d\nu=\mathcal{U}\lim\intop_{A_i} |f_i|d\nu_i\leq \mathcal{U}\lim ||f_i||_\rho\cdot\rho(\nu(A_i))=(\mathcal{U}\lim||f_i||_\rho)\cdot\rho(\nu(A))\]
\end{proof}

\begin{corollary}\label{Corollary 2 from Vitali theorem}
Let \(\rho\in\mathbf{M}\) and let \(\eta_i,\nu_i\) be probability measures on \((X_i,\Sigma_i)\). Consider \((X,\Sigma)\) with the ultralimit measures \(\eta,\nu\). If \(\eta_i\stackrel{\rho}{\ll}\nu_i\) then \(\eta\stackrel{\rho}{\ll}\nu\) and \(\frac{d\eta}{d\nu}=\mathcal{U}\lim\frac{d\eta_i}{d\nu_i}\).
\end{corollary}
\begin{proof}
Let \(f_i=\frac{d\eta_i}{d\nu_i},f=\mathcal{U}\lim f\). Then by Corollary \ref{Corollary 1 from Vitali theorem} \(f\in\mathcal{C}_\rho(X,\nu)\) with \(||f||_\rho\leq1\) and moreover by Theorem \ref{Theorem ultralimit and integrals} we have \(\intop_X fd\nu=1\). Define \(\eta_0=f\cdot\nu\) then \(\eta_0\) is a probability measure on \((X,\Sigma)\) with \(\eta_0\stackrel{\rho}{\ll}\nu,\:\frac{d\eta_0}{d\nu}=f\).\\
We will show that \(\eta=\eta_0\) and conclude the corollary. To show this we may show it only on the generating sub-algebra \(\mathcal{A}\). Let \(A\) be a set of the form \(\mathcal{U}\lim A_i\) then by Theorem \ref{Theorem ultralimit and integrals}:
\[\eta_0(A)=\intop_X f\cdot\boldsymbol{1}_A\:d\nu=\intop_X \mathcal{U}\lim f_i\cdot \boldsymbol{1}_{A_i}\:d\nu=\mathcal{U}\lim \intop_{X_i} f_i\cdot\boldsymbol{1}_{A_i}\:d\nu_i=\mathcal{U}\lim \eta_i(A_i)=\eta(A)\]
\end{proof}

\subsection{Comparison with ultralimit of C*-algebras}\label{subsection Comparison with ultralimit of C*-algebras}
We now give a comparison between this notion of ultralimit and the notion for \(C^*\)-probability spaces as in \cite{sayag2022entropy}. For the convenience of the reader, we remind the necessary notation:\\
Given a collection of Banach spaces \((B_{i})_{i\in\mathcal{I}}\), we define \(\ell^{\infty}(\mathcal{I},\{B_i\})\) to be the Banach space of bounded sequences in the product of \(B_i\). It has a closed subspace \(\mathcal{N}_{\mathcal{U}}\) consisting of sequences \((v_i)\) with \(\mathcal{U}\lim ||v_i||_{B_i}=0\) (\cite[Notation 4.4]{sayag2022entropy}). The quotient \(\mathcal{U}\lim B_i:=\ell^{\infty}(\mathcal{I},\{B_i\})/\mathcal{N}_{\mathcal{U}}\) is defined to be the ultralimit of those Banach spaces (\cite[Definition 4.6]{sayag2022entropy}).\\
A \(C^{*}\) probability space is a pair \((A,\nu)\) of a unital commutative \(C^{*}\)-algebra with a faithful state \(\nu\) (\cite[Definition 2.11]{sayag2022entropy}).\\
Given a collection of \(C^{*}\)-probability spaces \((A_i,\nu_i)_{i\in\mathcal{I}}\), on \(\mathcal{U}\lim A_i\) there is a canonical state \(\nu=\mathcal{U}\lim \nu_i\). Taking a quotient by the ideal of elements \(a\) with \(\nu(a^{*}a)=0\) we get a \(C^{*}\)-probability space denoted by \(\mathcal{U}\lim (A_i,\nu_i)\) (\cite[Definition 4.11]{sayag2022entropy}).

\begin{prop}\label{proposition compartion in L infinity of ultralimit}
Let \(\mathfrak{X}_i\) be probability spaces, then we have a canonical isomorphism of \(C^{*}\)-probability spaces:
\[\mathcal{U}\lim (L^\infty(\mathfrak{X}_i),\nu_i)\cong (L^\infty(\mathcal{U}\lim\mathfrak{X}_i),\nu)\]
\end{prop}
\begin{proof}
First define 
\(\Phi:\ell^\infty\big(\mathcal{I},\big\{L^\infty(\mathfrak{X}_i)\big\}\big)\to L^\infty\big(\mathcal{U}\lim\mathfrak{X}_i\big)\) by \(\Phi\big((f_i)\big)=\mathcal{U}\lim f_i\). This is a \(C^*\) homomorphism. It is clear that if \((f_i)\) are such that \(\mathcal{U}\lim ||f_i||=0\) then \(\mathcal{U}\lim f_i\) is a.e. \(0\). Hence \(\Phi\) factors to 
\(\mathcal{U}\lim L^\infty(\mathfrak{X}_i)\to L^\infty(\mathcal{U}\lim\mathfrak{X}_i)\). \\ 
By Theorem \ref{Theorem ultralimit and integrals} we conclude that \(\Phi^*(\nu)=\mathcal{U}\lim \nu_i\), since \(\nu\) is faithful on \(L^\infty(\mathcal{U}\lim\mathfrak{X}_i)\) we conclude that \(\Phi\) factors to an injection \(\mathcal{U}\lim (L^\infty(\mathfrak{X}_i),\nu_i)\to (L^\infty(\mathcal{U}\lim\mathfrak{X}_i),\nu)\).\\
\(\Phi\) is surjective since its image is closed and contains all simple functions of \(\mathcal{U}\lim\mathfrak{X}_i\) by Corollary \ref{Corollary defining ultralimit measure}(2). Thus \(\Phi\) is an isomorphism. 
\end{proof}

The next result shows a compatability of the Banach spaces \(\mathcal{C}_\rho\) and the operation of taking ultralimit. This is an important tool in the deduction of Theorem \ref{Theorem entropy for amenable intro} from Theorem \ref{Theorem ultralimit ameanble action intro} (e.g. see corollary \ref{Corollary lower semi cont of I in ultralimit}). 
The proof uses results in Appendix \ref{appendix uniform integrability}, in particular Proposition \ref{prop finding a rho integrable part} about extracting the \(\rho\)-integrable part out of any integrable function.

\begin{prop}\label{proposition compartion in C rho of ultralimit}
Let \(\mathfrak{X}_i\) be probability spaces and let \(\mathfrak{X}=\mathcal{U}\lim \mathfrak{X}_i\). Then for any \(\rho\in\mathbf{M}\), the mapping \((f_i)\mapsto\mathcal{U}\lim f_i\) induces a continuous linear operator 
\[\Phi_\rho:\mathcal{U}\lim \mathcal{C}_\rho(\mathfrak{X}_i)\to \mathcal{C}_\rho(\mathfrak{X})\]
that satisfies \(||\Phi_\rho||\leq 1\) and commutes with integration, e.g. \(\Phi_\rho^*(\nu)=\mathcal{U}\lim \nu_i\).\\
Moreover,
\[\ker\Phi_\rho=\big\{(f_i)+\mathcal{N}_\mathcal{U}\big|\:\mathcal{U}\lim\intop_{X_i} |f_i|d\nu_i=0\big\}\] 
and 
\[\overline{\Phi_\rho}:\frac{\mathcal{U}\lim \mathcal{C}_\rho(\mathfrak{X}_i)}{\ker\Phi_{\rho}}\stackrel{\cong}{\longrightarrow} \mathcal{C}_\rho(\mathfrak{X})\]
is an isometric isomorphism.
\end{prop}
\begin{proof}
By Corollary \ref{Corollary 1 from Vitali theorem}, \(\Phi_\rho\) is well defined and \(||\Phi_\rho||\leq 1\). By Theorem \ref{Theorem ultralimit and integrals}, \(\Phi_\rho\) commutes with integration. Note that:
\[\Phi_{\rho}(f_i+\mathcal{N}_\mathcal{U})=0\iff \: \intop_X |\mathcal{U}\lim f_i|d\nu=0 \iff \: \mathcal{U}\lim\intop_{X_i} |f_i|d\nu_i=0\]
which shows the description of the kernel.\\
Let us show that \(\overline{\Phi}_{\rho}\) is an isomorphism.
First, we show the following claim: 
\begin{itemize}
    \item Suppose \(\lambda\in\mathbf{M}\) and \(f\in \mathcal{C}_\rho(\mathfrak{X})\cap Im(\Phi_\lambda)\) then for any \(\delta>0\) there are \(f_i\in \mathcal{C}_\rho(\mathfrak{X}_i)\) with \(||f_i||_\rho\leq(1+\delta)||f||_\rho\) and \(f=\mathcal{U}\lim f_i\).
\end{itemize}
Indeed, by assumption we may write \(f=\mathcal{U}\lim g_i\) where \(g_i\in \mathcal{C}_{\lambda}(\mathfrak{X})\). Let \(M=||f||_\rho\), using Proposition \ref{prop finding a rho integrable part}, for any \(i\) there is \(B_i\in\Sigma_i\) with
\[\intop_{B_i} |g_i|d\nu_i\geq(1+\delta)M\rho\big(\nu_i(B_i)\big) \]
and such that \(f_i:=g_i\cdot \boldsymbol{1}_{B^c_i}\) satisfies that \(f_i\in \mathcal{C}_\rho(\mathfrak{X}_i)\) and \(||f_i||_\rho\leq (1+\delta)||f||_\rho\). To prove \(\mathcal{U}\lim f_i=f\) we will show \(\nu(\mathcal{U}\lim B_i)=0\). Indeed (by Theorem \ref{Theorem ultralimit and integrals} applied for \(\lambda\)):
\[M\rho\big(\nu(\mathcal{U}\lim B_i)\big)\geq\intop_{\mathcal{U}\lim B_i} |f|d\nu=\mathcal{U}\lim \intop_{B_i} |g_i|d\nu_i\geq\mathcal{U}\lim \: (1+\delta)M\rho\big(\nu_i(B_i)\big)=(1+\delta)M\rho\big(\nu(\mathcal{U}\lim B_i) \big)\]
Thus \(\nu(\mathcal{U}\lim B_i)=0\) as required.\\
\(\)\\
We return to the proof of the Proposition. The claim applied to the case \(\lambda=\rho\) yields that \(\overline{\Phi_\rho}\) is an isometric embedding. Thus we only need to show \(\overline{\Phi_\rho}\) is surjective. \\
Take \(\lambda=\rho^{\frac{1}{2}}\) then as we explained \(Im(\Phi_\lambda)\subset \mathcal{C}_\lambda(\mathfrak{X})\) is closed (as the image of an isometric embedding) and contains all simple functions. By Lemma \ref{lemma properties of C rho spaces}(4) we conclude it contains \(\mathcal{C}_\rho(\mathfrak{X})\). Thus by the claim we conclude \(\overline{\Phi}_\rho\) is surjective, finishing the proof.
\end{proof}

\begin{corollary}\label{Corollary of compartion of ultralimits}
Let \(\mathfrak{X}_i\) be probability spaces and let \(\mathfrak{X}=\mathcal{U}\lim\mathfrak{X}_i\).
\begin{enumerate}
    \item
    For any \(f\in L^1(\mathfrak{X})\) there is \(\rho\in\mathbf{M}\) such that for all \(\delta>0\) we can find \(f_i\in \mathcal{C}_\rho(\mathfrak{X}_i)\) with \(\sup||f_i||_\rho\leq(1+\delta)||f||_\rho\) and \(f=\mathcal{U}\lim f_i\).
    \item
    If \(\eta\) is a probability measure on \((X,\Sigma)\) absolutely continuous with respect to \(\nu\) then there is \(\rho\in\mathbf{M}\) and probability measures \(\eta_i\) on \((X_i,\Sigma_i)\) with \(\eta_i\stackrel{\rho}{\ll}\nu_i\) and \(\eta=\mathcal{U}\lim\eta_i\) 
\end{enumerate}
\end{corollary}
\begin{proof}
Item 1 is immediate from Proposition \ref{proposition compartion in C rho of ultralimit}. For item 2, consider \(f=\frac{d\eta}{d\nu}\), by item 1 we can find \(\rho_0\in\mathbf{M}\) and \(f_i\in \mathcal{C}_{\rho_0}(\mathfrak{X})\) with \(\sup ||f_i||_{\rho_0}\leq 2||f||_{\rho_0}=:M\). Define \(\rho(t)=\min(1,M\cdot\rho_0(t))\) then \(\rho\in\mathbf{M}\) and \(||f_i||_{\rho}\leq 1\). Thus \(\eta_i:=f_i\cdot \nu_i\) satisfies \(\eta=\mathcal{U}\lim \eta_i\:,\: \eta_i\stackrel{\rho}{\ll}\nu_i\).
\end{proof}

\subsection{Equivariant ultralimit}\label{section Equivariant ultralimit}
Let \(G\) be a discrete countable group.\\
In this section we define ultralimit of \(G\)-spaces. In the case where the spaces are quasi-invariant (QI) we provide a criterion (so called "uniform QI") for the ultralimit to be QI as well.\\
This is a generalization of the construction given in \cite[section 4.4]{sayag2022entropy} (see Proposition \ref{proposition comparing ultralimit BQI}).

\subsubsection{Quasi invariant G-spaces}
Let us remind the basic notations from \cite[Subsection 2.1]{sayag2022entropy}.
\begin{definition}
\(\)
\begin{itemize}
    \item
    A Borel \(G\)-space is a Borel space \(\mathfrak{X}=(X,\Sigma)\) together with a measurable action of \(G\) on \((X,\Sigma)\).\\
    We denote by \(M(X,\Sigma)\) the collection of probability measures on \(X\).
    \item 
    A \emph{quasi invariant \(G\)-space} (QI \(G\)-space) is a probability measure space \(\mathfrak{X}=(X,\Sigma,\nu)\) together with a measurable action of \(G\) on \(X\) such that for all \(g\in G\) the measures \(g\nu,\nu\) are in the same measure class (that is, has the same null sets).
    \item
    The Radon-Nikodym cocyle is defined by: \(R_\nu(g;x):=\frac{dg\nu}{d\nu}(x)\).
\end{itemize}
\end{definition}
    It is easy to see that \(R_{\nu}\) is indeed a cocyle:
\[R(gh;x)=R(h;g^{-1}x)\cdot R(g;x)\]
Recall that \(\mathbf{M}\) stands for the space of majorants (see Definition \ref{definition space of majorants main text}).
\begin{definition}
A \emph{majorant} for a QI \(G\)-space \(\mathfrak{X}\) is a function \(B:G\to\mathbf{M}\) such that for all \(g\in G\) we have \(g\cdot\nu\stackrel{B(g)}{\ll}\nu\) (see Definition \ref{definition C rho spaces main text}(2)).
\end{definition}
We remind the notion of BQI \(G\)-spaces from \cite[Definition 2.1]{sayag2022entropy}.
\begin{definition}
We say that \(\mathfrak{X}\) is \emph{bounded quasi invariant \(G\)-space} (BQI \(G\)-space) if it has a majorant \(B\) of the form \(B(g)(t)=\min(1,M(g)\cdot t)\) for a function \(M:G\to[1,\infty)\).
\end{definition}

\begin{lemma}\label{lemma majorant on convolution}
Let \(\mu\) be a probability measure on \(G\) and let \(B\) be a majorant for the QI \(G\)-space \(\mathfrak{X}=(X,\nu)\), then \(B(\mu):=\sum_{g\in G}\mu(g)B(g)\in\mathbf{M}\) and \(\mu\ast\nu\stackrel{B(\mu)}{\ll}\nu\).
\end{lemma}
\begin{proof}
\(B(\mu)\in\mathbf{M}\) by Lemma \ref{lemma properties of majorants}(2), and:
\[(\mu\ast\nu)(A)=\sum_{g}\mu(g)(g\nu)(A)\leq\sum_{g} \mu(g)B(g)\big(\nu(A)\big)=B(\mu)(\nu(A))\]
\end{proof}

\begin{definition}
\(\)
\begin{enumerate}
    \item
    A \emph{factor} \(p:\mathfrak{X}\to\mathcal{Y}\) between two QI \(G\)-spaces is a factor of probability spaces which is \(G\)-equivariant. We say \(\mathcal{Y}\) is a factor of \(\mathfrak{X}\) and \(\mathfrak{X}\) is an extension of \(\mathcal{Y}\).
    \item
    A factor \(p:\mathfrak{X}\to\mathcal{Y}\) is said to be \emph{measure preserving} if \(\frac{dg\nu_X}{d\nu_X}= \frac{dg\nu_Y}{d\nu_Y}\circ p\).
    \item
    The \emph{Radon Nikodym factor} of a QI \(G\)-space \(\mathfrak{X}\) is the final object in the category of measure preserving factors of \(\mathfrak{X}\). We denote it by \(\mathfrak{X}_{RN}\). More concretely, \(\mathfrak{X}_{RN}\) is equipped with a measure preserving factor \(\pi:\mathfrak{X}\to \mathfrak{X}_{RN}\), and has the property that for any other measure preserving factor \(p:\mathfrak{X}\to\mathcal{Y}\) there is a unique factor \(q:\mathcal{Y}\to \mathfrak{X}_{RN}\) so that \(\pi=q\circ p\).
\end{enumerate}
\end{definition}
\begin{remark}
The Radon-Nikodym factor of \(\mathfrak{X}\) indeed exists, it is the underling QI \(G\)-space for a topological model of \((X,\Sigma_{RN},\nu)\) where \(\Sigma_{RN}=\sigma(\frac{dg\nu}{d\nu}\big| g\in G)\) the \(\sigma\)-algebra generated by the Radon-Nikodym cocyle. The universal property guarantees that the pair \((\mathfrak{X}_{RN},\pi)\) is unique up to a unique isomorphism.
\end{remark}
The following lemma is useful for constructing majorants to actions:
\begin{lemma}\label{lemma majorant on convolution on G}
Suppose that \(\omega\) is a QI measure on \(G\) with majorant \(B:G\to \mathbf{M}\). Suppose that \((X,\Sigma,\nu)\) is measurable \(G\)-space such that there is a probability measure \(\nu_{0}\) on \((X,\Sigma)\) for which \(\nu=\omega\ast\nu_{0}\). Then \(\nu\) is QI with \(B\) as a majorant.
\end{lemma}
\begin{proof}
Consider the action map \(a: G\times X\to X\). Note that \(a_{*}(\omega\times\nu_{0})=\nu\:,\: a_{*}\big((g\omega)\times\nu_{0}\big)=g\nu\). By Lemma \ref{lemma functoriality of C rho spaces}(2) it is enough to show \((g\omega)\times \nu_{0}\stackrel{B(g)}{\ll}\omega\times\nu_{0}\). However, considering the projection \(\pi: G\times X\to G\) we have \(\frac{d\: \big((g\omega)\times \nu_{0}\big)}{d\:\big(\omega\times \nu_{0}\big)}=\pi^{*}(\frac{dg\omega}{d\omega})\). Thus we conclude the lemma by Lemma \ref{lemma functoriality of C rho spaces}(3).
\end{proof}

\subsubsection{Definition of ultralimit of G-spaces}
In this subsection, \(\mathcal{I}\) is a set and \(\mathcal{U}\) is an ultrafilter on \(\mathcal{I}\).\\
Lemma \ref{Lemma functoriality ultralimit measure spaces} yields immediately the following:
\begin{lemma}
Let \((X_i,\Sigma_{i})_{i\in\mathcal{I}}\) be a collection of Borel \(G\)-spaces and let \(\nu_i\in M(X_i,\Sigma_i)\). Consider the ultralimit \(\mathcal{U}\lim (X_i,\Sigma_i)=(X,\Sigma)\) with the ultralimit measure \(\nu=\mathcal{U}\lim \nu_{i}\). Then \((X,\Sigma)\) posses a natural structure of a Borel \(G\)-space. Moreover, for any \(g\in G\) we have that \(\mathcal{U}\lim g\nu_i=g\nu\).
\end{lemma}

\begin{lemma}\label{lemma convolution and ultralimit}
Suppose \((X_i,\Sigma_i)\) are Borel \(G\)-spaces, let \(\nu_i\in M(X_i,\Sigma_i)\) and consider \(\nu=\mathcal{U}\lim \nu_{i}\) the ultralimit measure on \((X,\Sigma)\). Then for any probability measure \(\mu\) on \(G\) we have:
\[\mathcal{U}\lim \mu\ast\nu_i=\mu\ast\nu\]
\end{lemma}
\begin{proof}
Let \(\epsilon>0\), we have a finite set \(S\subset G\) with \(\mu(G\setminus S)<\frac{\epsilon}{2}\). By linearity of \(\mathcal{U}\lim\) we conclude \(\sum_{g\in S} \mu(g)g\nu=\mathcal{U}\lim \sum_{g\in S}\mu(g)g\nu_i\). Thus (\(||\cdot||\) denotes the total variation of a measure):
\begin{multline*}
    ||(\mu\ast\nu)-\mathcal{U}\lim(\mu\ast\nu_i)||= ||(\mu\ast\nu)-\sum_{g\in S} \mu(g)\cdot(g\nu)+\mathcal{U}\lim\sum_{g\in S} \mu(g)(g\nu_i)-\mathcal{U}\lim (\mu\ast\nu_i)||\\
    \leq ||\sum_{g\in G\setminus S} \mu(g) (g\nu)||+||\mathcal{U}\lim \sum_{g\in G\setminus S} \mu(g)(g\nu_i)||\leq2\mu(G\setminus S)<\epsilon
\end{multline*}
Taking \(\epsilon\to0\) we conclude \(\mathcal{U}\lim (\mu\ast\nu_i)=\mu\ast\nu\).
\end{proof}
Next, we would like to know when the ultralimit of QI \(G\)-spaces is QI. For this we have the following definition and proposition:
\begin{definition}
We say that a collection of QI \(G\)-spaces is \emph{uniformly QI} if there is a function \(B:G\to \mathbf{M}\) that is a common majorant for all of them.
\end{definition}

\begin{prop}\label{proposition convolution and ultralimit}
Let \((\mathfrak{X}_i)_{i\in\mathcal{I}}\) be uniformly QI \(G\)-spaces,
then \(\mathcal{U}\lim\mathfrak{X}_i\) is a QI \(G\)-space.\\
Moreover, for any probability measure \(\mu\) on \(G\) we have
\[\mathcal{U}\lim \frac{d\mu\ast\nu_i}{d\nu_i}=\frac{d\mu\ast\nu}{d\nu}\]
\end{prop}
\begin{proof}
Let \(B\) be a uniform majorant for \((\mathfrak{X}_i)_i\). Denote \(\mathfrak{X}=\mathcal{U}\lim \mathfrak{X}_i\).\\
For any \(g\in G\) we have a measurable map \(g:X_i\to X_i\), and by functionality (Lemma \ref{Lemma functoriality ultralimit measure spaces}) the induced mapping \(g:\mathcal{U}\lim X_i\to\mathcal{U}\lim X_i\) given by \(g\cdot[x_i]=[g\cdot x_i]\) is measurable. Thus we get a measurable action of \(G\) on \(\mathcal{U}\lim X_i\).\\
We have by functionality \(g\nu=\mathcal{U}\lim g\nu_i\), which is, by Corollary \ref{Corollary 2 from Vitali theorem}, \(B(g)\)-absolutely continuous with respect to \(\nu\).\\
In conclusion, \(\mathfrak{X}\) is a QI-\(G\)-space with \(B\) as a majorant. \\
For the moreover part, by Lemma \ref{lemma majorant on convolution} we have \(\mu\ast\nu_i\stackrel{B(\mu)}{\ll}\nu_i\) and thus we conclude from Corollary \ref{Corollary 2 from Vitali theorem} and Lemma \ref{lemma convolution and ultralimit} that \(\mu\ast\nu=\mathcal{U}\lim (\mu\ast\nu_i)\ll\nu\) and 
\[\mathcal{U}\lim \frac{d\:\mu\ast\nu_i}{d\nu_i}=\frac{d\:\mathcal{U}\lim (\mu\ast\nu_i)}{d\nu}=\frac{d\:\mu\ast\nu}{d\nu}\]
as required.
\end{proof}

\begin{definition}
With the notations above, we will call \(\mathcal{U}\lim\mathfrak{X}_i\) with the \(G\)-action (which is structure a QI-\(G\)-space) the ultralimit of the uniformly QI-\(G\)-spaces \(\mathfrak{X}_i\).
\end{definition}



Recall that in \cite[Definition 4.17]{sayag2022entropy} we defined an ultralimit for a collection of uniformly BQI \(G\) spaces, which we denoted by \(\mathcal{U}\lim_{big} X_i\) and its Radon-Nikodym factor by \(\mathcal{U}\lim_{RN} X_i\). 
\begin{prop}\label{proposition comparing ultralimit BQI}
Suppose \(\mathfrak{X}_i\) is a collection of uniformly BQI \(G\)-spaces. Then:
\begin{itemize}
    \item \(\mathcal{U}\lim_{big} X_i\) is a compact Hausdorff space with an isomorphism of BQI-\(G\)-\(C^*\)-spaces:
    \[C(\mathcal{U}{\lim}_{big} \mathfrak{X}_i)\cong L^{\infty}(\mathcal{U}\lim \mathfrak{X}_i)\]
    \item
    \(\mathcal{U}\lim_{RN} \mathfrak{X}_i\) is the Radon-Nikodym factor of \(\mathcal{U}{\lim} \mathfrak{X}_i\).
\end{itemize}
\end{prop}
\begin{proof}
Follows immediately from \ref{proposition compartion in L infinity of ultralimit}.
\end{proof}

\section{Realizations via ultralimit}\label{section Realizations via ultralimit}
In this section we provide an ultralimit constructions for the Poisson boundary of a time dependent matrix-valued random walk on a discrete countable group \(G\). This is similar to \cite[Theorem E]{sayag2022entropy}. Then we will review the notion of amenable actions and prove Theorem \ref{Theorem ultralimit ameanble action intro}.

\subsection{Realization of the Poisson boundary of a time dependent matrix-valued random walk as an ultralimit}
Let us recall the notations of \cite{connes1989hyperfinite}, \cite{Amenable} about the Poisson boundary of a time dependent matrix-valued random walk -- the Poisson boundary of matrix-valued random walk on a group. For more details see Appendix \ref{appendix Poisson boundary of a time dependent matrix-valued random walk}. \\
Let \(\boldsymbol{\ell}=(\ell_n)_{n\geq0}\) be a sequence of positive integers, denote \(\ell_{-1}=1\). 
Let \(\boldsymbol{\sigma}=\big(\sigma^{(n)}\big)_{n\geq0}\) be a \emph{\(\boldsymbol{\ell}\)-stochastic sequence}, that is a sequence where \(\sigma^{(n)}=\big(\sigma_{i,j}^{(n)}\big)_{i\in[\ell_{n-1}],j\in[\ell_n]}\) is an \([\ell_{n-1}]\times [\ell_n]\) stochastic matrix of measures on \(G\) (see Definition \ref{definition stochastic sequence}). Here, \([\ell] = \{0,\dots,\ell - 1\}\).\\
We assume (as in \cite{Amenable}),
that for any \(j\in[\ell_{0}]\), the measure \(\sigma^{(0)}_{0,j}\) is supported on all of \(G\). We also suppose (without loss of generality, see Remark \ref{remark no zero columns on sigma}) \(\sigma^{(n)}\) has no zero columns.\\
We denote \(V_n=[\ell_n]\times G\) for \(n\geq0\). This is a \(G\)-space with trivial action on the first coordinate.
We consider the random walk \((X_n)_{n\geq0 }\) (where \(X_n\) takes values in \(V_n\)) with transition probabilities:
\[\mathbb{P}_{V_{n-1}\mapsto V_n}\bigg((i,g)\to(j,h)\bigg)=\sigma_{i,j}^{(n)}(g^{-1} h)\] 
We get a probability space \((\Omega,\mathcal{F},\mathbb{P})\) where \(\Omega=\prod_{n\geq0} V_n\) is the space of paths of the radon walk and the probability measure \(\mathbb{P}\) on \(\Omega\) is given by the Markov measure with transition probabilities as above and initial distribution \(\sigma^{(0)}\).
The \(G\)-action on \(V_n\) gives rise to a \(G\)-action on \(\Omega\) with. Let \(\mathcal{F}_m=\sigma(X_k \big| k\geq m)\) and define the asymptotic algebra of the random walk \(\mathcal{A}_\Omega=\bigcap_m \mathcal{F}_m\). \\
For each \(m\) we define 
\(\mathbb{P}^{(m)}\) to be the column vector of probability measures \((\mathbb{P}^{(m)}_{j,e})_{j\in [\ell_{m}]}\). Here, \(\mathbb{P}^{(m)}_{j,g}\) is the Markov measure of the random walk starting at \((j,g)\in V_m\).\\
The following definition is the basic generalization of the classical theory:
\begin{definition}
    A \(\boldsymbol{\sigma}\)-stationary space (or \(\boldsymbol{\sigma}\)-system), \(\mathcal{X}=(X,\boldsymbol{\nu})\) is a measure space \(X\) together with a measurable \(G\)-action, and a sequence \(\boldsymbol{\nu}=(\nu^{(n)})_{n=-1}^{\infty}\) of \(\ell_n\)-column vectors \(\nu^{(n)}=(\nu^{(n)}_j)_{j\in [\ell_n]}\) such that any \(\nu^{(n)}_j\) is a probability measure on \(X\), and such that \(\sigma^{(n)}\ast \nu^{(n)} = \nu^{(n-1)}\) for any \(n\geq 0\). That is, for any \(i\in[\ell_{n-1}]\) one has \(\sum_{j\in [\ell_n]} \sigma^{(n)}_{i,j} \ast \nu^{(n)}_{j} = \nu^{(n-1)}_{i}\).
\end{definition}
One can define factors and measure-preserving factors and the Radon-Nikodym factor for such systems (see Definitions \ref{definition stationary spaces}, \ref{definition measure preserving factors sigma-systems} and Lemma \ref{lemma RN factor of sigma-system}).\\
We say that a \(\boldsymbol{\sigma}\)-system \(\mathcal{X}=(X,\Sigma,\boldsymbol{\nu})\) is \emph{regular} if the underlying Borel space \((X,\Sigma)\) is a standard Borel space (see Definition \ref{definition regular sigma-system}).\\
One can see that \(\big(\Omega, \mathcal{A}_\Omega, (\mathbb{P}^{(m)})_m)\big)\) is a \(\boldsymbol{\sigma}\)-system. A regular \(\boldsymbol{\sigma}\)-system equivalent to it is called the \textbf{Poisson boundary of \(\boldsymbol{\sigma}\)} and denoted by \(\mathcal{B}(G,\boldsymbol{\sigma})\) (see Definition \ref{Definition Poisson boundary of sigma} and Lemma \ref{prop existence and uniqueness poisson boundary}). The Poisson boundary equals its own Radon-Nikodym factor (see Lemma \ref{prop RN factor of Poisson boundary is itself}).\\
Another basic property is that an extension of the Poisson boundary is necessarily measure preserving (see Proposition \ref{proposition maximality of Poisson boundary}).
\(\)\\
We now give our ultralimit construction for this Poisson boundary of a time dependent matrix-valued random walk. Fix an ultrafilter \(\mathcal{U}\) on \((\frac{1}{2},1)\) such that \(\mathcal{U}\lim_{a} a=1\), that is,  \(\forall \epsilon>0 :\: (1-\epsilon,1)\in\mathcal{U}\).
\begin{definition}
Consider the space \(E=\coprod_{n\geq0} \{n\}\times V_n=\coprod_{n,\:j\in[\ell_{n}]} \{n\}\times \{j\}\times G\) which is a \(G\)-space by the left multiplication on each copy of \(G\). \\
Let \(K\) be a non-negative integer. For each \(0<a<1\), \(t\geq -1\) and \(r\in [\ell_t]\), define a measure \(\nu_{r,a;K}^{(t)}\) on \(E\):\\
\[\nu_{r,a;K}^{(t)}(n,j,g)=\frac{1-a}{a^{t+1+K}}\cdot \boldsymbol{1}_{n\geq t+1+K}\cdot  a^{n}\cdot (\sigma^{(t+1)}\ast\dots\ast \sigma^{(n)})_{r,j}(g)\]


\end{definition}
Since \(\sigma^{(m)}\) are stochastic matrices of measures we get 
that \(\nu_{r,a;K}^{(t)}\) are probability measures on \(E\).\\
We have the following basic formula: for any \(t\geq0,s\in[\ell_{t-1}],K\geq0\) we have 
\[\sum_{r\in[\ell_{t}]}\sigma^{(t)}_{s,r}\ast\nu_{r,a;K}^{(t)}= \nu_{s,a;K+1}^{(t-1)}\quad\quad\quad (!)\]
Also note that
\[\nu_{r,a;K+1}^{(t)}\leq \frac{1}{a}\nu_{r,a;K}^{(t)}\]

\begin{lemma}\label{lemma measures on E independet of choices}
On \(\mathcal{U}\lim_{a} E\) the measure \(\mathcal{U}\lim \nu_{r,a;K}^{(t)}\) is independent of \(K\).
\end{lemma}
\begin{proof}
From, \(\nu_{r,a;K+1}^{(t)}\leq \frac{1}{a}\nu_{r,a;K}^{(t)}\) we conclude that \(\mathcal{U}\lim_{a}\nu_{r,a;K+1}^{(t)}\leq \mathcal{U}\lim_{a}\nu_{r,a;K}^{(t)}\). As both sides are probability measures, we conclude the equality.
\end{proof}

\begin{lemma}\label{lemma uniform integrability on E}
For any \(K\geq0\) the collection of measures \(\nu_{0,a;K}^{(-1)}\) for \(0<a<1\) is uniformly QI.\\
If \(\sigma^{(0)}_{i}\: (i\in [\ell_{0}])\) are BQI then \(\nu_{0,a;K}^{(-1)}\) for \(0<a<1\) is uniformly BQI.\\
Moreover, for any 
\(t\geq0,\:r\in[\ell_{t}]\), there is \(\rho\in\mathbf{M}\) so that for all \(\frac{1}{2}<a<1\) we have: \(\nu_{r,a;K}^{(t)}\stackrel{\rho}{\ll} \nu_{0,a;K}^{(-1)}\).
\end{lemma}
\begin{proof}
By construction there are column vectors of probability measures of length \(\ell_{0}\) on \(E\) denoted by \(\mu_{a;K}\) such that:
\[\nu_{0,a;K}^{(-1)}=\sigma^{(0)}\ast \mu_{a;K}\] 
Since \(\sigma^{(0)}\) is QI we conclude (Lemma \ref{lemma majorant on convolution on G}) that the collection \(\nu_{0,a;K}^{(-1)}\) is uniformly QI. If \(\sigma^{(0)}\) is BQI we conclude that the collection \(\nu_{0,a;K}^{(-1)}\) is uniformly BQI.\\
For the moreover part, we use \((!)\) to conclude that for any \(t\geq 0,r\in[\ell_t]\) and \(a\geq\frac{1}{2}\)
\[(\sigma^{(0)}\ast\dots\ast\sigma^{(t)})_{0,r}\ast\nu_{r,a;K}^{(t)} \leq \nu_{r, a; K+t+1}^{(-1)} \leq 2^{t+1} \nu_{0,a;K}^{(-1)}\]
Since there are no-zero columns on \(\boldsymbol{\sigma}\), there is \(g\in G\) such that \((\sigma^{(0)}\ast\dots\ast\sigma^{(t)})_{0,r}(g)>0\). Thus:
\[\nu_{r,a;K}^{(t)}\leq 2^{t+1}\cdot  g^{-1}\nu_{0,a;K}^{(-1)}\]
Using the uniform QI of \(\nu_{0,a;K}^{(-1)}\) we conclude the result.
\end{proof}

\begin{definition}
\(\mathfrak{B}_{\mathcal{U}}(\boldsymbol{\sigma})\) is the space \(\mathcal{U}\lim_{a} E\) equipped with the measures \(\nu^{(t)}_i:=\mathcal{U}\lim _a\nu^{(t)}_{i,a;0}\) and \(G\)-action.
\end{definition}

\begin{prop}\label{proposition the Abel construction is a sigma system}
\(\mathfrak{B}_{\mathcal{U}}(\boldsymbol{\sigma})\) is a \(\boldsymbol{\sigma}\)-system.
\end{prop}
\begin{proof}
Follows immediately from formula (!), Lemma \ref{lemma measures on E independet of choices} and Lemma \ref{lemma convolution and ultralimit}.
\end{proof}

\begin{definition}
We denote by \(\mathfrak{B}_{\mathcal{U}}(\boldsymbol{\sigma})_{RN}\) the Radon Nikodym factor of \(\mathfrak{B}_{\mathcal{U}}(\boldsymbol{\sigma})\) as a \(\boldsymbol{\sigma}\)-system (see Lemma \ref{lemma RN factor of sigma-system} for the Radon-Nikodym factor of a \(\boldsymbol{\sigma}\)-system).
\end{definition}

Our next goal is to mimic the proof of \cite[Theorem 5.9]{sayag2022entropy} for the Poisson boundary of a time dependent matrix-valued random walk. 

\begin{prop}\label{proposition maximality of the Abel construction}
Let \(\mathcal{X}\) be a regular \(\boldsymbol{\sigma}\)-system. Then we have a \(\boldsymbol{\sigma}\)-system \(\mathcal{Y}\) and a diagram:
\[\begin{CD}
\mathcal{Y} @>\mu>>  \mathcal{X} \\
@VVpV \\
\mathfrak{B}_{\mathcal{U}}(\boldsymbol{\sigma})_{RN}
\end{CD}
\]
Where \(\mu\) is a factor and \(p\) is measure preserving extensions.
\end{prop}
\begin{proof}
Let \(\mathcal{X}=(X,\boldsymbol{m})\). Consider \(E\times X\) where the \(G\)-action is only on the first coordinate. We have the following diagram:
\[\begin{CD}
E\times X @>\mu>>  X  \\
@VVpV \\
E
\end{CD}
\]
where \(p\) is the projection, and \(\mu((n,i,g),x)=g\cdot x\). Both \(\mu,p\) are \(G\)-equivariant. \\
For \(0<a<1\:,\:t\geq-1,r\in[\ell_t]\) we define the following probability measure \(m_{r,a}^{(t)}\) on \(E\times X\): \\
On \(\{n\}\times \{p\}\times G\times X\subset E\times X\) we put \(\frac{1-a}{a^{t+1}}\cdot \boldsymbol{1}_{n\geq t+1}\cdot a^{n}\cdot  \:(\sigma^{(t+1)}\ast\dots\ast\sigma^{(n)})_{r,p} \times m^{(n)}_{p}\).\\
Then we have:
\begin{multline*}
\mu_{*}(m_{r,a}^{(t)})=\sum_{n\geq 0;p\in[\ell_{n}]}\frac{1-a}{a^{t+1}}\cdot \boldsymbol{1}_{n\geq t+1}\cdot a^{n}\cdot \mu_{*}( \:(\sigma^{(t+1)}\ast\dots\ast\sigma^{(n)})_{r,p} \times m^{(n)}_{p})=\\
\frac{1-a}{a^{t+1}} \sum_{n=t+1}^{\infty} a^{n}\sum_{p\in [\ell_{n}]}(\sigma^{(t+1)}\ast\dots\ast\sigma^{(n)})_{r,p} \ast m^{(n)}_{p})=\frac{1-a}{a^{t+1}} \sum_{n=t+1}^{\infty} a^{n} m_{r}^{(t)}=m_{r}^{(t)}
\end{multline*}
On the other hand, we have that \(p_{*}(m_{r,a}^{(t)})=\nu_{r,a;0}^{(t)}\) and that for any \(g\in G\)
\[\frac{d\:gm_{r,a}^{(t)}}{d\:m_{0,a}^{(-1)}}=\frac{d\:g\nu_{r,a;0}^{(t)}}{d\:\nu_{0,a;0}^{(-1)}}\circ p\]
Consider the ultralimits diagram (Lemma \ref{Lemma functoriality ultralimit measure spaces}) where we get:
\[\begin{CD}
(\mathcal{U}\lim_{a} E\times X,\mathcal{U}\lim_{a} m_{r,a}^{(t)}) @>\mu>>(\mathcal{U}\lim_{a} X, \mathcal{U}\lim_{a} m_{r}^{(t)}) \\
@VVpV \\
(\mathcal{U}\lim_{a} E,\mathcal{U}\lim_{a} \nu_{r,a; 0}^{(t)})
\end{CD}
\]
Using  Lemma \ref{lemma uniform integrability on E} and Corollary \ref{Corollary 2 from Vitali theorem} we conclude
\[\frac{d\:g\mathcal{U}\lim m_{r,a}^{(t)}}{d\:\mathcal{U}\lim m_{0,a}^{(-1)}}=\frac{d\:g \nu_{r}^{(t)}}{d\:\nu_{0}^{(-1)}}\circ p\]
In particular we conclude that \(\mathcal{Y}=(\mathcal{U}\lim E\times X, \big((\mathcal{U}\lim m_{r,a}^{(t)})_{r=1,\dots,\ell_n}\big)_{t\geq0}\) give rise to a \(\boldsymbol{\sigma}\)-system and that \(p\) is a measure preserving extension. Thus we got:
\[\begin{CD}
\mathcal{Y} @>\mu>>(\mathcal{U}\lim X, \mathcal{U}\lim \boldsymbol{m}) @>\Delta>> (X,\boldsymbol{m}) \\
@VVpV \\
\mathfrak{B}_{\mathcal{U}}(\boldsymbol{\sigma})\\
@VVqV\\
\mathfrak{B}_{\mathcal{U}}(\boldsymbol{\sigma})_{RN}
\end{CD}
\]
where \(\Delta\) is defined using the fact that \(X\) is a standard probability space and we have \(\Delta: \: L^\infty(X)\to L^\infty(\mathcal{U}\lim X)\) which sends each \(\mathcal{U}_{a}\lim m_{r}^{(t)}\) to \(m_{r}^{(t)}\) [and applying \cite[Theorem 2.1]{RAMSAY1971253} ].\\
Note that \(p,q\) (and \(\Delta\)) are measure preserving extensions, proving the result.
\end{proof}

Finally we obtain the key tool of this paper - the construction of the Poisson boundary of a time dependent matrix-valued random walk in terms of ultralimits:

\begin{theorem}\label{Theorem ultralimit construction to Poisson boundary of a time dependent matrix-valued random walk}
The Poisson boundary \(\mathcal{B}(G,\boldsymbol{\sigma})\) of \(\boldsymbol{\sigma}\) is isomorphic to \(\mathfrak{B}_{\mathcal{U}}(\boldsymbol{\sigma})_{RN}\).
\end{theorem}
\begin{proof}
Take \(\mathcal{X}=\mathcal{B}(G,\boldsymbol{\sigma})\) to be the Poisson boundary and apply Proposition \ref{proposition maximality of the Abel construction}. By Proposition \ref{proposition maximality of Poisson boundary} the factor \(\mathcal{Y}\to\mathcal{B}(G,\boldsymbol{\sigma})\) is measure preserving. Since the Poisson boundary is its own Radon Nikodym factor (Lemma \ref{prop RN factor of Poisson boundary is itself}) and \(\mathcal{Y}\to \mathfrak{B}_{\mathcal{U}}(\boldsymbol{\sigma})_{RN}\) is measure preserving we conclude:
\[\mathcal{B}(G,\boldsymbol{\sigma})\cong\mathcal{B}(G,\boldsymbol{\sigma})_{RN}\cong \mathcal{Y}_{RN}\cong \mathfrak{B}_{\mathcal{U}}(\boldsymbol{\sigma})_{RN}\]
\end{proof}

\begin{remark}
We now sketch another proof of Theorem \cite[Theorem A]{sayag2022entropy} without appealing to Furstenberg's conjecture (\cite[Theorem 4.3]{BoundaryEntropy})). \\
Suppose \(G\) is amenable, let \(\ell_n=1\). Consider \(\sigma^{(0)}\) which is BQI. 
We define inductively an exhausting F{\o}lner sequence \(F_n\). Indeed, take \(F_{n}\) to be a F{\o}lner set for \((\epsilon,S)=(\frac{1}{2^n},F_{1}\cdots F_{n-1})\) that contains \(e\) (here, \(F_1\cdots F_{n-1}\) is the set of products).\\
Let \(\sigma^{(n)}=\frac{1}{|F_n|}\delta_{F_n}\) for \(n\geq0\).\\
Consider the Poisson boundary \(\mathcal{B}(G,\boldsymbol{\sigma})=(B,(m^{(n)}))\). Note that \(\sigma^{(n+1)}\ast m^{(n+1)}=m^{(n)}\) which implies that for \(g\in F_{1}\cdots F_{n}\) one has \(||g\cdot m^{(n)}-m^{(n)}||\leq ||g\sigma^{(n+1)}-\sigma^{(n+1)}||\leq\frac{1}{2^{n+1}}\). Thus \(||m^{(0)}-m^{(n)}||=||\sigma^{(1)}\ast\dots \ast \sigma^{(n)}\ast m^{(n)}-m^{(n)}||\leq \frac{1}{2^{n+1}}\). Thus for any \(g\in G\), there is \(n_0\) so that for \(n\geq n_0\) one has: \(||m^{(0)}-gm^{(0)}||<\frac{3}{2^{n+1}}\). In particular \(m^{(0)}\) is \(G\)-invariant and thus \(m^{(-1)}=\sigma^{(0)}\ast m^{(0)}\) is \(G\)-invariant.\\
In particular considering \(\pi:G\times \mathbb{N}\to G\) and \(\lambda_{k}=\pi_{*}\nu_{0,1-\frac{1}{k};0}^{(-1)}\) we conclude that \(\lambda_{k}\) are uniformly BQI measures on \(G\) such that for any non-principle ultrafilter \(\mathcal{U}\) we have \(\mathcal{U}\lim (G,\lambda_{k})\) is \(G\)-invariant.\\
In particular, the measures \(\lambda_{k}\) are KL-almost invariant (they are also \(D_{f}\)-almost invariant for any \(f\)-divergence).
\end{remark}

\subsection{Amenable actions and ultralimits}
We begin by giving the definition for amenability, we use \cite[Section 3]{MR2013348} as our main resource. We treat only the case of discrete groups, which simplifies the presentation. However we do not assume that our probability spaces are standard.
\begin{definition}\label{Definition amenable pair}
Given a factor of QI \(G\)-spaces \(\pi:Y\to X\), a \(G\)-equivariant conditional expectation is a linear mapping \(E:L^{\infty}(Y)\to L^{\infty}(X)\) such that:
\begin{enumerate}
    \item
    \(E\) is \(G\)-equivariant: for every \(f\in L^{\infty}(Y),g\in G\) we have \(E(f\circ g)=E(f)\circ g\).
    \item 
    \(E\) is non-negative: if \(L^{\infty}(Y)\ni f\geq 0\) then \(E(f)\geq 0\).
    \item
    \(E\) is normalized: \(E(1)=1\).
    \item
    \(E\) is \(L^{\infty}(X)\)-linear: for \(f\in L^{\infty}(Y),h\in L^{\infty}(X)\) we have \(E(\pi^{*}(h)\cdot f)=h\cdot E(f)\).
\end{enumerate} 
If such a \(G\)-equivariant expectation exists, we say that the pair \((Y,X)\) is amenable.\\
We say that a QI \(G\)-space \(S\) is amenable (in Zimmer's sense) if the pair \((G\times S, S)\) is amenable. Here, \(G\times S\) has the product measure class and diagonal action, and the factor is the projection.
\end{definition}
\begin{remark}
Here, the actual measure on our QI \(G\)-spaces is immaterial, only its measure class will play a role.
\end{remark}
\begin{example}\label{example conditional expectation of mpe is G equivariant}
A factor of QI \(G\)-spaces \(\pi:(Y,m)\to (X,\nu)\) is a measure preserving extension iff the conditional expectation \(\mathbb{E}_{m,\pi}: L^{\infty}(Y)\to L^{\infty}(X)\) is \(G\)-equivariant.\\
In particular, suppose \(S\) is a QI \(G\)-space, then for any measure \(\nu\) in the measure class, the pair \((S,T)\) is amenable, where \(T=(S,\nu)_{RN}\) is the RN factor of \((S,\nu)\).
\end{example}
In the case where the \(G\)-spaces are standard, this notion of amenable action is the one from \cite{Amenable},\cite{AmenableGeneral}. 
This concept is equivalent to Zimmer's: he introduced his notion of amenability in the papers \cite{zimmer1978amenable},\cite{zimmer1977neumann}, \cite{MR470692} based on the idea of a fixed point property, and proved the equivalence to the definition above.\\
We remind that we consider non-standard probability spaces, so we need to check that some basic properties still hold in this setting.
\begin{lemma}\label{lemma amenable pair and amenable actions}
Let \(\pi:S\to T\) be a factor of QI \(G\)-spaces.
\begin{enumerate}
    \item If \((S,T)\) is an amenable pair and \(S\) is an amenable \(G\)-space, then \(T\) is an amenable \(G\)-space.
    \item
    If \((S,T)\) is an amenable pair and suppose \(X\) is a QI \(G\)-space such that there are factors \(S\stackrel{\tau}{\to} X\stackrel{p}{\to} T\) with \(\pi=p\circ\tau\). Then \((X,T)\) is an amenable pair.
\end{enumerate}
\end{lemma}
\begin{proof}
\begin{enumerate}
    \item 
    Let \(E:L^{\infty}(S)\to L^{\infty}(T)\) be a \(G\)-equivariant conditional expectation. As \(S\) is amenable, we have a \(G\)-equivariant conditional expectation \(E_{S}:L^{\infty}(G\times S)\to L^{\infty}(S)\). Consider
    \[E_{T}:=E\circ E_{S}\circ(id_{G}\times\pi)^{*}: L^{\infty}(G\times T)\to L^{\infty}(G\times S)\to L^{\infty}(S)\to L^{\infty}(T)\]
    it is easy to see that \(E_{T}\) is a \(G\)-equivariant conditional expectation \(L^{\infty}(G\times T)\to L^{\infty}(T)\) and thus \(T\) is amenable.
    \item
    Let \(E:L^{\infty}(S)\to L^{\infty}(T)\) be a \(G\)-equivariant conditional expectation, then:
    \[E\circ \tau^{*}:L^{\infty}(X)\to L^{\infty}(S)\to L^{\infty}(T)\]
    is a \(G\)-equivariant conditional expectation.
\end{enumerate}
\end{proof}


The next lemma allows us to reduce to the case of standard probability spaces in problems of amenability.
\begin{lemma}\label{lemma action ameanble iff rn amenable}
Let \((S,\nu)\) be a QI \(G\)-space and consider the Radon-Nikodym factor \(T=(S,\nu)_{RN}\). Then \(S\) is amenable iff \(T\) is amenable.
\end{lemma}
\begin{proof}
If \(S\) is amenable, by Lemma \ref{lemma amenable pair and amenable actions} and Example \ref{example conditional expectation of mpe is G equivariant} we conclude that \(T\) is amenable.\\
For the other direction, we need to define \(\Phi: L^{\infty}(G\times S)\to L^{\infty}(S)\). Let \((S_{\alpha},\pi_{\alpha})_{\alpha\in I}\) be the set of measure preserving factors of \((S,\nu)\) that are standard probability spaces (up to isomorphism). \(I\) has a partial order \(\leq\) defined by \(\alpha\leq \beta \iff \pi_{\alpha}^{*}L^{\infty}(S_{\alpha})\subset \pi_{\beta}^{*} L^{\infty}(S_{\beta})\). It is clear that \((I,\leq)\) is directed.\\
Applying \cite[Corollary C]{AmenableGeneral} for \(S_{\alpha}\to T=(S,\nu)_{RN}\) and the amenability of \(T\) we conclude that \(S_{\alpha}\) is amenable for every \(\alpha\in I\).
Thus, there is a \(G\)-equivariant conditional expectation \(\Phi_{\alpha}: L^{\infty}(G\times S_{\alpha})\to L^{\infty}(S_\alpha)\). Using \(\pi_{\alpha}\) we consider \(L^{\infty}(S_\alpha),L^{\infty}(G\times S_{\alpha})\) naturally as sub-algebras \(L^{\infty}(S),L^{\infty}(G\times S)\) respectively. Extending
\(\Phi_{\alpha}\) by zero we obtain a function \(\Phi_{\alpha}: L^{\infty}(G\times S)\to L^{\infty}(S)\). Note that \(\Phi_{\alpha}\) defines an element of
\[E:=\prod_{f\in L^{\infty}(G\times S)} \overline{B}_{L^{\infty}(S)}(||f||) \]
We give \(E\) the product topology of the \(w^{*}\)-topology. By Tychonoff's theorem we conclude that \(E\) is compact. So that the net \((\Phi_{\alpha})_{\alpha\in I}\) has a sub-net converging to \(\Phi\in E\).\\ 
Note that for any \(f\in L^{\infty}(G\times S)\) there is \(\alpha_{0}\in I\) so that for every \(\alpha\geq \alpha_0\) we have \(f\in L^{\infty}(G\times S_{\alpha})\). Indeed, consider the \(\sigma\)-algebra generated by the Radon-Nikodym cocyle of \(\nu\) (that is the functions \(\frac{dg\nu}{d\nu}(s)\)) and \(f(g,h\cdot s)\) for \(g,h\in G\). This \(\sigma\)-algebra is separable and \(G\)-invariant thus corresponds to \(S_{\alpha_0}\) for some \(\alpha_{0}\in I\), which satisfies the required property. \\
This easily implies that \(\Phi: L^{\infty}(G\times S)\to L^{\infty}(S)\) is a \(G\)-equivariant conditional expectation.
\end{proof}

Let us prove Theorem \ref{Theorem ultralimit ameanble action intro} (we are using Proposition \ref{proposition comparing ultralimit BQI} to translate it to the measurable ultralimit construction):
\begin{theorem}\label{Theorem Amenable actions and ultralimit}
Let \(S\) be an ergodic amenable \(G\)-space. Then there is a BQI measure \(\nu\) on \(S\) (in the measure class) and a sequence of uniformly BQI measures \(\lambda_n\) on \(G\times\mathbb{N}\) such that for any non-principle ultrafilter \(\mathcal{U}\) on \(\mathbb{N}\) we have that there is an isomorphism 
\[(S,\nu)_{RN}\cong (\mathcal{U}\lim G\times\mathbb{N},\mathcal{U}\lim\lambda_n)_{RN} \]
\end{theorem}
\begin{proof}
We first reduce to the case where \(S\) is standard. Let \(T=(S,m)_{RN}\) for some measure \(m\) on \(S\) (in the measure class), then \(T\) is standard. By Lemma \ref{lemma action ameanble iff rn amenable}, \(T\) is amenable and since it is a factor of an ergodic action it is also ergodic. Moreover, the result for \(T\) implies it for \(S\). Indeed, for any BQI probability measure \(\nu_0\) on \(T\) there is a BQI probability measure \(\nu\) on \(S\) for which \((S,\nu)\to (T,\nu_0)\) is measure preserving (e.g. \(\nu=(\frac{d\nu_{0}}{d\pi_{*}m}\circ \pi)\cdot m\) where \(\pi: S\to T\) is the factor map).\\
Thus we may assume that \(S\) is a standard probability space.\\ 
By \cite[Theorem A]{AmenableGeneral}, there is a probability measure \(\nu\) on \(S\), a sequence \(\boldsymbol{\ell}\) and an \(\boldsymbol{\ell}\)-stochastic system of measures \(\boldsymbol{\sigma}\) on \(G\), so that \((S,\nu)\) is isomorphic to the underlying QI \(G\)-space of the Poisson boundary of \(\boldsymbol{\sigma}\) -- that is, to \((B,m^{(-1)})\) where we denote \(\mathcal{B}(G,\boldsymbol{\sigma})=(B,\boldsymbol{m})\).\\
By replacing \(\boldsymbol{\sigma}\) we may assume that there are no zero columns (see Remark \ref{remark no zero columns on sigma}). By taking any BQI measure \(\omega\) on \(G\) and replacing \(\sigma^{(0)}\) with \(\omega\ast \sigma^{(0)}\) we may assume that \(\sigma^{(0)}\) is BQI. This changes the Poisson boundary only by changing \(m^{(-1)}\) to \(\omega\ast m^{(-1)}\), so by changing \(\nu\) to \(\omega\ast\nu\) we keep the isomorphism between \((S,\nu)\) and the underlying QI \(G\)-space of the Poisson boundary of \(\boldsymbol{\sigma}\).\\
Take \(a_n=1-\frac{1}{n}\) a sequence that goes to 1, an ultrafilter on \(\mathbb{N}\) can be identified with an ultrafilter on \((\frac{1}{2},1)\), and being non-principle means that \(\mathcal{U}\lim_n a_n=1\). Note that \(\lambda_n=\nu^{(-1)}_{0,a_n;0}\) are uniformly BQI (Lemma \ref{lemma uniform integrability on E}).
By Theorem \ref{Theorem ultralimit construction to Poisson boundary of a time dependent matrix-valued random walk} we conclude:
\[(S,\nu)_{RN}\cong  (B,m^{(-1)})_{RN}\cong (\mathfrak{B}_{\mathcal{U}}(\boldsymbol{\sigma})_{RN},\nu_{0}^{(-1)})_{RN}\cong (\mathfrak{B}_{\mathcal{U}}(\boldsymbol{\sigma}),\nu_{0}^{(-1)})_{RN}\cong (\mathcal{U}\lim_{n} \:G\times \mathbb{N}, \mathcal{U}\lim \lambda_n)_{RN}\]
\end{proof}
\begin{remark}
In the displayed sequence of isomorphisms in the end of the proof above, the subscript RN has two meanings. One when applied to a QI \(G\)-space, and the other when applied to a \(\boldsymbol{\sigma}\)-system.\\
The latter was used only once during the proof in the notation \(\mathfrak{B}_{\mathcal{U}}(\boldsymbol{\sigma})_{RN}\).\\
The mapping \((\mathfrak{B}_{\mathcal{U}}(\boldsymbol{\sigma}),\nu_{0}^{(-1)}) \to (\mathfrak{B}_{\mathcal{U}}(\boldsymbol{\sigma})_{RN},\nu_{0}^{(-1)})\) is measure preserving (for QI \(G\)-spaces), and thus induces an isomorphism for RN factors.
\end{remark}

\begin{remark}
We actually used \cite[Theorem A]{AmenableGeneral} only for discrete groups, which is the main result of the earlier paper \cite{Amenable}.
\end{remark}

\begin{remark}
For the applications of the next section, we could use the following weaker version of the above results:\\
For any amenable ergodic and standard \(S\), there is a BQI measure \(\nu\) on \(S\), a BQI \(G\)-space \((X,m)\) and a collection of uniformly BQI measures \(\lambda_{n}\) on \(G\)
such that:
\begin{enumerate}
    \item
    \((X,m)\) is a measure preserving extension of \((S,\nu)\).
    \item
    \((X,m)\) is an extension of \((\mathcal{U}\lim G,\lambda_{n})_{RN}\).
\end{enumerate}
To prove this weaker version, one need not use Theorem \ref{Theorem ultralimit construction to Poisson boundary of a time dependent matrix-valued random walk}, rather it is enough to use
Proposition \ref{proposition the Abel construction is a sigma system} combined with a structure theorem for stationary actions. Indeed, one first
uses \cite[Theorem A]{AmenableGeneral} to exhibit \(S\) as a Poisson boundary, then apply Proposition \ref{proposition the Abel construction is a sigma system} and Theorem \ref{Theorem Furstenberg Glasner} (Taking \(\Lambda=\mathcal{B}(G,\boldsymbol{\sigma})\)) to construct \(X\).
\end{remark}

\begin{remark}
\cite[Lemma 5.6]{sayag2022entropy} implies that for any \(G\)-space \(Y\), any any measure \(\lambda\) on \(G\), there is a measure \(\nu\) on \(Y\) so that \((Y,\nu)\) is a factor of a space which is measure preserving extension of \((G,\lambda)\) (we can take \(\nu=\lambda\ast m\) for any \(m\)). In this sense, actions of \(G\) on \((G,\lambda)\) are the "biggest".\\
In this approach, we see that amenable action as actions are "big", in the sense that they can be obtained as limits of free actions.\\
This perspective is motivated by Theorem \ref{Theorem entropy for amenable intro}.
\end{remark}

\section{Applications to Entropy}
In this section we prove the main theorems of the paper. In subsection \ref{subsection Entropy in amenable actions} we prove Theorem 
\ref{Theorem entropy for amenable intro}, in sub-subsection \ref{subsection Entropy for hyperbolic groups} we prove Theorems \ref{Theorem Entropy for hyperbolic groups intro},\ref{Theorem entropy for lattices intro} and in sub-subsection \ref{subsection Entropy for free groups} we prove Theorem \ref{Theorem entropy for free groups intro}.

\subsection{Recollection of entropy}\label{subsection Recollection of Entropy}
We recall the basic notions regarding entropy from \cite[Section 3]{sayag2022entropy}.\\
Let us recall the definition of \(f\)-divergence (see \cite[Chapter 6]{polyanskiy2014lecture}):
\begin{definition}
Let \(f\) be a convex function on \((0,\infty)\) with \(f(1)=0\).\\
We denote \(f^{\prime}(\infty)=\lim_{t\to\infty}\frac{f(t)}{t}\) and \(f(0)=f(0^{+})\) (these numbers exist in \(\mathbb{R}\cup\{+\infty\}\)).\\
Given a Borel space \((X,\Sigma)\) we denote by \(M(X,\Sigma)\) the collection of probability measures on \((X,\Sigma)\). The \(f\)-divergence between two measures \(P,Q\in M(X,\Sigma)\) is defined in the following way:\\
Take a measure \(R\) on \(X\) with \(P,Q\ll R\), write \(\frac{dP}{dR}=p, \frac{dQ}{dR}=q\) and define:
\[D_{f}(P||Q)=\intop_{\{q>0\}} f\Big(\frac{p}{q}\Big)q\: dR + P(\{q=0\})\cdot f^{\prime}(\infty)\]
\end{definition}
This is independent of the choice of \(R\) (we have the agreement here that \(0\cdot\infty=0\)).\\
Note that one can change \(f(t)\) by \(a \cdot (t-1)\) for any \(a\in\mathbb{R}\) and it does not change \(D_{f}\). Thus we may assume \(f\geq0\) and in particular the integral above gives rise to a well defined value in \([0,\infty]\).\\
We summarize the basic properties of \(f\)-divergence in the following:
\begin{lemma}\label{lemma properties of f-divergence}
Let \(f\) be a convex function on \((0,\infty)\) with \(f(1)=0\).
\begin{enumerate}
    \item
    For any Borel space \((X,\Sigma)\), the function \(D_{f}:M(X,\Sigma)\times M(X,\Sigma)\to [0,\infty]\) is convex.
    \item
    If \(P,Q\in M(X,\Sigma)\) are of the same measure class then:
    \[D_{f}(P||Q)=\intop_{X} f\Big(\frac{dP}{dQ}\Big)dQ\]
    \item
    For a measurable mapping between Borel spaces \(\pi: (X,\Sigma_{X})\to (Y,\Sigma_{Y})\) for any \(P,Q\in M(X,\Sigma_{X})\) we have:
    \[D_{f}(\pi_{*}P||\pi_{*}Q)\leq D_{f}(P||Q)\]
    \item
    Suppose \(P,Q\in M(X,\Sigma_{X})\) and \(\nu\in M(Y,\Sigma_{Y})\) and consider the measures \(P\times \nu, Q\times \nu\) on the product \((X\times Y, \Sigma_{X}\otimes \Sigma_{Y})\). Then we have:
    \[D_{f}(P\times \nu||Q\times \nu)=D_{f}(P||Q)\]
\end{enumerate}
\end{lemma}
\begin{remark}
If \(f^{\prime}(\infty)=+\infty\) and \(D_{f}(P||Q)<\infty\) then \(P\ll Q\).
\end{remark}
Recall the definition of entropy \cite[Definition 3.4]{sayag2022entropy}:
\begin{definition}
Let \(f\) be a convex function on \((0,\infty)\) with \(f(1)=0\) and let \(\lambda\) be a probability measure on \(G\).
Given a probability measure \(\nu\) on a Borel \(G\)-space \((X,\Sigma)\), its \emph{Furstenberg's \((\lambda,f)\)-entropy} is defined by:
\[h_{\lambda,f}(X,\nu)=\sum_{g\in G}\lambda(g) D_{f}(g\nu||\nu)\]
\end{definition}
By Lemma \ref{lemma properties of f-divergence} we conclude that the entropy \(h_{\lambda,f}\) is a convex function that decreases along factors:
\begin{prop}\label{Proposition properties of entropy lambda f}
Let \(f\) be a convex function on \((0,\infty)\) with \(f(1)=0\) and let \(\lambda\) be a probability measure on \(G\).
\begin{enumerate}
    \item
    For any Borel \(G\)-space \((X,\Sigma)\) the function \(h_{\lambda,f}:M(X,\Sigma)\to [0,\infty]\) is convex.
    \item
    Given a measurable \(G\)-equivariant map \(\pi:(X,\Sigma_{X})\to (Y,\Sigma_{Y})\) between Borel \(G\)-spaces, for any \(\nu\in M(X,\Sigma_{X})\) we have:
    \[h_{\lambda,f}(Y,\pi_{*}\nu)\leq h_{\lambda,f}(X,\nu)\]
\end{enumerate}
\end{prop}
\begin{remark}
In \cite{sayag2022entropy} we only considered the case of \(\nu\) being quasi-invariant, however, using the definition of \(f\)-divergence for every pair of measures, there is no reason to do so.\\
However, when \(f^{\prime}(\infty)=+\infty\) and \(Supp(\lambda)\) is generating as a semi-group then \(h_{\lambda,f}(\nu)<\infty\) implies that \(\nu\) is quasi-invariant.
\end{remark}
We come to defining the main object studied in this section -- various versions of the minimal entropy number.
\begin{definition}
Let \(f\) be a convex function on \((0,\infty)\) with \(f(1)=0\) and let \(\lambda\) be a probability measure on \(G\).
\begin{enumerate}
    \item 
    Given a QI \(G\)-space \(X\), we denote by \(M(X)\) the probability measures on \(X\) in the measure class, and define the minimal entropy number of \(X\) to be:
    \[I_{\lambda,f}(X)=\inf_{\nu\in M(X)} h_{\lambda,f}(X,\nu)\]
    \item
    Given a measurable action \(G\curvearrowright (X,\Sigma)\) on a Borel space, we define the minimal entropy number of \((X,\Sigma)\) to be:
    \[I_{\lambda,f}^{Borel}(X,\Sigma)=\inf_{\nu\in M(X,\Sigma)} h_{\lambda,f}(X,\nu)\]
    We recall that \(M(X,\Sigma)\) denotes the collection of all probability measures on \((X,\Sigma)\).
    \item
    Given a continuous action \(G\curvearrowright X\) where \(X\) is a compact metrizable space, we define the topological minimal entropy number of \(X\) to be: \[I_{\lambda,f}^{top}(X):=I_{\lambda,f}^{Borel}(X,\mathcal{B})\]
    Here, \(\mathcal{B}\) is the Borel \(\sigma\)-algebra on \(X\).
\end{enumerate}
\end{definition}
\begin{remark}
The concepts \(I_{\lambda,f}(X), I_{\lambda,f}^{top}(X) \) were introduced in \cite{sayag2022entropy}. 
\end{remark}

\begin{lemma}\label{lemma properties of I lambda f}
\(\)
\begin{enumerate}
\item
If \(\pi:S\to T\) is a factor map between QI \(G\)-spaces then \(I_{\lambda,f}(S)\geq I_{\lambda,f}(T)\).\\
Moreover, if there is a measure \(m\in M(S)\) so that \(\pi:(S,m)\to (T,\pi_*m)\) is measure preserving extension then \(I_{\lambda,f}(S)= I_{\lambda,f}(T)\).
\item
For any Borel \(G\)-space \((X,\Sigma)\) and \(\omega\in M(G),\nu\in M(X,\Sigma)\) we have \(h_{\lambda,f}(\omega\ast\nu)\leq h_{\lambda,f}(\omega)\).
\item
For any QI \(G\)-space \(S\) we have \(I_{\lambda,f}(S)\leq I_{\lambda,f}(G)\).
\item
Suppose \(\lambda\) is a finitely supported probability measure on \(G\). Let \((S,\nu)\) be a BQI \(G\)-space. Then: 
\[I_{\lambda,f}(S)=\inf\{h_{\lambda,f}(S,m)\big|\: m=\omega\ast m_0 \:\text{ where }\: m_0\in M(S)\: \text{ satisfies }\: \frac{d\nu}{dm_0}\in L^{\infty}(S)\:\: \text{ and }\: \omega\in M(G)\: \text{ is BQI}\:\}\]
\end{enumerate}
\end{lemma}
\begin{proof}
\begin{enumerate}
\item
The inequality \(I_{\lambda,f}(S)\geq I_{\lambda,f}(T)\) follows from \cite[Lemma 3.6]{sayag2022entropy}. If \(\pi: (S,m)\to(T,\pi_*m)\) is measure preserving extension then for any \(\nu\in M(T)\) we can take \(\kappa=(\frac{d\nu}{d\pi_*m}\circ \pi)\cdot \nu\in M(S)\) and then \(\pi:(S,\kappa)\to (T,\nu)\) is measure preserving. By \cite[Lemma 3.6]{sayag2022entropy} we conclude \(I_{\lambda,f}(S)\leq h_{\lambda,f}(\kappa)=h_{\lambda,f}(\nu)\), which shows the other direction.
\item
The same proof as \cite[Lemma 3.8]{sayag2022entropy}: Consider the Borel \(G\)-space \(G\times X\) (with the product \(\sigma\)-algebra) where the \(G\)-action is only on the first coordinate. The action map \(a:G\times X\to X\) is \(G\)-equivariant and maps \(\omega\times \nu\) to \(\omega\ast\nu\) we conclude by Lemma \ref{lemma properties of f-divergence}(4) and \ref{Proposition properties of entropy lambda f}(2) :
\[h_{\lambda,f}(G,\omega)=h_{\lambda,f}(G\times X,\omega\times \nu)\geq h_{\lambda,f}(X,\omega\ast\nu)\]
\item
Follows from the previous item (this is also \cite[Corollary 3.9]{sayag2022entropy}).
\item 
Denote by \(I^\prime\) the right hand side. Let \(\kappa\in M(S)\), we need to show \(h_{\lambda,f}(\kappa)\geq I^{\prime}\).\\
Choose a BQI measure \(\omega\) on \(G\) and let \(\epsilon>0\).\\
Define \(m_0=(1-\epsilon)\kappa+\epsilon \nu\) and \(\omega_\epsilon=(1-\epsilon)\sum_{n\geq0}\epsilon^n\omega^{n}\) and \(m:=\omega_\epsilon\ast m_0\). Note that \(\omega_{\epsilon}\) is a BQI measure and since \(\frac{d\nu}{dm_0}\leq\epsilon^{-1}\) we conclude that \(h_{\lambda,f}(m)\geq I^{\prime}\). \\
Note that \(\omega_\epsilon=(1-\epsilon)\delta_e+\epsilon \omega\ast\omega_\epsilon\), using convexity of \(h_{\lambda,f}\) and item 2 we get:
\begin{multline*}
    I^\prime\leq h_{\lambda,f}(m)= h_{\lambda,f}\big((1-\epsilon)m_0+ \epsilon \omega\ast m \big)\leq (1-\epsilon)h_{\lambda,f}(m_0)+ \epsilon h_{\lambda,f} (\omega\ast m)\leq\\
    (1-\epsilon)^2 h_{\lambda,f}(\kappa)+ (1-\epsilon)\epsilon h_{\lambda,f}(\nu) + \epsilon h_{\lambda,f}(\omega)
\end{multline*}
As \(\omega\) and \(\nu\) are BQI we conclude \(h_{\lambda,f}(\omega),h_{\lambda,f}(\nu)<\infty\) (indeed \(\lambda\) is finitely supported).\\
Taking \(\epsilon\to0\) we conclude that \(h_{\lambda,f}(\kappa)\geq I^\prime\).
\end{enumerate}
\end{proof}

The next lemma shows that our concept of entropy is consistent with ergodic decomposition.\\
We recall, that given a QI \(G\)-space \(X\) which is a standard probability space, one considers \(L^{\infty}(X)^{G}\subset L^{\infty}(X)\) which is \(w^{*}\)-closed and thus of the form \(L^{\infty}(E)\) for a standard probability space \(E\). The inclusion above yields a measurable mapping \(\pi:X\to E\) which is \(G\)-invariant (there is no \(G\)-action on \(E\)). Given a measure \(\nu\) on \(X\) we denote \(P=\pi_{*}\nu\) and \(\{\nu_{e}\}_{e\in E}\) be the disintegration, \(\nu=\intop_{E}\nu_{e}dP\).\\
For almost every \(e\in E\) one has that \((X,\nu_{e})\) is QI and ergodic.\\
We call this the ergodic decomposition of \((X,\nu)\).
\begin{lemma}\label{lemma ergodic decomposition and entropy}
Suppose \((X,\nu)\) is a QI \(G\)-space which is a standard probability space. Let \(\pi: (X,\nu)\to (E,P)\) and let \(\nu=\intop_{E} \nu_{e} 
dP(e)\) be the ergodic decomposition. Then:
\begin{enumerate}
    \item
    For almost every \(e\in E\) we have  \(\frac{dg\nu}{d\nu}=\frac{dg\nu_{e}}{d\nu_{e}}\) [\(\nu_{e}\)-a.e.].
    \item
    For any probability measure \(\lambda\) on \(G\) and any convex function \(f\) with \(f(1)=0\) we have: 
    \[h_{\lambda,f}(\nu)=\intop_{E} h_{\lambda,f}(\nu_{e})dP(e)\]
\end{enumerate}
\end{lemma}
\begin{proof}
\begin{enumerate}
    \item 
    Since there is a countable algebra  of measurable sets of \(X\) that generates its \(\sigma\)-algebra, we need only to show that for each \(A\subset X\) we have for a.e. \(e\in E\) that  \(\intop_{A} \frac{dg\nu}{d\nu} d\nu_{e}=g\nu_{e}(A)\). This is equivalent to \(\varphi(e):=\intop_{X} (\frac{dg\nu}{d\nu}\cdot 1_{A}- 1_{g^{-1}A})\:d\nu_{e}=0\) for a.e. \(e\in E\). To show this we will show that for any \(T\subset E\) we have \(\intop_{T}\varphi \:dP=0\). Indeed, 
    \[\intop_{T}\Big(\intop_{X} (\frac{dg\nu}{d\nu}\cdot 1_{A}- 1_{g^{-1}A})\:d\nu_{e}\Big) dP=\intop_{X} \big(\frac{dg\nu}{d\nu}\cdot 1_{A}- 1_{g^{-1}A}\big)1_{\pi^{-1}T}\:d\nu=g\nu(A\cap \pi^{-1}(T))-\nu(\pi^{-1}(T)\cap g^{-1}A)=0\]
    where the last equality follows from \(\pi^{-1}(T)=g^{-1}\pi^{-1}(T)\). Thus we conclude the item.
    \item
    Assume without loss of generality that \(f\geq 0\).
    Using the previous item, for a.e. \(e\in E\) we have:
    \(D_{f}(g\nu_{e}||\nu_{e})=\intop_{X} f(\frac{dg\nu}{d\nu}) d\nu_{e}\). Integrating over \(E\) and applying Fubini-Tonelli we get: \[D_{f}(g\nu||\nu)=\intop_{E} D_{f}(g\nu_{e}||\nu_{e}) dP(e)\]
    applying the monotone convergence theorem we get the result.
\end{enumerate}
\end{proof}


\begin{lemma}\label{lemma only QI measures on topological actions}
Let \(f\) be a convex function with \(f(1)=0\) and let \(\lambda\) be a finitely supported measure on \(G\). Then for any \(G\curvearrowright (X,\Sigma)\) measurable action we have:
\[I_{\lambda,f}^{Borel}(X,\Sigma)=\inf \big\{h_{\lambda,f}(X,\nu)\::\:\:\nu\in M(X,\Sigma)\:\: \text{quasi-invariant}\big\}\]
\end{lemma}
\begin{proof}
Denote by \(I^{\prime}\) the right hand side. We need to show that for any \(\nu\in M(X,\Sigma)\) we have \(h_{\lambda,f}(\nu)\geq I^{\prime}\). 
Take a BQI measure \(\omega\) on \(G\), then \(h_{\lambda,f}(\omega)<\infty\). 
Define \(\omega_{\epsilon}=(1-\epsilon)\sum \epsilon^{n}\omega^{*n}\), then \(\omega_{\epsilon}=\epsilon\cdot \omega\ast\omega_{\epsilon}+(1-\epsilon)\delta_{e}\). By Lemma \ref{lemma properties of I lambda f}(2) and Proposition \ref{Proposition properties of entropy lambda f}(1), for any \(\nu\in M(X,\Sigma)\) we have:
\[I^{
\prime}\leq h_{\lambda,f}(\omega_{\epsilon}\ast\nu) \leq \epsilon h_{\lambda,f}(\omega\ast\omega_{\epsilon}\ast \nu)+(1-\epsilon)h_{\lambda,f}(\nu)\leq \epsilon h_{\lambda,f}(\omega)+(1-\epsilon)h_{\lambda,f}(\nu)\]
Since \(h_{\lambda,f}(\omega)<\infty\) taking \(\epsilon\to0\) we see \(h_{\lambda,f}(\nu)\geq I^{\prime}\).
\end{proof}

\subsubsection{Entropy minimal number and ultralimit}
The following follows either from Theorem \ref{Theorem ultralimit and integrals} (in the uniformly bounded case) or from \cite[Corollary 4.19]{sayag2022entropy} (combined with Proposition \ref{proposition comparing ultralimit BQI})
\begin{prop}\label{Proposition entropy and ultralimit BQI case} 
Suppose \((X_i,\nu_i)_{i\in\mathcal{I}}\) are uniformly BQI \(G\)-spaces. Let \(\mathcal{U}\) be an ultrafilter on \(\mathcal{I}\) and let \((X,\nu)=\mathcal{U}\lim (X,\nu_i)\).
Then for any convex function \(f\) with \(f(1)=0\) and finitely supported probability measure \(\lambda\) on \(G\) we have:
\[h_{\lambda,f}(X,\nu)=\mathcal{U}\lim h_{\lambda,f}(X_i,\nu_i)\]
\end{prop}



\begin{lemma}\label{lemma different measures on Ultralimit}
Let \((X_i,\nu_i)_{i\in\mathcal{I}}\) be uniformly QI \(G\)-spaces, let \(\mathcal{U}\) be an ultrafilter on \(\mathcal{I}\) and consider \((X,\nu)=\mathcal{U}\lim (X_i,\nu_i)\).
Suppose \(\omega\) be a BQI measure on \(G\) and \(m_0\in M(X)\) is a probability measure on \(X\) of the same measure class as \(\nu\). Assume that \(\frac{d\nu}{dm_0}\in L^{\infty}\) and let \(m=\omega\ast m_0\).\\
Then there are uniformly BQI measures \(\eta_i\) on \(X_i\) of the same measure class as \(\nu_i\) with \(\mathcal{U}\lim \eta_i=m\).
\end{lemma}
\begin{proof}
Consider \(f=\frac{d m_0}{d\nu}\in L^{1}(X,\nu)\), by Corollary \ref{Corollary of compartion of ultralimits}, there is a majorant \(\rho\in\mathbf{M}\) and \(f_i\in \mathcal{C}_{\rho}(X_i,\nu_i)\) with \(\sup ||f_i||_\rho<\infty\) and \(f=\mathcal{U}\lim f_i\). \\
By assumption there is \(a>0\) for which \(f\geq a\). We may assume that \(f_i\geq \frac{a}{2}\) and \(\intop_{X} f_i d\nu_i=1\) for all \(i\). Indeed, otherwise consider \(g_i=\max(f_i,a)\) then \(\mathcal{U}\lim g_i=\max(\mathcal{U} \lim f_i,a)=f\), moreover, \(\mathcal{U}\lim \intop_{X_i} g_i d\nu_i=\intop_{X} \mathcal{U}\lim g_i d\nu=1\) and thus \(\mathcal{U}\)-a.e. \(\frac{1}{2}\leq \intop_{X_i} g_i d\nu_i\leq 2\). Define \(h_{i}=\frac{g_i}{||g_i||_{L^{1}}}\), then for \(\mathcal{U}\)-a.e. \(i\) we have \(h_i\geq \frac{a}{2}\), \(h_i\) are \(\rho\)-uniformly bounded and \(\mathcal{U}\lim h_i=f\).\\ 
Hence we can make this assumption on \(f_i\). 
Thus, \(\kappa_i=f_i\cdot \nu_i\) are probability measures on \(X_i\) of the same measure class as \(\nu_{i}\) and \(\mathcal{U}\lim \kappa_i=m_0\). This implies \(\eta_i:=\omega\ast \kappa_i\) are uniformly BQI (since \(\omega\) is BQI) and by Lemma \ref{lemma convolution and ultralimit} we have
\(\mathcal{U}\lim\eta_i=\omega\ast m_0=m\).
\end{proof}
The following corollary plays a key role in our proof of Theorem \ref{Theorem entropy for amenable intro}:
\begin{corollary}\label{Corollary lower semi cont of I in ultralimit}
Let \((X_i,\nu_i)_{i\in\mathcal{I}}\) be uniformly QI \(G\)-spaces, let \(\mathcal{U}\) be an ultrafilter on \(\mathcal{I}\) and consider \((X,\nu)=\mathcal{U}\lim (X_i,\nu_i)\).
Then for any convex function \(f\) with \(f(1)=0\) and finitely supported probability measure \(\lambda\) on \(G\) we have:
\[\mathcal{U}\lim I_{\lambda,f}(X_i) \leq I_{\lambda,f}(X)\]
\end{corollary}
\begin{proof}
By replacing \(\nu_{i}\) by \(\omega\ast\nu_i\) for a BQI measure \(\omega\) on \(G\) we may assume that \((X,\nu)\) is BQI.
Let \(I:=\mathcal{U}\lim I_{\lambda,f}(X_i)\), by Lemma \ref{lemma properties of I lambda f} it is enough that show that if \(m\in M(X)\) is of the form \(m=\omega\ast m_0\) where \(\omega\) is BQI and \(\nu\stackrel{L^{\infty}}{\ll}m_0\) then \(h_{\lambda,f}(m)\geq I\). However, using Lemma \ref{lemma different measures on Ultralimit}, there are uniformly BQI measures \(\eta_i\) on \(X_i\) of the same measure class as \(\nu_{i}\) with \(\mathcal{U}\lim \eta_i=m\). By Proposition \ref{Proposition entropy and ultralimit BQI case} we conclude:
\[h_{\lambda,f}(m)=\mathcal{U}\lim h_{\lambda,f}(\eta_i)\geq \mathcal{U}\lim I_{\lambda,f}(X_i)=I\]
\end{proof}

The following general proposition will not be used in the next subsections, however we include it here:
\begin{prop}\label{Proposition lower semi continouity of entropy}
Suppose \((X_i,\nu_i)_{i\in\mathcal{I}}\) is a collection of QI \(G\)-spaces. Let \(\lambda\) be a generating (as a semi-group) measure on \(G\) and \(f\) a convex function with \(f(1)=0\) satisfying
\(\lim_{t\to \infty} \frac{f(t)}{t}=\infty\).\\
If \(\sup_{i} h_{\lambda,f}(X_i,\nu_i)<\infty\), then \((X_i,\nu_i)\) are uniformly QI. Moreover, for any ultrafilter \(\mathcal{U}\) on \(\mathcal{I}\) the ultralimit \((X,\nu)=\mathcal{U}\lim (X_i,\nu_i)\) satisfies:
\[h_{\lambda,f}(X,\nu)\leq \mathcal{U}\lim h_{\lambda,f}(X_i,\nu_i)\]
\end{prop}
\begin{proof}
We can change \(f\) by a linear function to assume further \(f\geq 0\). Let \(M=\sup_{i} h_{\lambda,f}(X_i,\nu_i)\).\\
For any \(g\in G\) we have \(\mathbb{E}_{\nu_{i}}(f(\frac{dg\nu_i}{d\nu_i}))\leq \frac{M}{\lambda(g)}\) and thus by Lemma \ref{lemma de la Valle´e-Poussin} we conclude that for any \(g\in Supp(\lambda)\) there is \(\rho\in\mathbf{M}\) and \(K>0\) for which \(||\frac{dg\nu_i}{d\nu_i}||_{\mathcal{C}_\rho(X_i,\nu_i)}\leq K\) (Both \(K,\rho\) depend on \(g\in Supp(\lambda)\)). Changing \(\rho\) to \(\min (1,K\cdot\rho)\) we get that \(g\nu_i\ll^{\rho} \nu_i\). Since \(Supp(\lambda)\) is generating \(G\) as a semi-group and using the fact that \(\mathbf{M}\) is closed under composition (see Lemma \ref{lemma properties of majorants}) we conclude that \((X_i,\nu_i)\) are uniformly QI.\\
The inequality \(h_{\lambda,f}(X,\nu)\leq \mathcal{U}\lim h_{\lambda,f}(X_i,\nu_i)\)  follows immediately from Proposition   \ref{proposition convolution and ultralimit} (compatibility of Radon-Nikodym derivative with ultralimit in the uniformly QI case), the continuity of \(f\) and
Lemma \ref{lemma ultra Fatou} (our version of Fatou's lemma).
\end{proof}

\subsection{Entropy in amenable actions - proof of Theorem \ref{Theorem entropy for amenable intro}}\label{subsection Entropy in amenable actions}

We can now complete the proof of Theorem \ref{Theorem entropy for amenable intro}:

\begin{theorem}\label{Theorem entropy and amenability}
Let \(S\) be an amenable action of \(G\).
Then for any convex function \(f\) with \(f(1)=0\) and finitely supported measure \(\lambda\) on \(G\) we have:
\[I_{\lambda,f}(S)=I_{\lambda,f}(G)\]
\end{theorem}
\begin{proof}
Note that by Lemma \ref{lemma properties of I lambda f}(3) we only need to show that \(I_{\lambda,f}(S)\geq I_{\lambda,f}(G)\). Using Lemma \ref{lemma properties of I lambda f}(1) we may assume \(S\) is a standard measure space and using Lemma \ref{lemma ergodic decomposition and entropy}(2) and \cite[Corollary B]{AmenableGeneral} we may assume that the action of \(G\) on \(S\) is ergodic.\\
Take a non-principle ultrafilter \(\mathcal{U}\) on \(\mathbb{N}\).
By Theorem \ref{Theorem Amenable actions and ultralimit} there is a measure \(\nu\) on \(S\) and BQI measures \(\eta_n\) on \(G\times \mathbb{N}\) (supported everywhere) with 
\[(S,\nu)_{RN}\cong (\mathcal{U}\lim G\times \mathbb{N},\mathcal{U}\lim \eta_n)_{RN}\] 
Consider the QI \(G\)-spaces \(Z:=(S,\nu)_{RN}\cong\mathcal{U}\lim (G\times \mathbb{N},\mathcal{U}\lim \eta_n)_{RN}\) and \(Y:=\mathcal{U}\lim (G\times \mathbb{N},\mathcal{U}\lim \eta_n)\). Then we conclude by Lemma \ref{lemma properties of I lambda f}(1) and Corollary \ref{Corollary lower semi cont of I in ultralimit}:
\[I_{\lambda,f}(S)= I_{\lambda,f}(Z)= I_{\lambda,f}(Y)\geq \mathcal{U}\lim_{n} I_{\lambda,f}(G\times \mathbb{N})=I_{\lambda,f}(G\times \mathbb{N})\geq I_{\lambda,f}(G)\]
proving the result.
\end{proof}

One can use the above theorem to conclude that certain actions do not posses a non-trivial amenable factor:

\begin{corollary}\label{cor no amenable factors to entropy minimizers}
Suppose  \(f\) is strictly convex function with \(f(1)=0\) and \(\lambda\) generating probability measure on \(G\). Suppose \((S,\nu)\) is a QI \(G\)-space with \(h_{\lambda,f}(S,\nu)=I_{\lambda,f}(S)\).\\
Then if \(\pi:(S,\nu)\to(T,m)\) is a factor and \(T\) is an amenable \(G\)-space, then \(\pi\) is a measure preserving extension.
\end{corollary}
\begin{proof}
Using Theorem \ref{Theorem entropy and amenability} for \(T\) we see
\[I_{\lambda,f}(G)\geq I_{\lambda,f}(S)=h_{\lambda,f}(S,\nu)\geq h_{\lambda,f}(T,m)\geq I_{\lambda,f}(T)=I_{\lambda,f}(G)\]
This implies that we have equality everywhere, in particular \(h_{\lambda,f}(S,\nu)= h_{\lambda,f}(T,m)\) which implies \(\pi\) is measure preserving extension (by \cite[Lemma 3.3(3)]{sayag2022entropy}).
\end{proof}

\subsection{Entropy in boundary actions}

\subsubsection{Entropy for hyperbolic groups and linear groups - proof of Theorems \ref{Theorem Entropy for hyperbolic groups intro}, \ref{Theorem entropy for lattices intro}}\label{subsection Entropy for hyperbolic groups}
In \cite{Spatzier1991FundamentalGO}, Spatzier and Zimmer proved that certain topological boundary actions are amenable with respect to all measure classes. Adams generalized their result to all hyperbolic groups in the paper \cite{ADAMS1994765}. The notion of topological amenability of a locally compact groupoid was introduced \cite{renault2006groupoid} as the existence of an approximate invariant mean. It implies the Zimmer amenability for any quasi-invariant measure. In \cite{anantharaman1987systemes} the converse was shown to be true. For a comprehensive study of the amenability in both settings for the general case of groupoids, see the book \cite{anantharaman2000amenable} (in particular \cite[Appendix B]{anantharaman2000amenable} for the topological amenability for hyperbolic groups on their Gromov boundary).

\begin{theorem}\label{Theorem topological amenability entropy}
Suppose \(G\) is a discrete countable group, and let \(X\) be a compact \(G\)-space which is topologically amenable. Then for any convex function \(f\) with \(f(1)=0\) and finitely supported measure \(\lambda\) on \(G\) 
\[I_{\lambda,f}(G)=I_{\lambda,f}^{top}(X)\]
\end{theorem}
\begin{proof}
Lemma \ref{lemma properties of I lambda f}(2) yields \(I_{\lambda,f}^{top}(X)\leq I_{\lambda,f}(G)\). By Lemma \ref{lemma only QI measures on topological actions} we only need to show that for quasi invariant Borel measure \(\nu\in M(X)\) we have \(h_{\lambda,f}(X,\nu)\geq I_{\lambda,f}(G)\). Let \(S\) be the quasi invariant \(G\)-space \(X\) with the measure class of \(\nu\). It is amenable. Thus, we may use Theorem \ref{Theorem entropy and amenability} to conclude 
\[h_{\lambda,f}(X,\nu)\geq I_{\lambda,f}(S)=I_{\lambda,f}(G)\]
\end{proof}

From this (and the result by Adams \cite{ADAMS1994765} mentioned above) we get Theorem \ref{Theorem Entropy for hyperbolic groups intro}:
\begin{theorem}\label{Theorem Entropy for hyperbolic groups}
Let \(G\) be a hyperbolic group, and let \(\partial G\) be its Gromov boundary. Then for any convex function \(f\) with \(f(1)=0\) and finitely supported measure \(\lambda\) on \(G\) we have:
\[I_{\lambda,f}(G)=I_{\lambda,f}^{top}(\partial G)\]
\end{theorem}

We also get Theorem \ref{Theorem entropy for lattices intro}:

\begin{theorem}\label{Theorem Entropy for lattices}
Let \(G\) be a discrete subgroup in a semi-simple lie group \(\mathbf{G}\). Let \(\mathbf{B}\) be a Borel subgroup and consider the flag space \(X = \mathbf{G}/\mathbf{B}\), then for any \(\lambda, f\) as above we have:
\[I^{top}_{\lambda,f}(X)=I_{\lambda,f}(G)\]
\end{theorem}
\begin{proof}
It is clear since the action of \(G\) on \(X\) is topologically amenable, see \cite[Example 5.2.2(2)]{anantharaman2000amenable} - the groupoid \(G \rtimes X\) is equivalent to \(\mathbf{B} \rtimes (\mathbf{G}/G)\) which is amenable since \(\mathbf{B}\) is an amenable group.    
\end{proof}

\subsubsection{Entropy in the boundary of the free group - proof of Theorem \ref{Theorem entropy for free groups intro}}\label{subsection Entropy for free groups}

We recall some notations and results from \cite[Section 7]{sayag2022entropy}.\\
Let \(F=F_d\) be a free group on \(d\geq2\) generators, which we denote by \(a_1,\dots,a_d\). Denote \(a_{-i}=a_i^{-1}\).\\
We consider the boundary \(X=\partial F\) as the space of infinite length reduced words, e.g. the subset of \(\{a_{\pm1},\dots,a_{\pm d}\}^{\mathbb{N}}\) consisting of words \((w_n)_{n\geq0}\) with \(w_{n+1}\neq w_n^{-1}\). The space \(X\) is a compact metrizable space with a continuous \(F\)-action. \\
We denote by \(\Delta_d\) the set of generating symmetric probability measures on \(F\) supported on \(\{a_{\pm i}\}_{i=1}^{d}\).\\
Let \(\mu\in \Delta_d\) and denote \(p_i=\mu(a_i)\).
We recall that there are unique \((q_i)_{i=\pm1,\dots,\pm d}\) in \((0,1)\) with \(q_j=p_j+q_j\sum_{j\neq i=\pm1,\dots,\pm d} p_i q_{-i}\). They satisfy \(q_i=q_{-i}\) and if we define \(v_i=\frac{q_i}{1+q_i}\) then \(\sum_{i=\pm1,\dots,\pm d} v_i=1\).\\
Let \(f\) be a strictly convex smooth function on \((0,\infty)\) with \(f(1)=0\). We denote \(\Psi_f(z)=f(z)-zf^{\prime}(z)+f^{\prime}(\frac{1}{z})\).
\begin{definition}
We define \(T:\Delta_d\to \Delta_d\) as follows - given \(\mu\in\Delta_d\), denote \(p_i=\mu(a_i)\) and let \(q_i\) be as above, define \(T(\mu)=\lambda\in\Delta_d\) by:
\[\lambda(a_{\pm j})=\lambda_j=\frac{c}{\Psi_f(q_j)-\Psi_{f}(\frac{1}{q_j})}\]
where \(c\) is a normalization \(c=\big(2\sum_{i=1}^{d}\frac{1}{\Psi_{f}(q_j)-\Psi_f(\frac{1}{q_j})}\big)^{-1}\).
\end{definition}
We denote by \(\nu_{\mu}\) the \(\mu\)-harmonic measure on \(\partial F_{d}\). We showed in \cite[Proposition 7.6]{sayag2022entropy}
\begin{lemma}\label{lemma nu mu local minimum}
\(T\) is a bijection and for \(\lambda=T(\mu)\) we have that \(\nu_{\mu}\) minimize \((\lambda,f)\)-entropy in its measure class on \(\partial F_{d}\).
\end{lemma}

Using Theorem \ref{Theorem Entropy for hyperbolic groups intro} and the previous Lemma we conclude Theorem \ref{Theorem entropy for free groups intro}:

\begin{theorem}\label{Theorem entropy for boundary of free groups}
For any \(\mu\in\Delta_d,\lambda=T(\mu)\) we have:
\[I_{\lambda,f}^{top}(\partial F_d)=h_{\lambda,f}(\partial F_d,\nu_\mu)\]
\end{theorem}
\begin{proof}
Consider the QI \(F_{d}\)-space \(X\) which is \((\partial F_{d},\nu_{\mu})\) (that is considering the measure class of \(\nu_{\mu}\)).\\
By Theorem \ref{Theorem Entropy for hyperbolic groups}, Lemma \ref{lemma properties of I lambda f}(3) and Lemma \ref{lemma nu mu local minimum} we conclude:
\[I_{\lambda,f}^{top}(\partial F_{d})= I_{\lambda,f}(F_{d})\geq I_{\lambda,f}(X)= h_{\lambda,f}(\partial F_{d},\nu_{\mu})\]
As the other direction is trivial, the Theorem follows.
\end{proof}

\begin{remark}
In \cite[Theorem 7.11]{sayag2022entropy} we showed that \(h_{\lambda,f}(\partial F_{d},\nu_{\mu})=I_{\lambda,f}(F_{d})\) which implies Theorem \ref{Theorem entropy for boundary of free groups} when we combine it with Theorem \ref{Theorem Entropy for hyperbolic groups}. However I wanted to stress that all the information from the previous paper we use is \cite[Proposition 7.6]{sayag2022entropy}, which is a direct computation.
\end{remark}

A nice application of Corollary \ref{cor no amenable factors to entropy minimizers} is:

\begin{corollary}
Let \(\mu\in \Delta_d\), then \((\partial F_{d},\nu_\mu)\) has no amenable factors other then itself.
\end{corollary}
\begin{proof}
Since \((X,\nu_{\mu})\) doesn't posses any measure preserving extensions (e.g. being the Poisson boundary of \(\mu\)), the corollary follows from Proposition \ref{cor no amenable factors to entropy minimizers} and Lemma \ref{lemma nu mu local minimum}, taking \(f(t)=t\ln(t)\) and \(\lambda=T(\mu)\).
\end{proof}

\begin{remark}
It is probable that the Corollary can be deduced from \cite{Nevo2000RigidityOF} using the embedding of the Free group as a lattice in a Lie group.
Our argument avoids such an embedding, and might be useful in situations where one has a good boundary theory but no direct connection to a Lie group.
\end{remark}



\section{Further questions and outlook}

\subsection{Entropy minimizing}

\subsubsection{Spaces attaining the minimal entropy and uniqueness of minimizing measures}
In light of Theorem \ref{Theorem entropy for free groups intro}, a natural question is about uniqueness of minimizers. The following very special case is still open:
\begin{problem}
Consider the action \(F_d\curvearrowright\partial F_d\), and let \(\lambda\in \Delta_d\) be the uniform measure. Are there any measures \(\nu\) on \(\partial F_{d}\) other then \(\nu_{\lambda}\) for which \(h_{\lambda}(\nu)\) is the minimal value for the Furstenberg \(\lambda\)-entropy?
\end{problem}

More generally, we formulate the following problem:
\begin{problem}
Suppose \(\lambda\) is a finitely support generating probability measure on \(G\), \(f\) is smooth absolutely convex function with \(f(1)=0\) and \(f^{\prime}(\infty)=+\infty\).
\begin{itemize}
    \item Is there an amenable action \((S,\nu)\) such that \(h_{\lambda,f}(\nu)=I_{\lambda,f}(G)\) ?
    \item Is an action as above unique up to measure preserving extensions?
\end{itemize}
\end{problem}
If the answer to Problem \ref{problem ultralimit is amenable} is yes (in particular for exact groups, see Lemma \ref{lemma exact groups and ultralimits}), then the existence follows from Proposition \ref{Proposition lower semi continouity of entropy}.\\
We have shown in \cite[Lemma 3.10]{sayag2022entropy}, that given ergodic \(S\), such a measure \(\nu\) on it is unique if it exist. It is also clear that a factor between minimizing amenable actions is a measure preserving extension.

\subsubsection{Boundary actions that are not amenable}
There are interesting boundary actions that are not topologically amenable, for example the action of 
\(SL_n(\mathbb{Z})\curvearrowright \mathbb{P}^{n-1}_{\mathbb{R}}\) for \(n\geq 3\).

\begin{problem}
Suppose that \(G\) is a lattice in a real semi-simple Lie group \(\mathbf{G}\) with no compact factors (or any Zariski-dense discrete subgroup). Let \(\mathbf{P}\) be a parabolic (but not a Borel) subgroup and consider the flag space \(X = \mathbf{G}/\mathbf{P}\). Let \(f\) be a convex function with \(f(1)=0\) and let \(\lambda\) be a finitely supported measure on \(G\). 
\begin{enumerate}
    \item Can we describe the numbers \(I_{\lambda,f}^{top}(X)\) in terms of the group \(G\) ?
    \item  Considering \(X\) as a quasi-invariant \(G\)-space with the Lebesgue measure class, is \(I_{\lambda,f}(X)\) the same as \(I_{\lambda,f}^{top}(X)\) ?
    \item Is there a minimizing measure of the Lebesgue measure class?
    \item Is the minimizing measure on \(X\) unique in this case? 
\end{enumerate}
\end{problem}


\subsubsection{Possible generalizations}
The notion of amenability has a natural generalization to groupoids, see \cite{anantharaman2000amenable}. It is natural to ask if there is a version for entropy and our results in this setting. 
\(\)\\
\(\)\\
Another question is about weakening the assumption of finite support on \(\lambda\) in Theorem \ref{Theorem entropy for amenable intro}. In the case of \(f(t)=t\ln(t)\) the equality holds for \(\lambda\) with finite entropy.
Indeed, in order to prove this, one need to refine Lemma \ref{lemma properties of I lambda f}(4) to consider only \(\omega\)'s of the form \(\lambda_{\epsilon}=(1-\epsilon)\sum_{n}\epsilon^{n}\lambda^{*n}\), and to prove that the for each \(0<a<1\) and a collection \((X_i,\nu_i)\) with \((X,\nu)=\mathcal{U}\lim (X_i,\nu_i)\) one has: \(\mathcal{U}\lim h_{\lambda}(\lambda_{a}\ast\nu_{i})= h_{\lambda}(\lambda_{a}\ast \nu)\). The details are left to the reader.
\(\)\\
\(\)\\
Another natural problem is a relative version for Theorem \ref{Theorem entropy for free groups intro}:
\begin{problem}
Let \(\pi: X\to Y\) be a factor between QI \(G\)-spaces for which \((X, Y)\) is an amenable pair (see Definition \ref{Definition amenable pair}), does the equality \(I_{\lambda,f}(X) = I_{\lambda,f}(Y)\) hold?
\end{problem}

\subsection{Amenable actions}

\subsubsection{Amenable actions and ultralimits}
The reason for the ergodicity assumption in Theorem \ref{Theorem ultralimit ameanble action intro} is that \cite[Theorem A]{AmenableGeneral} deals only with ergodic actions. However, the Poisson boundary of a time dependent matrix-valued random walk is always an amenable action, but not necessarily ergodic - the ergodic component corresponds to the Poisson boundary of the random walk on the \([\ell_{n}]\) induced from \(V_n\to [\ell_{n}]\).\\
If the generalized Poisson boundaries cover all amenable actions, one can easily eliminate the ergodicity assumption in our realization of amenable actions as ultralimits.
\(\)\\
One can ask the following converse to Theorem \ref{Theorem ultralimit ameanble action intro}:
\begin{problem}\label{problem ultralimit is amenable}
Suppose \((S_{i},\nu_{i})\) are uniformly QI amenable actions, is \(\mathcal{U}\lim (S_i,\nu_i)\) necessarily amenable?
\end{problem}
\begin{lemma}\label{lemma exact groups and ultralimits}
The answer to Problem \ref{problem ultralimit is amenable} is yes for exact groups.
\end{lemma}
\begin{proof}
By \cite[Theorem 3.16]{buss2020amenability}, if \(G\) is exact, a \(G\)-space \(S\) is amenable iff there is a u.c.p. (unital completely positive) \(G\)-map \(\Phi: L^{\infty}(G)\to L^{\infty}(S)\).\\
Suppose \((S_{i},\nu_{i})\) are amenable and let \((S,\nu)=\mathcal{U}\lim (S_i,\nu_i)\). By assumption we get \(G\)-equivariant u.c.p mappings \(\varphi_{i}:L^{\infty}(G)\to L^{\infty}(S_i)\). Since \(\varphi_i\) are u.c.p. they are contractions.\\
Define \(\varphi: L^{\infty}(G)\to L^{\infty}(S)\) by \(\varphi(f)=\mathcal{U}\lim \varphi_i(f)\in L^{\infty}(S)\). Then \(\varphi\) is a \(G\)-equivariant u.c.p. which implies that \(S\) is amenable.
\end{proof}
\subsubsection{Amenable actions and entropy}
Another natural problem is whether a converse to Theorem \ref{Theorem entropy for amenable intro} holds, namely:
\begin{problem}
Suppose \(S\) is a QI \(G\)-space with \(I_{\lambda,f}(S)=I_{\lambda,f}(G)\) for any finitely supported \(\lambda\) and convex \(f\) with \(f(1)=0\). Is the \(G\)-space \(S\) necessarily amenable?
\end{problem}
Of course, a negative answer to Problem \ref{problem ultralimit is amenable} is also a negative answer for this, however this Problem is interesting also for exact groups.

\section*{Acknowledgments}
I would like to thank my thesis advisor Prof. Yehuda Shalom for his guidance, suggestion of the problem and his helpful remarks.\\
I would like to thank Uri Kreitner for useful discussion regarding Lemma \ref{lemma properties of majorants}(6).\\
I would also like to thank Prof. Jesse Peterson and Prof. Nicolas Monod for helpful email communication regarding amenability.\\
I was partially supported by the ISF grant 1483/16.

\appendix

\section{Spaces of uniform integrability}\label{appendix uniform integrability}
In this appendix, we introduce a filtration on \(L^{1}\)-space by a  Banach-spaces of uniformly integrable functions, indexed by the convex set of "majorants". This filtration form a clean packaging for uniform-integrability. The introduction of these spaces allows one to extract the \enquote{absolutely-integrable part} out of any function, which is very useful for our study of ultralimits.

\begin{definition}\label{definition space of majorants}
The space of majorants \(\mathbf{M}\) is the set of functions \(\rho:[0,1]\rightarrow[0,1]\) such that:
\begin{enumerate}
    \item \(\rho(0^+)=\rho(0)=0,\: \rho(1)=1\).
    \item
    \(\rho\) is concave: for any \(x,y,t\in[0,1]\) we have \(\rho((1-t)x+ty)\geq (1-t)\rho(x)+t\rho(y)\).
\end{enumerate}
\end{definition}

\begin{example}\label{example power of t in M}
Given \(q\geq1\) the function \(\rho(t)=t^\frac{1}{q}\) is in \(\mathbf{M}\).
\end{example}

We have the following lemma:
\begin{lemma}\label{lemma properties of majorants}
The space of majorant \(\mathbf{M}\) satisfies the following properties:
\begin{enumerate}
    \item 
    Any \(\rho\in\mathbf{M}\) is continuous, non-decreasing, sub-additive and satisfies \(\rho(t)\geq t\).
    \item
    \(\mathbf{M}\) is strongly convex: If \((\rho_n)_{n\geq 0} \) are in \(\mathbf{M}\) and \((p_n)_{n\geq 0}\) is a probability measure on \(\mathbf{N}\) then \(\rho=\sum_{n=0}^\infty p_n\:\rho_n\) is in \(\mathbf{M}\).
    \item
        If \(\rho,\eta\in\mathbf{M}\) then \(\rho\circ\eta \in \mathbf{M}\).
    \item
     If \(\rho,\eta\in\mathbf{M}\) then \(\max(\rho,\eta) \in \mathbf{M}\).
    \item
    If \(\rho\in\mathbf{M}\) and \(K\geq 1\) then the function \(\rho_{K}(t)=\min(1,K\cdot\rho(t)\)) is in \(\mathbf{M}\)
    \item
    Let \(\rho_0:[0,1]\rightarrow[0,1]\) be any function with \(\rho_0(0^+)=\rho(0)=0\). Then, there is a function \(\rho\in\mathbf{M}\) with \(\rho_0(t)\leq \rho(t)\). Moreover, we can choose such \(\rho\in \mathbf{M}\) such that \(\frac{\rho_0}{\rho}(0^+)=0\).
\end{enumerate}
\end{lemma}
\begin{proof}
\begin{enumerate}
    \item
    \(\rho\) non decreasing- suppose \(t\leq s\), then \(s\in [t,1]\) and we conclude by concavity that  \(\rho(s)\geq\min(\rho(1),\rho(t))=\rho(t)\).\\
    Note that the concavity implies \(\rho(tx)\geq t\rho(x)\) for \(x,t\in [0,1]\). Since \(\rho(1)=1\) we get \(\rho(t)\geq t\).\\
    \(\rho\) is sub-additive: for \(x,y\in[0,1]\) with \(x+y\leq 1\) we have 
    \[\rho(x)+\rho(y)=\rho(\frac{x}{x+y}(x+y))+\rho(\frac{y}{x+y}(x+y))\geq \frac{x}{x+y}\rho(x+y)+\frac{y}{x+y}\rho(x+y)=\rho(x+y)\]
    \(\rho\) is continuous: for any \(\epsilon>0\) there is \(\delta>0\) with \(\rho(\delta)<\epsilon\) and then for \(0<t<\delta\) we have for any \(s\in[0,1-t]\) that \(\rho(s)\leq\rho(s+t)\leq\rho(s)+\epsilon\).
    \item It is obvious that \(\rho\) is a concave function with \(\rho(0)=0,\rho(1)=1\). 
    To see that \(\rho(0^+)=0\), let \(\epsilon>0\) and take \(N\) with \(\sum_{n=N+1}^{\infty}p_{n}<\frac{\epsilon}{2}\).
    Since \(\rho_n(0^+)=0\) for all \(n\) we can find \(\delta>0\) with \(\rho_n(\delta)<\frac{\epsilon}{2}\: (n=1,\dots N)\). Hence 
    \(\rho(\delta)<\sum_{n=1}^{N}p_{n}\frac{\epsilon}{2}+\sum_{n=N+1}^{\infty}p_{n}<\epsilon\).
    \item
    By item 1 we conclude that \(\kappa=\rho\circ\eta\) is continuous, non-decreasing and concave. Thus \(\kappa(0^+)=\kappa(0)=0,\kappa(1)=1\).
    \item
    By item 1 we conclude that \(\lambda=\max(\rho,\eta)\) is continuous and concave.
    \item
    For the concavity of \(\rho_{K}\), note that \(\{(t,y): \:y\leq \rho_{K}(t)\} = \{(t,y): y\leq K\rho(t)\} \:\cap [0,1]^{2}\) thus it is a convex set as the intersection of convex sets. Thus \(\rho_{K}\) is concave. 
    \item
    Consider \(G=\{(x,y)\in[0,1]^{2}\big|\: y\leq \rho_{0}(x)\}\cup\{(1,1)\}\) and let \(C=conv(G)\) be the convex hull. As \(G\) is compact we get that \(C\) is compact. Define \(\rho_{1}(x)=\max\{y: (x,y)\in C\}\), then \(\rho:[0,1]\to [0,1]\) is concave, \(\rho(1)=1\) and \(\rho\geq \rho_0\). Let us show that \(\rho(0)=\rho(0^{+})=0\). Indeed: let \(\epsilon>0\), by assumption there is \(\delta>0\) such that for \(t\leq \delta\) we have \(\rho_0(t)<\epsilon\).\\
    Thus \(G\subset [0,\delta]\times[0,\epsilon] \: \bigcup \: [\delta,1]\times [0,1]\). This implies that \(C\subset \{(x,y)\in [0,1]^{2}\big| y\leq\epsilon+\frac{1-\epsilon}{\delta} x\}\). We conclude that for \(x\leq \epsilon \cdot \delta\) we have \(\rho(x)\leq 2\epsilon\). Thus \(\rho(0)=\rho(0^{+})=0\).\\
    Thus \(\rho\in\mathbf{M}\) satisfies \(\rho\geq\rho_0\).
    For the moreover part we can take \(\rho(t)^\frac{1}{2}\in\mathbf{M}\) and then \(\lim_{t\to0^+}\frac{\rho_0(t)}{\rho(t)^{\frac{1}{2}}}=0\).
\end{enumerate}
\end{proof}


\begin{definition}
Let \(\mathfrak{X}=(X,\Sigma,\nu)\) be a probability space and let \(\rho\in\mathbf{M}\).
\begin{enumerate}
    \item
We define the normed space \(\mathcal{C}_\rho (\mathfrak{X})\) to be the space of functions (mod a.e. 0 functions) in \(L^1(\mathfrak{X})\) such that
\[||f||_\rho=||f||_{\mathcal{C}_\rho (\mathfrak{X})}:=\sup_{A\in\Sigma:\ \nu(A)\neq0}\frac{1}{\rho(\nu(A))}\intop_{A}|f|d\nu<\infty\]
\item
Given a probability measures \(m\) on \((X,\Sigma)\),  we say that \(m\) is \emph{\(\rho\)-absolutely continuous with respect to \(\nu\)} if \(m\) is absolutely continuous with respect to \(\nu\) and \(||\frac{dm}{d\nu}||_\rho\leq1\). We denote this relation by \(m \stackrel{\rho}{\ll}\nu\). 
\end{enumerate}
\end{definition}
\begin{remark}
Note that \(m \stackrel{\rho}{\ll}\nu\) if and only if  \(m(A)\leq\rho\big(\nu(A)\big)\) for every \(A\in\Sigma\).

\end{remark}
\begin{example}\label{example of C rho} 
\(\)
\begin{itemize}
    \item 
    For \(\infty > p>1\) let \(q=\frac{p}{p-1}\) be the H\"{o}lder conjugate and let \(\rho(t)=t^\frac{1}{q}\). Then \(\rho\in\mathbf{M}\) (see Example \ref{example power of t in M}). Note we have a contraction
    \(L^p(\mathfrak{X})\subset \mathcal{C}_\rho (\mathfrak{X})\).\\
    Indeed, by H\"{o}lder's inequality:
    \[\intop_A|f|d\nu\leq (\intop_X |f|^p)^\frac{1}{p}(\intop_X \boldsymbol{1}_{A})^\frac{1}{q}=||f||_{L^p} \nu(A)^\frac{1}{q}\]
    \item 
    For \(\rho(t)=t\) we have that \(||\cdot||_{\mathcal{C}_\rho}=||\cdot||_\infty\) and thus \(L^\infty(\mathfrak{X})= \mathcal{C}_\rho (\mathfrak{X})\).\\ Indeed, the inequality \(||\cdot||_{\rho}\leq||\cdot||_{\infty}\) is the same as the previous item with \(p=\infty\). For the reverse inequality, take \(A=\{|f|\geq||f||_{\mathcal{C}_\rho}+\epsilon\}\). If \(\nu(A)\neq0\) then \(\frac{1}{\nu(A)}\intop_{A}|f|\geq||f||_{\mathcal{C}_\rho}+\epsilon\) which is a contradiction. Thus \(f\) is essentially bounded by \(||f||_{\mathcal{C}_{\rho}}\).
\end{itemize}
\end{example}

The next lemma explains the role of the spaces \(\mathcal{C}_{\rho}(\mathfrak{X})\):
\begin{lemma}\label{lemma properties of C rho spaces}
Let \(\mathfrak{X}\) be a probability space.
\begin{enumerate}
    \item 
    For any \(\rho \in \mathbf{M}\) the space \(\mathcal{C}_{\rho}(\mathfrak{X})\) is a Banach space.
    \item
    Let \(\rho_0,\rho_1\in\mathbf{M}\) and \(K>0\). Suppose that \(\frac{\rho_0}{\rho_1}\leq K\), then the inclusion induces a continuous embedding \(\mathcal{C}_{\rho_0}(\mathfrak{X})\subset\mathcal{C}_{\rho_1}(\mathfrak{X})\) with operator norm bounded by \(K\).
    \item
    Let \(\rho\in\mathbf{M}\) and \(K\geq1\), and let \(\rho_{K}(t)=\min(1,K\cdot\rho(t))\). Then \(||f||_{\rho_{K}}\leq \max\big(||f||_{L^{1}},\frac{1}{K}||f||_{\rho}\big)\).
    \item
    Let \(\rho_0,\rho_1\in\mathbf{M}\) be such that \(\frac{\rho_0}{\rho_1}(0^+)=0\). Then \(\mathcal{C}_{\rho_0}(\mathfrak{X})\) is contained in the closure of the simple functions in \(\mathcal{C}_{\rho_1}(\mathfrak{X})\).
    \item
    For any \(f\in L^1(\mathfrak{X})\) there is \(\rho\in\mathbf{M}\) with \(f\in \mathcal{C}_\rho(\mathfrak{X})\).
    
\end{enumerate}
\end{lemma}

\begin{proof}
\begin{enumerate}
    \item
    Suppose \(f_n\in\mathcal{C}_\rho(\mathfrak{X})\) with \(\sum_{n=0}^\infty ||f_n||_{\mathcal{C}_\rho(\mathfrak{X})}<\infty\), we will show their sum converge in \(\mathcal{C}_\rho(\mathfrak{X})\). Notice that \(||f||_{L^1}\leq ||f||_\rho\) and thus we have that \(\sum_{n=0}^\infty \intop_X|f_n|d\nu<\infty\) and in particular \(f=\sum_{n=0}^\infty f_n\) converges almost everywhere and in \(L^1(\mathfrak{X})\). Moreover
    \(\intop_A |f|\leq \sum_{n=0}^\infty \intop_A|f_n|d\nu\) which shows
    \(f\in \mathcal{C}_\rho(\mathfrak{X}),\: ||f||_{\mathcal{C}_\rho(\mathfrak{X})}\leq \sum _{n=0}^\infty ||f_n||_{\mathcal{C}_\rho(\mathfrak{X})}\).\\
    Applying the same reasoning, we see that for any \(N\), the a.e. converging sum \(\sum_{n=N+1}^{\infty} f_n\) is in \(\mathcal{C}_{\rho}(\mathfrak{X})\) with \(||\sum_{n=N+1}^{\infty} f_n||_{\mathcal{C}_{\rho}(\mathfrak{X})}\leq \sum_{n=N+1}^{\infty} ||f_{n}||_{\mathcal{C}_{\rho}(\mathfrak{X})}\). Thus
    \[\limsup_{N\to \infty}||f-\sum_{n=0}^N f_n||_{\mathcal{C}_\rho(\mathfrak{X})}=\limsup_{N\to\infty}||\sum_{n=N+1}^\infty f_n||_{\mathcal{C}_\rho(\mathfrak{X})}\leq\limsup_{N\to\infty}\sum_{n=N+1}^\infty ||f_n||_{\mathcal{C}_\rho(\mathfrak{X})}=0\]
    Which yields \(f=\sum_{n=0}^\infty f_n\) in \(\mathcal{C}_{\rho}(\mathfrak{X})\), as desired.
    \item
    Indeed:
    \[\intop_A |f|\:d\nu\leq ||f||_{\rho_0}\rho_0\big(\nu(A)\big)\leq K\cdot||f||_{\rho_0}\rho_1\big(\nu(A)\big)\Longrightarrow ||f||_{\rho_1}\leq K||f||_{\rho_0}\]
    \item
    Indeed:
    \[\intop_{A}|f|\:d\nu\leq \min\Big(||f||_{L^1}, ||f||_{\rho}\rho\big(\nu(A)\big)\Big)\leq \max\big(||f||_{L^1}, \frac{1}{K}||f||_{\rho}\big) \cdot\:  \min\big(1,K\rho(\nu(A))\big)\]
    \item
    Let \(f\in \mathcal{C}_{\rho_0}(\mathfrak{X})\), we will show it can be approximated in \(\mathcal{C}_{\rho_1}(\mathfrak{X})\) by simple functions. We may assume \(f\geq 0\). Let \(\epsilon>0\) and take \(C\) large enough such that:
    \[\forall t\in[0,\nu(\{f>C\})]:\quad \frac{\rho_0}{\rho_1}(t)<\frac{\epsilon}{2||f||_{\rho_{0}}}\]
    Let \(\varphi\) be a simple function such that on the set \(\{f\leq C\}\) we have \(f-\frac{\epsilon}{2}\leq\varphi\leq f\) and outside of this set it is \(0\). Then:
\begin{align*}
\intop_A |f-\varphi| \: d\nu \leq\frac{\epsilon}{2}\cdot\nu(A)+\intop_{A\cap\{f>C\}}|f|\: d\nu \leq \frac{\epsilon}{2}\cdot \rho_1\big(\nu(A)\big)+||f||_{\rho_0}\rho_0\Big(\nu\big(A\cap\{f>C\}\big)\Big) \\
\leq\rho_{1}\big(\nu(A)\big)\bigg(\frac{\epsilon}{2}+||f||_{\rho_0}\cdot \frac{\rho_0}{\rho_1}\big(\nu(A\cap\{f>C\})\big)\bigg)\leq\epsilon\rho_1\big(\nu(A)\big)
\end{align*}
we conclude \(||f-\varphi||_{\rho_{1}}\leq\epsilon\).

\item
Without loss of generality we may assume \(0\neq f\geq0\). We define a function \(\rho_0:[0,1]\to [0,1]\) in the following way:
\[\rho_0(t)=\frac{1}{||f||_{L^1}}\sup\bigg\{ \intop_A f\:d\nu \Big| \nu(A)= t \bigg\}\]
We obviously have \(\rho_0(0)=0\). By Lebesgue dominant convergence theorem, for any sequence of measurable sets \((A_n)\) with \(\nu(A_n)\to0\) we have \(\intop_{A_n} f d\nu\to 0\). Thus \(\rho_0(0^+)=0\), using Lemma \ref{lemma properties of majorants}(5) we can find \(\rho\in\mathbf{M}\) with \(\rho_0\leq\rho\). With this \(\rho\) we obviously have \(f\in\mathcal{C}_\rho (\mathfrak{X})\).
\end{enumerate}
\end{proof}

\begin{corollary}\label{Corollary exsitence of majorant for measures}
Given a measured space \((X,\Sigma)\) and two probability measures \(m,\nu\) on it. If \(m\ll\nu\) then there is some \(\rho\in\mathbf{M}\) with \(m\stackrel{\rho}{\ll}\nu\).
\end{corollary}
\begin{proof}
Using Lemma \ref{lemma properties of C rho spaces}(4) for \(f=\frac{dm}{d\nu}\) we find \(\rho_0\in\mathbf{M}\) with \(f\in \mathcal{C}_{\rho}(\mathfrak{X})\). Consider \(M=||f||_{\rho}\) (note \(M\geq 1\)) and \(\rho(t)=\min(M\cdot\rho_0(t),1)\), then it is easy to see \(\rho\in\mathbf{M}\) and that \(||f||_{\rho}\leq 1\).
\end{proof}

\begin{example}
Item 4 in Lemma \ref{lemma properties of C rho spaces} is indeed \enquote{sharp}, in the sense that it is not always true that the simple (bounded) functions are dense in \(\mathcal{C}_{\rho}(X)\). Let us give an example:\\
Take \(\rho(t)=\sqrt{t}\) which is a majorant. Let \(X=[1,\infty)\) with the Borel measure \(\nu=\frac{2 dx}{x^{3}}\) and let \(f:X\to\mathbb{R}\) be the function \(f(x)=x\). We claim \(f\in\mathcal{C}_{\rho}(X,\nu)\) but not in the closure of the bounded functions.\\
The distribution function of \(f\) is \(\lambda(s)=\nu(\{f>s\})=\frac{1}{s^{2}}\cdot \boldsymbol{1}_{[1,\infty]}\). Note that for any \(v\in(0,1]\) we have:
\[\sup_{A \:\:\text{Borel with}\:\: \nu(A)=v}\:\: \frac{1}{\rho(\nu(A))}\intop_{A} f d\nu=\frac{1}{\rho(v)}\intop_{\{f>\lambda^{-1}(v)\}} f d\nu=\frac{1}{\sqrt{v}}\intop_{\lambda^{-1}(v)}^{\infty} \lambda(s)\: ds=\frac{1}{\sqrt{v}}\intop_{\frac{1}{\sqrt{v}}}^{\infty} \frac{ds}{s^{2}}=1\]
Thus we see \(||f||_{\rho}=1\). On the other hand, we claim that \(||f-\varphi||\geq 1\) for any bounded \(\varphi\). Indeed, suppose \(\varphi\) is a function bounded by \(C\), then consider a small \(v>0\) and \(A=\{f>\lambda^{-1}(v)\}\). We have by the previous computation:
\[||f-\varphi||_{\rho}\geq \frac{1}{\rho(\nu(A))}\intop_{A}|f-\varphi| d\nu \geq \frac{1}{\rho(v)}\intop_{\{f>\lambda^{-1}(v)\}} (f-C)\: d\nu=1- \frac{1}{\sqrt{v}}C\nu(\{f>\lambda^{-1}(v)\})=1-C\cdot \sqrt{v}\]
Taking \(v\to0\) we get \(||f-\varphi||_{\rho}\geq 1\).
\end{example}

\begin{lemma}\label{lemma concave of rho effect on rho norm}
Let \(\mathfrak{X}\) be a probability space and let \(\rho\in\mathbf{M}\).
Suppose that \(f\in\mathcal{C}_{\rho}(\mathfrak{X})\) and that \(\varphi\in L^{\infty}(\mathfrak{X})\) with \(0\leq \varphi\leq 1\). Then :
\[\Bigg|\intop_{X} f\cdot \varphi \:d\nu\Bigg|\leq ||f||_{\rho}\cdot \rho\big(\intop_{X}\varphi\: d\nu\big)\]
\end{lemma}
\begin{proof}
By replacing \(f\) with \(|f|\) we may assume that \(f\geq 0\).\\ Consider \(X\times [0,1]\) with the measure \(\nu\times Leb\) where \(Leb\) is the Lebesgue measure. By Fubini's theorem and the concavity of \(\rho\) we conclude:
\begin{multline*}
\intop_{X} f\cdot \varphi \: d\nu = \intop_{X} f(x) Leb(\{t\in [0,1]|\: t\leq \varphi(x)\}) d\nu(x)=\intop_{\{(x,t)\in X\times [0,1]\:|\: t\leq \varphi(x)\}} f(x)d(\nu\times Leb)(x,t)\\
=\intop_{0}^{1} \intop_{\{x\in X |\: \varphi(x)\geq t\}} f(x)\:d\nu(x)\: dt\leq ||f||_{\rho}\cdot \intop_{0}^{1}  \rho\big(\nu(\{x\in X\:|\:\varphi(x)\geq t\})\big) dt\leq \\
||f||_{\rho} \cdot\: \rho\Big(\intop_{0}^{1} \nu(\{x\in X\:|\: \varphi(x)\geq t\})\:dt\Big)=||f||_{\rho}\cdot \: \rho\big(\intop_{X}\varphi \:d\nu\big)
\end{multline*}
\end{proof}
The next is a version of a lemma by de la  Vall\'{e}e-Poussin (see \cite[Theorem 19, Chapter 2]{dellacherie1980probabilities}).
\begin{lemma}\label{lemma de la Valle´e-Poussin}
Suppose \(G:[0,\infty)\to[0,\infty)\) is a convex function with \(\lim_{t\to\infty} \frac{G(t)}{t}=+\infty\) and let \(M>0\). Then there is \(\rho\in\mathbf{M}\) and \(K>0\) such that for any probability space \(\mathfrak{X}\) and a measurable function \(f\) one has:
\[\mathbb{E}_{\mathfrak{X}}[G\circ|f|]\leq M\:\implies\: ||f||_{\mathcal{C}_{\rho}(\mathfrak{X})}\leq K\]
\end{lemma}
\begin{proof}
Define for any \(v\in(0,1]\):
\[\rho_{1}(v):=\sup\Big\{t\:\big|\: G(\frac{t}{v})\leq \frac{M}{v} \Big\}\]
Since \(\lim_{t\to\infty} \frac{G(t)}{t}=+\infty\) we have \(\rho_{1}(v)<\infty\) and \(\rho_1(0^{+})=0\).\\
Define \(K=\rho_{1}(1)\) and let \(\rho_0=\frac{\rho_1}{K}\). Using Lemma \ref{lemma properties of majorants}(5) we find \(\rho\in\mathbf{M}\) with \(\rho_0(t)\leq \rho(t)\).\\
Let us show \(\rho,K\) satisfy the properties. Suppose \(\mathfrak{X}=(X,\Sigma,\nu)\) and that \(\intop_{X} G\circ|f| \: d\nu\leq M\). Then for any \(A\in \Sigma\) with \(\nu(A)=v>0\) we have using Jensen:
\[G\Big(\:\frac{1}{v}\intop_{A} |f|\:d\nu\:\Big)\leq \frac{1}{v}\intop_{A} G\circ |f|\: d\nu \leq \frac{M}{v}\implies \intop_{A} |f|\:d\nu \leq \rho_1(v)\leq K\cdot \rho(\nu(A))\]
Thus \(||f||_{\mathcal{C}_{\rho}(\mathfrak{X})}\leq K\).
\end{proof}

The following proposition shows that one can extract the absolutely-integrable part out of any function, for this, note that any majorant is sub-additive.
\begin{prop}\label{prop finding a rho integrable part}
Let \(\mathfrak{X}\) be a probability space, and let \(f\in L^1(\mathfrak{X}),\: \rho\in\mathbf{M},\: C>0\).\\
Then there is \(B\in\Sigma\) satisfying:
\begin{enumerate}
    \item
    \(f\cdot\boldsymbol{1}_{B^c}\in \mathcal{C}_\rho(\mathfrak{X})\) and \(||f\cdot\boldsymbol{1}_{B^c}||_{\rho}\leq C\).
    \item 
    If \(B\neq \emptyset\) then \(\intop_B |f|d\nu>C\cdot \rho\big(\nu(B)\big)\).
\end{enumerate}
\end{prop}
\begin{proof}
For the proof of the proposition we will call \(A\in\Sigma\) \emph{bad} if \(\intop_A |f|\:d\nu>C\rho\big(\nu(A)\big)\).
We begin by verifying the following two observations:
\begin{enumerate}
    \item 
    If \((A_n)_{n}\) are pairwise disjoint bad sets then \(A=\bigcupdot_n A_n\) is bad.
    \item For any \(B\in\Sigma\) we have: \(||f\cdot \boldsymbol{1}_{B^c}||_\rho>C\) if and only if there is a bad set \(D\) disjoint from \(B\).
\end{enumerate}
Indeed:
\begin{enumerate}
    \item 
    Since \(\rho\) is continuous and sub-additive:
\[ \intop_A|f|d\nu=\sum_{n} \intop_{A_n} |f|d\nu>\sum_{n} C\cdot\rho\big(\nu(A_n)\big)\geq C\cdot\rho\big(\sum_{n} \nu(A_n))=C\cdot\rho\big(\nu(A)\big)\]
    
    \item
    By definition, \(||f\cdot \boldsymbol{1}_{B^c}||_\rho>C\) is equivalent to the existence of \(A\in\Sigma\) with \(\intop_{A\cap B^c} |f|\:d\nu>C\cdot\rho\big(\nu(A\cap B^c)\big)\). The "if" direction follows from taking \(A=D\). The "only if" direction follows by taking \(D=A\cap B^c\).
\end{enumerate}
We return to the proof the proposition. Via observation 2, we need to show that there is a bad or empty set \(B\subset X\) that intersects every other bad set. Assume otherwise, then for any bad or empty \(B\subset X\) we have
\[\delta(B):=\sup\big\{\nu(D) \big|\: D\: \text{bad} ,\:\:D\subset B^c \big\}>0\]
Choose \(\mathcal{D}[B]\) to be a bad set disjoint from \(B\) with \(\nu(\mathcal{D}[B])>\frac{\delta(B)}{2}\).\\
Define inductively:
\begin{center}
\begin{tabular}{ c c } 
 \(B_0=\emptyset\) & \(D_0=\mathcal{D}{[B_0]}\) \\
 \(B_1=D_0\) & \(D_1=\mathcal{D}{[B_1]}\) \\ 
 ... & ... \\
 \(B_{n+1}=B_n\cupdot D_n=\bigcupdot_{i=0}^{n} D_i\) \:&\: \(D_{n+1}=\mathcal{D}{[B_{n+1}]}\) \\
\end{tabular}
\end{center}
This is well defined: at stage \(n+1\) of the construction \(D_i\: (i=0,\dots,n)\) are disjoint bad sets so by observation 1 we conclude that \(B_{n+1}\) is also a bad set which allows us to define \(D_{n+1}\).\\
Consider \(B=\bigcup B_n=\bigcupdot D_n\), by observation 1 this is a bad set. Let \(D=\mathcal{D}{[B]}\), then \(\nu(D)>0\). Note that as \(D_n\) are disjoint in a probability space we must have \(\lim_{n\to\infty}\nu(D_n)=0\).
Let \(N\) be such that \(\nu(D_N)<\frac{1}{2}\nu(D)\), then, \(D\) is a bad set, disjoint from \(B_N\) and \(\nu(D)>2\nu(D_N)=2\nu(\mathcal{D}[B_{N}])>\delta(B_N)\).
This is a contradiction to the definition of \(\delta\).
\end{proof}

This has the following corollary:
\begin{corollary}\label{Corollary finding a rho integrable part}
Let \(\mathfrak{X}\) be a probability space, and let \(f\in L^1(\mathfrak{X}),\: \rho\in\mathbf{M},\: \epsilon,C>0\) such that the following implication is true:
\[(A\in\Sigma)\quad \intop_A |f|d\nu>C\rho\big(\nu(A)\big) \Longrightarrow \nu(A)\leq\epsilon\]
Then, there is \(B\in\Sigma\) with \(\nu(B)\leq\epsilon\) such that \(f\cdot\boldsymbol{1}_{B^c}\in \mathcal{C}_\rho(\mathfrak{X})\) and \(||f\cdot\boldsymbol{1}_{B^c}||_{\rho}\leq C\).
\end{corollary}
\begin{proof}
The \(B\) from Proposition \ref{prop finding a rho integrable part} is either empty or satisfies \(\intop_B |f|d\nu>C\cdot \rho\big(\nu(B)\big)\). By the implication this means that \(\nu(B)\leq \epsilon\).
\end{proof}

The last thing we discuss in this section is some functoriality.
Suppose \(\pi: \mathcal{Y}=(Y,\Sigma_{Y},m)\to \mathfrak{X}=(X,\Sigma_{X},\nu)\) is a factor of measure spaces. Then we have the composition with \(\pi\) which we denote by \(\pi^{*}:L^{1}(\mathfrak{X})\to L^{1}(\mathcal{Y})\). We also have the conditional expectation \(\mathbb{E}_{m,\pi}:=\mathbb{E}_{m}[\cdot\:|X]: L^{1}(\mathcal{Y})\to L^{1}(\mathfrak{X})\). 
\begin{lemma}\label{lemma functoriality of C rho spaces}
Suppose \(\pi: \mathcal{Y}\to \mathfrak{X}\) is a factor of probability spaces and let \(\rho\in\mathbf{M}\). Then:
\begin{enumerate}
    \item We have a contraction: \(\mathbb{E}_{m,\pi}:\mathcal{C}_{\rho}(\mathcal{Y})\to\mathcal{C}_{\rho}(\mathfrak{X})\). 
    \item If \(m^{\prime}\stackrel{\rho}{\ll} m\) on \(Y\) then \(\pi_{*}m^{\prime}\stackrel{\rho}{\ll}\nu\) on \(X\).
    \item
    We have an isometric embedding: \(\pi^{*}: \mathcal{C}_{\rho}(\mathfrak{X})\to \mathcal{C}_{\rho}(\mathcal{Y})\).
\end{enumerate}
\end{lemma}
\begin{proof}
\begin{enumerate}
    \item 
    Suppose \(f\in \mathcal{C}_{\rho}(\mathcal{Y})\) and \(A\in \Sigma_{X}\), then:
\[\intop_A |\mathbb{E}_{m,\pi}(f)| d\nu= \intop_{\pi^{-1}(A)} |f| dm \leq ||f||_{\mathcal{C}_{\rho}(Y)}\cdot  \rho\Big(m(\pi^{-1}(A)\big)\Big)=||f||_{\mathcal{C}_{\rho}(Y)}\cdot  \rho\big(\nu(A)\big)\]
Thus \(||\mathbb{E}_{m,\pi}(f)||_{\mathcal{C}_{\rho}(\mathfrak{X})}\leq ||f||_{\mathcal{C}_{\rho}(\mathcal{Y})}\).
\item 
Follows from the previous item as \(||\frac{d\pi_{*}m}{d\nu}||_{\mathcal{C}_{\rho}(X,\nu)}=||\mathbb{E}_{m,\pi}\big(\frac{dm^{\prime}}{dm}\big)||_{\mathcal{C}_{\rho}(X,\nu)}\leq ||\frac{dm^{\prime}}{dm}||_{\mathcal{C}_{\rho}(Y,m)}=1\).
\item
    Suppose \(f\in \mathcal{C}_{\rho}(\mathfrak{X})\) and \(A\in \Sigma_{Y}\), then applying Lemma \ref{lemma concave of rho effect on rho norm}
\[\intop_{A}|\pi^{*}(f)|\: dm=\intop_{Y} 1_{A}\cdot \pi^{*}(|f|)\:dm=\intop_{X} \mathbb{E}_{m,\pi}(1_{A})\cdot |f|\: d\nu \leq ||f||_{\mathcal{C}_{\rho}(X)} \rho\Big(\intop_{X}\mathbb{E}_{m,\pi}(1_{A})\:d\nu\Big)=||f||_{\mathcal{C}_{\rho}(X)} \rho\Big(m(A)\Big)\]
Thus \(||\pi^{*}f||_{\mathcal{C}_{\rho}(\mathcal{Y})}\leq ||f||_{\mathcal{C}_{\rho}(\mathfrak{X})}\). By item 1 we conclude that \(\pi^{*}\) is an isometric embedding.
\end{enumerate}
\end{proof}

\section{Poisson boundary of a time dependent matrix-valued random walk}\label{appendix Poisson boundary of a time dependent matrix-valued random walk}
In this section we recall the Poisson boundary of a time dependent matrix-valued random walk from the papers \cite{connes1989hyperfinite}, \cite{Amenable}. The goal of this section is to develop some of its properties that are required here. In particular, we consider a notion of stationary spaces parallel to \(\mu\)-stationary systems for the classical \(\mu\)-Poisson-Furstenberg boundary. In Theorem \ref{Theorem Furstenberg Glasner} we prove an analogue of result of Furstenberg-Glasner \cite[Theorem 4.3]{FurstenbergGlasner}.\\
Throughout this section, \(G\) is a discrete countable group.

\subsection{Basic Definitions}\label{subsection Basic Definitions}
\begin{notation}
For a non-negative integer \(\ell\), we set \([\ell]:=\{0,\dots,\ell-1\}\).
\end{notation}
Throughout, we will have the following notation:\\
Let \(\boldsymbol{\ell}=(\ell_n)_{n\geq0}\) be a sequence of positive integers, denote \(\ell_{-1}=1\).
\begin{definition}\label{definition stochastic sequence}
An \emph{\(\boldsymbol{\ell}\)-stochastic sequence} is a sequence \(\boldsymbol{\sigma}=\big(\sigma^{(n)}\big)_{n\geq0}\) where \(\sigma^{(n)}\) is an \([\ell_{n-1}]\times [\ell_n]\) stochastic matrix of measures on \(G\).\\
In more details, \(\sigma^{(n)}=\big(\sigma_{i,j}^{(n)}\big)_{i\in[\ell_{n-1}],j\in[\ell_n]}\) where each \(\sigma^{(n)}_{i,j}\) is a non-negative measure on \(G\), and we have for each \(n,\:i\in [\ell_{n-1}]\) that \(\sum_{j\in[\ell_n]}\sigma_{i,j}^{(n)}(G)=1\).\\
We will always assume (as in \cite{Amenable}),
that for any \(j\in[\ell_{0}]\), the measure \(\sigma^{(0)}_{0,j}\) is supported on all of \(G\).
\end{definition}
\begin{notation}\label{notation sigma random walk}
Suppose \(\boldsymbol{\sigma}\) is an \(\boldsymbol{\ell}\)-stochastic sequence.
\begin{itemize}
\item 
We denote \(V_n=[\ell_n]\times G\) for \(n\geq0\). This is a \(G\)-space with trivial action on the first coordinate.
\item 
We consider the random walk \((X_n)_{n\geq0 }\) (where \(X_n\) takes values in \(V_n\)) with transition probabilities (\(n\geq1\)):
\[\mathbb{P}_{V_{n-1}\mapsto V_n}\bigg((i,g)\to(j,h)\bigg)=\sigma_{i,j}^{(n)}(g^{-1} h)\] 
We get a probability space \((\Omega,\mathcal{F},\mathbb{P})\) where \(\Omega=\prod_{n\geq0} V_n\) is the space of paths of the radon walk and the probability measure \(\mathbb{P}\) on \(\Omega\) is given by the Markov measure with transition probabilities as above and initial distribution \(\sigma^{(0)}\).
\item 
We define the following random variables on \(\Omega\):
\begin{enumerate}
    \item
     \(Y_n\big((i_k,g_k)_k\big)=g_n\) (this is \(G\)-valued)
    \item
    \(I_n\big((i_k,g_k)_k\big)=i_n\) (this is \([\ell_n]\)-valued)
\end{enumerate}
In this notation, \(X_n=(I_n,Y_n)\)
\item
The \(G\)-action on \(V_n\) gives rise to a \(G\)-action on \(\Omega\) with:
\[g\cdot\big((i_n,g_n)\big)=\big((i_n,g\cdot g_n)\big)\:,\: I_n(g\omega)=I_n(\omega),\:Y_n(g\omega)=gY_n(\omega)\]
\item
Let \(\mathcal{F}_m=\sigma(X_k \big| k\geq m),\mathcal{G}_m=\sigma(X_k\big|k\leq m)\). These are \(\sigma\)-algebras on \(\Omega\).
\item 
The asymptotic algebra of the random walk: \(\mathcal{A}_\Omega=\bigcap_m \mathcal{F}_m\). 
\item
For each \(m\),\(\:j\in [\ell_{m}]\) consider the random walk with the same transition probabilities starting at \((j,g)\in V_m\).
The Markov measure will be denoted by \(\mathbb{P}^{(m)}_{(j,g)}\) and we can consider it as a measure on \((\Omega,\mathcal{F}_{m})\).\\
We define \(\mathbb{P}^{(m)}\) to be the column vector of probability measures \((\mathbb{P}^{(m)}_{j,e})_{j\in [\ell_{m}]}\).
\end{itemize}
\end{notation}
Note that in the notations above,
\begin{itemize}
    \item
    \(\mathbb{P}^{(-1)}=\mathbb{P}\).
    \item 
    \(\mathbb{P}_{(j,g)}^{(m)}=g\mathbb{P}_{(j,e)}\).
    \item \(X_n=(I_n,Y_n)\) has distribution given by matrix multiplication:
    \[
\mathbb{P}(I_n=i,Y_n=g)=
\big(\sigma^{(0)}\ast\dots\ast\sigma^{(n)}\big)_{i}(g)
\]

\end{itemize}

\begin{definition}
A bounded \(\boldsymbol{\sigma}\)-harmonic function is a sequence \(\boldsymbol{h}=(h_m)_{m\geq0}\) consisting of bounded complex valued functions \(h_{m}:V_m\to \mathbb{C}\) with \(\sup_m ||h_m||_{\ell^\infty(V_m)}<\infty\) and such that for all \(m\geq1\):
\[h_{m-1}\big((i,g)\big)=\sum_{(j,h)\in V_m} h_m\big((j,h)\big)\cdot \sigma_{i,j}^{(m)}(g^{-1}h)
\]
The space of all bounded \(\boldsymbol{\sigma}\)-harmonic functions will be denoted by \(\mathcal{H}_{\boldsymbol{\sigma}}^\infty(G)\). When equipped with the norm \(||\boldsymbol{h}||=\sup_m ||h_m||_{\ell^\infty(V_m)}\), \(\mathcal{H}_{\boldsymbol{\sigma}}^{\infty}(G)\) is a Banach space.
It has the following isometric left \(G\)-action 
\[g\cdot (h_m)_{m\geq0}=(g\cdot h_m)_{m\geq0}\:\:\:\text{ where:}\:\: (g\cdot f)(z)=f(g^{-1}z)\]
\end{definition}
We may rewrite this condition of harmonicity in the following ways, the first is:
\[\forall v\in V_{m-1}:\:\: h_{m-1}\big(v\big)=\sum_{u\in V_m} h_m\big(u\big)\cdot \mathbb{P}_{V_{m-1}\mapsto V_m}\bigg(v\mapsto u\bigg)\]
And the second is to think of \(h_m\) as a \(\ell_m\)-row vector and then in matrix notation:
\[h_{m-1}=h_m\ast\Check{\sigma}^{(m)}\]
Here, for a matrix \(\sigma=(\sigma_{i,j})\) we denote \(\Check{\sigma}(g)=(\sigma_{j,i}(g^{-1}))\).
\begin{definition}\label{definition stationary spaces}
Let \(\boldsymbol{\sigma}\) be an \(\boldsymbol{\ell}\)-stochastic sequence.
\begin{enumerate}
    \item 
    A \(\boldsymbol{\sigma}\)-stationary space (or \(\boldsymbol{\sigma}\)-system), \(\mathcal{X}=(X,\boldsymbol{\nu})\) is a measure space \(X\) together with a measurable \(G\)-action, and a sequence \(\boldsymbol{\nu}=(\nu^{(n)})_{n=-1}^{\infty}\) of \(\ell_n\)-column vectors \(\nu^{(n)}=(\nu^{(n)}_j)_{j\in [\ell_n]}\) such that any \(\nu^{(n)}_j\) is a probability measure on \(X\), and such that \(\sigma^{(n)}\ast \nu^{(n)}=\nu^{(n-1)}\) for any \(n\geq 0\).
    \item
    A \emph{factor} of \(\boldsymbol{\sigma}\)-stationary spaces \(\pi:\mathcal{X}=(X,\boldsymbol{\nu})\to\mathcal{Y}=(Y,\boldsymbol{m})\) is a measurable \(G\)-mapping \(\pi:X\to Y\) such that \(\pi_*(\nu^{(n)}_j)=m^{(n)}_j\). 
    We say that \(\mathcal{Y}\) is a factor of \(\mathcal{X}\), and that \(\mathcal{X}\) is an extension of \(\mathcal{Y}\).

\end{enumerate}
\end{definition}
Since \(\nu^{(-1)}=\sigma^{(0)}\ast \nu^{(0)}\) and \(\sigma^{(0)}_{j}\) are QI we see that \((X,\nu^{(-1)})\) is a QI \(G\)-space. We will consider \(X\) with this measure class as the underlying \(G\)-space for \(\mathcal{X}\) and denote \(L^{\infty}(\mathcal{X})=L^{\infty}(X,\nu^{(-1)})\).\\
From now on, \textbf{we will assume that the matrices \(\sigma^{(n)}\) do not have a zero column}. More formally, this assumption on \(\boldsymbol{\sigma}\) is:
For any \(n\) and \(j\in [\ell_{n}]\) there is \(i\in [\ell_{n-1}]\) with \(\sigma_{i,j}^{(n)}(G)>0\).\\
\(\)\\
Under the assumption, it is easy to see that for each for any \(n,\:j\in[\ell_{n}]\) we have \(\nu_{j}^{(n)}\ll \nu^{(-1)}\).

\begin{remark}\label{remark no zero columns on sigma}
This assumption does not change much- for any \(\boldsymbol{\sigma}\), by deleting the zero-columns (and their corresponding rows in the next matrix) we will obtain a new stochastic sequence \(\boldsymbol{\tau}\) (for a different \(\boldsymbol{\ell}\)). For any \(\boldsymbol{\sigma}\)-stationary space, considering only the relevant \(\nu_{j}^{(n)}\) yields a \(\boldsymbol{\tau}\)-stationary space. Their underlying \(G\)-space is the same.\\
Later we will consider the Poisson boundary (see Definition \ref{Definition Poisson boundary of sigma}), it will be clear that this operation sends the Poisson boundary of \(\boldsymbol{\sigma}\) to the one of \(\boldsymbol{\tau}\).
\end{remark}
We can now define a version of Furstenberg-Poisson transform:
\begin{prop}\label{Proposition Poisson integral for stationary actions}
Let \(\mathcal{X}=(X,\boldsymbol{\nu})\) be a \(\boldsymbol{\sigma}\)-stationary space, for any \(f\in L^\infty(\mathcal{X})\), define \(\mathcal{P}_{X}(f)\) by:
\[\big(\mathcal{P}_{X}(f)\big)_n(i,g):=\intop_{X} f\:d\big(g\cdot \nu^{(n)}_i\big)=\intop_X f(gx) \: d\nu_i^{(n)}(x)\]
Then \(\mathcal{P}_X(f)\in\mathcal{H}_{\boldsymbol{\sigma}}^\infty(G)\). Moreover,
\begin{enumerate}
    \item 
    \(\mathcal{P}_X:L^\infty(\mathcal{X})\to \mathcal{H}_{\boldsymbol{\sigma}}^\infty(G)\) is a \(G\)-equivariant linear contraction.
    \item
    If \(\pi:X\to Y\) is a factor of \(\boldsymbol{\sigma}\)-stationary spaces then \(\mathcal{P}_{X}\circ\pi^{*}=\mathcal{P}_{Y}\).
\end{enumerate}

\end{prop}
\begin{proof}
The equalities
\begin{multline*}
\sum_{(j,h)\in V_m} \sigma_{i,j}^{(m)}(g^{-1}h) \big(P_X(f)\big)_m(j,h)=\intop_X f \: d\Big(\sum_{(j,h)\in V_m} \sigma_{i,j}^{(m)}(g^{-1}h)\: h\nu_{j}^{(m)}\Big)=\\
\intop_X f\: d\Big(g\cdot \sum_{(j,s)\in V_m} \sigma_{i,j}(s)\: s\nu_j^{(m)}\Big)=\\
\intop_X f\: d\Big(g\cdot ((\sigma^{(m)}\ast\nu^{(m)})_i\Big)=
\intop_X f \: d\Big(g\cdot \nu_{i}^{(m-1)}\Big)=
\big(\mathcal{P}_X(f)\big)_{m-1}(i,g)
\end{multline*}
shows that \(\mathcal{P}_X(f)\in\mathcal{H}_{\boldsymbol{\sigma}}^\infty(G)\).\\
For the moreover part: it is obvious that \(||\mathcal{P}_X||\leq1\). For the \(G\)-equivariance:
\[\big(g(\mathcal{P}_X(f))\big)_n (i,h)= (\mathcal{P}_X(f))_n(i,g^{-1}h)=\intop_{X} f\: d(g^{-1}h\cdot \nu_i^{(n)})=\intop_X (g\cdot f)\: d(h\cdot \nu_i^{(n)})=(\mathcal{P}_X(g\cdot f))_n(i,h)\]
The naturality of \(\mathcal{P}\) is trivial.
\end{proof}
We call the operator \(\mathcal{P}_{X}\) the Poisson transform.\\
Our next goal is to define the Poisson boundary of \(\boldsymbol{\sigma}\) using the random walk described above and to show that the Poisson transform attached to the Poisson boundary is an isomorphism.

\begin{prop}\label{Proposition stationarity of Poisson boundary}
We have equality of vectors of measures on \((\Omega,\mathcal{F}_{m})\)
\[\sigma^{(m)}\ast \mathbb{P}^{(m)}=\mathbb{P}^{(m-1)}\]
In particular, \(\Big(\Omega,\mathcal{A}_\Omega,\big(\mathbb{P}^{(n)}\big)_{n}\Big)\) is a \(\boldsymbol{\sigma}\)-stationary space.
\end{prop}
\begin{proof}
We only need to show the equality for the generating algebra consisting of cylinder sets of the form \(A=\big\{X_{m+r}=(i_{m+r},g_{m+r}) \: , (0\leq r\leq \ell)\big\}\). And indeed, for \(i\in[\ell_{m-1}]\) 
\begin{multline*}
    \mathbb{P}_i^{(m-1)}(A)=\sigma_{i,i_m}^{(m)}(g_m)\cdot\sigma_{i_m,i_{m+1}}^{(m+1)}(g_m^{-1}g_{m+1})\dots\sigma_{i_{m+\ell-1},i_{m+\ell}}^{(m+\ell)}(g_{m+\ell-1}^{-1}g_{m+\ell})=\\
    =\sum_{g\in G,\: j\in[\ell_m]} \sigma_{i,j}^{(m)}(g)\cdot\delta_{(j,g),(i_{m},g_{m})}\cdot\prod_{r=1}^{\ell}\sigma_{i_m+r-1,i_{m+r}}^{(m+r)}(g_{m+r-1}^{-1}g_{m+r})=\\
    \sum_{g\in G,\: j\in[\ell_m]} \sigma_{i,j}^{(m)}(g)\cdot\delta_{(j,e),(i_{m},g^{-1}g_m)}\cdot\prod_{r=1}^{\ell}\sigma_{i_{m+r-1},i_{m+r}}^{(m+r)}\big((g^{-1}g_{m+r-1})^{-1}(g^{-1}g_{m+r})\big)=\\
    \sum_{g\in G,\: j\in[\ell_m]} \sigma_{i,j}^{(m)}(g)\cdot\mathbb{P}_j^{(m)}\big(\{X_{m+r}=(i_{m+r},g^{-1}\cdot g_{m+r})\:,\: (0\leq r\leq \ell)\}\big)=\\
    \sum_{g\in G,\: j\in[\ell_m]} \sigma^{(m)}_{i,j}(g)\cdot \mathbb{P}_{j}^{(m)}\big(g^{-1}A)=
    \sum_{g\in G,\: j\in[\ell_m]} \sigma^{(m)}_{i,j}(g)\cdot (g\mathbb{P}_{j}^{(m)})\big(A\big)=\\
    \sum_{j\in[\ell_m]} \big(\sigma_{i,j}^{(m)}\ast \mathbb{P}_{j}^{(m)}\big)(A)=
    \big(\sigma^{(m)}\ast\mathbb{P}^{(m)}\big)_i(A)
\end{multline*}
\end{proof}

\begin{definition}\label{Definition Poisson boundary of sigma}
A Poisson boundary of \(\boldsymbol{\sigma}\), denoted by \(\mathcal{B}(G,\boldsymbol{\sigma})=(B,\Sigma,(\nu_{j}^{(m)})_{m\geq -1})\) is a \(\boldsymbol{\sigma}\)-stationary space with \((B,\Sigma)\) a standard Borel space together with a \(G\)-equivariant isomorphism \(L^\infty(B,\Sigma,\nu^{(-1)})\cong L^\infty(\Omega,\mathcal{A}_\Omega,\mathbb{P})\) for which \(\nu_j^{(m)}\) corresponds to \(\mathbb{P}_j^{(m)}\).
\end{definition}
Using the separability of \(L^{\infty}(\Omega,\mathcal{A}_{\Omega},\mathbb{P})\) with respect to the \(w^{*}\)-topology and \cite[Theorem 2.1]{RAMSAY1971253} we conclude:
\begin{prop}\label{prop existence and uniqueness poisson boundary}
\(\)
\begin{enumerate}
    \item 
    A Poisson boundary \((B,\Sigma,\boldsymbol{\nu})\) of \(\boldsymbol{\sigma}\) exists. 
    \item
    A Poisson boundary is unique up to a unique isomorphism that commutes with the isomorphism \(L^\infty(B,\Sigma,\nu^{(-1)})\cong L^\infty(\Omega,\mathcal{A}_\Omega,\mathbb{P})\).
    \item
    There is a factor map \(\tau:\Omega\to B\) inducing the isomorphism \(L^\infty(B,\Sigma,\nu^{(-1)})\cong L^\infty(\Omega,\mathcal{A}_\Omega,\mathbb{P})\). 

\end{enumerate}
\end{prop}


\begin{remark}
It follows from Proposition \ref{Proposition Poisson transform for the boundary} below that the isomorphism (as \(\boldsymbol{\sigma}\)-systems) is actually unique and the commuting assumption in the uniqueness of Poisson boundary is redundant.
\end{remark}

\begin{prop}\label{proposition random walk martingale}
Let \(\boldsymbol{h}\in\mathcal{H}_{\boldsymbol{\sigma}}^\infty(G)\) 
and consider 
\[M_n(\omega)=h_{n}(X_n(\omega))=\big(h_{n}\big)_{I_n(\omega)}(Y_n(\omega))\]
Then \(M_n\) is a bounded martingale for \((\Omega,\{{\mathcal{G}_n}\},\mathbb{P})\).

\end{prop}

\begin{proof}
Indeed, \(M_n\) is \(\mathcal{G}_n\)-measurable and bounded by \(||\boldsymbol{h}||\).\\ We need to show that for \(A\in\mathcal{G}_n\) we have \(\intop_{\Omega} M_n\cdot\boldsymbol{1}_{A}d\mathbb{P}=\intop_{\Omega} M_{n+1}\cdot\boldsymbol{1}_{A}d\mathbb{P}\). It is enough to show this for 
\(A=\big\{X_{r}=v_r\:,\: (0\leq r\leq n)\big\} \) where \(v_r=(i_r,g_r)\in V_{r}\). Indeed:
\begin{multline*}
    \intop_{\Omega} M_{n+1}\cdot \boldsymbol{1}_A d\mathbb{P}= \intop_{\Omega} h_{n+1}(X_{n+1}(\omega))\sum_{v_{n+1}\in V_{n+1}} \boldsymbol{1}_{A\cap\{X_{n+1}=v_{n+1}\} }(\omega) d\mathbb{P}(\omega)=\\
    \sum_{v_{n+1}\in V_{n+1}} h_{n+1}(v_{n+1}) \cdot\mathbb{P}\bigg(\big\{X_{r}=v_r \: ,\: (0\leq r\leq n+1)\big\}\bigg)=\\
    \sum_{v_{n+1}\in V_{n+1}} h_{n+1}(v_{n+1}) \mathbb{P}_{V_n\mapsto V_{n+1}}(v_n\mapsto v_{n+1})\cdot \mathbb{P}\big(A\big)=h_n(v_n)\mathbb{P}(A)=\intop_{\Omega} M_n\cdot\boldsymbol{1}_{A}d\mathbb{P}
\end{multline*}
\end{proof}

\begin{lemma}\label{Lemma conditioning random walk measure}
Let \(v=(g,j)\in V_n\). Let \(B\in\mathcal{G}_{n-1}\) and let \(A=B\cap \{X_n=v\}\). Then for any \(f\in L^\infty(\Omega,\mathcal{F}_{n+1})\) we have:
\[\mathbb{P}(A)\cdot \intop_{\Omega} f \:d(g\mathbb{P}_j^{(n)})=\intop_A f \: d\mathbb{P}\]
\end{lemma}
\begin{proof}
It is enough to verify the identity in the case where \(f\) is the indicator function of the cylinder \(C=\{X_{n+1}=v_{n+1},\dots,X_{n+\ell}=v_{n+\ell}\}\) for \(v_{n+r}=(i_{n+r},g_{n+r})\in V_{n+r}\:,\: (r=1,\dots,\ell)\). Moreover, we may assume that \(B=\{X_0=v_0,\dots,X_{n-1}=v_{n-1}\}\) for \(v_{t}=(i_t,g_t)\in V_t\:,\:(t=0,\dots n-1)\).
In this case:
\begin{multline*}
    \mathbb{P}(A)\cdot \intop_{\Omega} f \:d(g\mathbb{P}_j^{(n)})=\mathbb{P}(A)\cdot \mathbb{P}_j^{(n)}(g^{-1}C)=\\
    \mathbb{P}\big(\{X_0=v_0,\dots,X_{n-1}=v_{n-1},X_{n}=v\}\big)\cdot\mathbb{P}_{j}^{(n)}\big(\{X_{n+r}=(i_r,g^{-1}g_r)\:\:(1\leq r\leq\ell\}\big)=\\
    =\Big(\sigma^{(0)}_{i_0}(g_0)\cdot\sigma^{(1)}_{i_0,i_1}(g_0^{-1}g_1)\dots\cdot\sigma_{i_{n-1},j}^{(n)}(g_{n-1}^{-1}g)\Big)\cdot\Big(\sigma_{j,i_{n+1}}^{(n+1)}(g^{-1}g_{n+1})\cdot\sigma_{i_{n+1},i_{n+2}}^{(n+2)}(g_{n+1}^{-1}g_{n+2})\dots \sigma_{i_{n+\ell-1},i_{n+\ell}}^{(n+\ell)}(g_{n+\ell-1}^{-1}g_{n+\ell})\Big)=\\
    \mathbb{P}\bigg(\{X_0=v_0,\dots,X_{n-1}=v_{n-1},X_{n}=v,X_{n+1}=v_{n+1},\dots,X_{n+\ell}=v_{n+\ell}\}\bigg)=
    \mathbb{P}(A\cap C)=\intop_{A} f\:d\mathbb{P}
\end{multline*}
\end{proof}

\begin{prop}\label{Proposition Poisson transform for the boundary}
The Poisson transform defines an isometric isomorphism:
\[\mathcal{P}_{\Omega}:\: L^\infty(\Omega,\mathcal{A}_\Omega,\mathbb{P})\stackrel{\sim}{\longrightarrow}\mathcal{H}_{\boldsymbol{\sigma}}^\infty (G)\]
Whose inverse satisfies:
\[\mathcal{P}_{\Omega}^{-1}(\boldsymbol{h})=\lim_{n\to\infty} h_n(X_n)\quad\text{a.e.}\]
\end{prop}
\begin{proof}
Let \(\boldsymbol{h}\in\mathcal{H}_{\boldsymbol{\sigma}}^{\infty}(G)\), considering the martingale \(M_n=h_n(X_n)\) of Proposition \ref{proposition random walk martingale}, by the martingale convergence theorem (see \cite[Chapter 11]{williams}), we conclude that \(M_{n}\) converges a.e. and in \(L^{1}\) to a function \(f\in L^{1}(\Omega,\mathbb{P})\).
Note that \(f\in L^{\infty}(\Omega,\mathcal{F},\mathbb{P})\) and \(||f||_{\infty}\leq ||\boldsymbol{h}||\).\\
For all \(m\) we have that \(M_n\) is \(\mathcal{F}_m\)-measurable for \(n\geq m\) and thus, \(f\) is \(\mathcal{F}_m\)-measurable for all \(m\). This implies \(\mathcal{M}({\boldsymbol{h}}):= f\in L^\infty(\Omega,\mathcal{A}_\Omega)\).\\
To prove the proposition we will show that the contractions \(\mathcal{M},\mathcal{P}_{\Omega}\) are inverse to each other.\\
Let \(\boldsymbol{h}\in\mathcal{H}_{\boldsymbol{\sigma}}^{\infty}(G)\) and consider \(f=\mathcal{M}(\boldsymbol{h})\). We wish to show \(\mathcal{P}_{\Omega}(f)=\boldsymbol{h}\). Indeed, for any \(n\geq0,j\in[\ell_{n}],g\in G\) consider \(A=\{X_{n}=(j,g)\}\). Note that since \(\sigma^{(0)}\) is positive everywhere, we have \(\mathbb{P}(A)>0\). Using Lemma \ref{Lemma conditioning random walk measure} and the martingale property we get:
\[h_{n}(j,g)=\frac{1}{\mathbb{P}(A)}\intop_{A} h_n(X_n) \:d\mathbb{P}=\frac{1}{\mathbb{P}(A)}\lim_{m\to\infty} \intop_{A} h_m(X_m)\:d\mathbb{P}=\frac{1}{\mathbb{P}(A)} \intop_{A} f\:d\mathbb{P}=\intop_{\Omega} f\:d(g\mathbb{P}_j^{(n)})=\mathcal{P}_{\Omega}(f)_{n}(j,g)\]
For the other composition, let \(f\in L^\infty(\Omega,\mathcal{A}_\Omega)\) and consider \(\boldsymbol{h}=\mathcal{P}_{\Omega}(f)\). We need to show \(f=\mathcal{M}(\boldsymbol{h})\), that is, that \(M_n=h_{n}(X_n)\) converges a.e. to \(f\). \\
Using Levy's upward convergence for martingales (see \cite[Chapter 14]{williams}) it is enough to show that \(M_n=\mathbb{E}_{\mathbb{P}}(f|\mathcal{G}_n)\). As \(M_n\) is \(\mathcal{G}_n\)-measurable we only need to check \(\intop_\Omega f\cdot 1_A\:d\mathbb{P}=\intop_{\Omega} M_n\cdot 1_A\:d\mathbb{P}\) for all \(A\in\mathcal{G}_n\). By additivity we may assume that \(X_n|_{A}\equiv(j,g)\) and then by Lemma \ref{Lemma conditioning random walk measure} we have:
\[\intop_\Omega f\cdot 1_A\: d\mathbb{P}=\intop_{A} f\:d\mathbb{P}=\mathbb{P}(A)\intop_{\Omega} f\: d\big(g\mathbb{P}^{(n)}_j\big)=\mathbb{P}(A) h_n(j,g)=\intop_{\Omega} h_n(X_n(\omega))\cdot 1_A d\mathbb{P}=\intop_{\Omega} M_n\cdot 1_A\:d\mathbb{P}\]
concluding the result.
\end{proof}

\begin{corollary}\label{Corollary RN of Poisson boundary}
The subspace spanned by \(\{\frac{dg\mathbb{P}_{i}^{(n)}}{d\mathbb{P}} \::\: n\geq0 ,\: i\in [\ell_n],\: g\in G \}\) is dense in \(L^{1}(\Omega,\mathcal{A}_\Omega,\mathbb{P})\)
\end{corollary}
\begin{proof}
Otherwise, by Hahn-Banach theorem, there would be a non zero \(f\in L^1(\Omega,\mathcal{A}_\Omega,\mathbb{P})^*\cong L^\infty(\Omega,\mathcal{A}_\Omega,\mathbb{P})\) that vanishes on all those functions, meaning that 
\[\intop_{\Omega} f\cdot \frac{dg\mathbb{P}^{(n)}_i}{d\mathbb{P}} d\mathbb{P}=0\]
But this implies that \(\mathcal{P}_{\Omega}(f)=0\in\mathcal{H}_{\boldsymbol{\sigma}}^{\infty}(G)\). By Proposition \ref{Proposition Poisson transform for the boundary} we conclude that \(f=0\) which is a contradiction.
\end{proof}


\begin{definition}\label{definition measure preserving factors sigma-systems}
We will say that a factor between \(\boldsymbol{\sigma}\)-stationary spaces \(\pi:(X,\boldsymbol{\nu})\to(Y,\boldsymbol{m})\) is \emph{measure preserving extension} if for any \(n\:,\:i\in[\ell_n]\:,\:g\in G\) we have that
\[\frac{d\:g\nu_{i}^{(n)}}{d\nu^{(-1)}}=\frac{d\:gm_{i}^{(n)}}{dm^{(-1)}}\circ\pi\]
\end{definition}
In particular, \((X,\nu^{(-1)})\to(Y,m^{(-1)})\) is measure preserving in this case.

\begin{lemma}\label{lemma RN factor of sigma-system}
Let \(\mathcal{X}\) be a \(\boldsymbol{\sigma}\)-stationary space. Then there is a \(\boldsymbol{\sigma}\)-stationary system \(\mathcal{X}_{RN}\) along with a measure preserving extension \(\pi:\mathcal{X}\to\mathcal{X}_{RN}\) with the following property: for any measure preserving extension of \(\boldsymbol{\sigma}\)-stationary spaces \(p: \mathcal{X}\to\mathcal{Y}\), we have that \(\pi\) factors through \(p\), that is, there is a measure preserving extension \(q:\mathcal{Y}\to \mathcal{X}_{RN}\) with \(\pi=q\circ p\).
\end{lemma}
\begin{proof}
Consider the \(\sigma\)-algebra \(\mathcal{E}\) generated by the (countably many) functions \(\frac{dg\nu_{i}^{(n)}}{d\nu^{(-1)}}\) and take any standard Borel space \((X_{RN},m^{(-1)})\) with \(L^{\infty}(X_{RN})\cong L^{\infty}(X,\mathcal{E})\). Then \(X_{RN}\) posses a natural Borel probability measures \(m_{i}^{(n)}\) and a measurable \(G\)-action that makes it a \(\boldsymbol{\sigma}\)-stationary space \(\mathcal{X}_{RN}\). As \(X_{RN}\) is a standard Borel space, using \cite[Theorem 2.1]{RAMSAY1971253} we get a measure preserving factor \(X\to X_{RN}\) satisfying the universal property.
\end{proof}
By abstract nonsense, the pair \((\mathcal{X}_{RN},\pi)\) is unique up to a unique isomorphism. We call \(\mathcal{X}_{RN}\) the Radon-Nikodym factor of the \(\boldsymbol{\sigma}\)-stationary space \(\mathcal{X}\). 

\begin{prop}\label{prop RN factor of Poisson boundary is itself}
The Poisson boundary \(\mathcal{B}(G,\boldsymbol{\sigma})\) is its own Radon-Nikodym factor.
\end{prop}
\begin{proof}
Clear from Corollary \ref{Corollary RN of Poisson boundary}.
\end{proof}

\subsection{Conditional measures and a structure theorem}
Throughout this subsection, \(\boldsymbol{\sigma}\) denotes an \(\boldsymbol{\ell}\)-stochastic sequence with no zero columns.
\begin{definition}
A topological \(\boldsymbol{\sigma}\)-system is a \(\boldsymbol{\sigma}\)-stationary space \((X,\Sigma,\boldsymbol{\nu})\) such that \(X\) is a compact metrizable space, \(\Sigma\) is the Borel \(\sigma\)-algebra on \(X\) and the \(G\)-action is continuous.
\end{definition}

We use the following well known lemma from functional analysis:
\begin{lemma}\label{lemma functional analysis limit of functionals}
Let \(E\) be a Banach space and let \(\lambda_n\in E^*\) be a sequence of continuous linear functionals satisfying \(\sup ||\lambda_n||_{E^*}\leq1\). Then the set:
\[L:=\{v\in E\big| \exists \lim_{n\to\infty} \lambda_n(v)\}\]
is a closed linear subspace, and the function \(\lambda: L\to\mathbb{C}\) defined by \(\lambda(v)=\lim_{n\to\infty} \lambda_n(v)\) is a bounded functional with \(||\lambda||_{L^{*}}\leq 1\).
\end{lemma}

The following is the existence of stationary measures in the \(\boldsymbol{\sigma}\)-stationary setting, generalizing the construction in the classical Poisson boundary (see e.g. \cite[section 2, III-VII]{BaderShalom}).\\
We recall from Notation \ref{notation sigma random walk} that \((I_{n},Y_n)\) is the \(n\)-th step of the \(\boldsymbol{\sigma}\)-random walk -- these are random variables on \((\Omega,\mathcal{F},\mathbb{P})\). Throughout we let \((B,\boldsymbol{m})=\mathcal{B}(G,\boldsymbol{\sigma})\) be the Poisson boundary and let \(\tau:\Omega\to B\) be the factor map.

\begin{prop}\label{Proposition conditional measures for stationary spaces}
Let \(\mathcal{X}=(X,\boldsymbol{\nu})\) be a topological \(\boldsymbol{\sigma}\)-stationary space, let \(M(X)\) denote the space of Borel probability measures on \(X\) with the \(w^*\) topology.\\
Then there is a unique (up to a.e. equality) measurable \(G\)-mapping \(\Phi_{X}: B\to M(X)\) with the following integral factorization for any \(n\geq-1,\: j\in [\ell_n]\):
\[\nu^{(n)}_{j}=\intop_{B} \Phi_{X}\: d m^{(n)}_{j}\] 
Moreover, we have the following limit in \(M(X)\) for \(\mathbb{P}\)-a.e. \(\omega\in\Omega\)
\[\lim_{n\to\infty} Y_n(\omega)\cdot \nu^{(n)}_{I_n(\omega)}=\Phi_X\circ\tau(\omega)\]
\end{prop}
\begin{proof}
\underline{Uniqueness of \(\Phi_{X}\)}: suppose \(\Phi_1,\Phi_2\) both satisfy the integral factorization. Since \(M(X)\) is separable with the \(w^{*}\) topology it is enough to show that for any \(f\in C(X)\) the equality a.e. of \(\Phi_2(b)(f)=\Phi_1(b)(f)\). And indeed, we know that for any \(n\), \(j\in[ \ell_{n}]\) and \(g\in G\) we have for \(k=1,2\):
\begin{multline*}
\mathcal{P}_{X}(f)_{n}(j,g)=\intop_{X} f\:\: d (g\nu_{j}^{(n)})=\intop_{X} f\circ g \:\:d \nu_{j}^{(n)}=\intop_{B}\bigg(\intop_{X} f\circ g \:\:d\Phi_{k}(b)\bigg)\: dm^{(n)}_{j}(b)=\\
\intop_{B}\bigg(\intop_{X} f\:\: dg\Phi_{k}(b)\bigg)\: dm^{(n)}_{j}(b)=
\intop_{B}\bigg(\intop_{X} f\:\: d\Phi_{k}(gb)\bigg)\: dm^{(n)}_{j}(b)=\\
\intop_{B}\bigg(\intop_{X} f\:\: d\Phi_{k}\bigg)\: d\big(gm^{(n)}_{j}\big)=
\mathcal{P}_{B}(\langle f,\Phi_{k} \rangle)_{n}(j,g)
\end{multline*}
Since the Poisson transform is an isomorphism for the Poisson boundary (Proposition \ref{Proposition Poisson transform for the boundary}) we conclude \(\langle f,\Phi_{1} \rangle=\langle f,\Phi_{2} \rangle\) a.e. as claimed.\\
\underline{Existence of \(\Phi_X\)}:
the algebra \(A=C(X)\) is separable. Take \(A_0\subset A\) a countable \(G\)-invariant \(\mathbb{Q}[i]\)-algebra that is dense in \(A\). 
For each \(f\in A_0\), we consider \(\mathcal{P}_X(f)\in\mathcal{H}_{\boldsymbol{\sigma}}^\infty (G)\) and using \(L^{\infty}(B)\cong \mathcal{H}_{\boldsymbol{\sigma}}^{\infty}(G)\) we choose a measurable function \(\varphi_{f}\) on \(B\) bounded by \(||f||\) such that 
\(\mathcal{P}_{X}(f)=\mathcal{P}_{B}(\varphi_{f})\).\\
Given \(b\in B\) consider the map \(\nu_b: A_0\to \mathbb{C}\) given by \(\nu_b(f)=\varphi_{f}(b)\). Since \(A_0\) is countable, we get there is a co-null subset \(D\subset B\) such that for any \(b\in D\) we have \(\nu_{b}\) is \(\mathbb{Q}[i]\)-linear, positive and \(\nu_{b}(1)=1\).\\
These properties insures that for each \(b\in D\), \(\nu_{b}\) extends to a state of \(A=C(X)\). Thus we may regard \(\nu_{b}\) as an element \(\nu_{b}\in M(X)\). We define \(\Phi_X(b)=\nu_{b}\) for \(b\in D\) (and for the null set \(B\setminus D\) take \(\Phi_X(b)=\nu^{(-1)}\) for example).\\
We now verify that \(\Phi_{X}\) satisfies the required properties. First, let us show \(\Phi_{X}\) is measurable. Indeed, for any \(f\in A_0\) and \(b\in D\) we have \(\langle f,\Phi_X(b)\rangle=\nu_{b}(f)=\varphi_{f}(b)\) which is a measurable function on \(D\). Since \(A_0\) generates the topology and thus the \(\sigma\)-algebra we conclude \(\Phi_X\) is measurable. Next, note that for any \(f\in A_0\) we have a.e. \(\langle f,\Phi_X(b)\rangle=\mathcal{P}_{B}^{-1}\circ\mathcal{P}_{X}(f)(b)\). This yields the \(G\)-equivariance of \(\Phi_{X}\) since the Poisson transforms are \(G\)-equivariant.\\
Thus we constructed a measurable \(G\)-equivariant map \(\Phi_X:B\to M(X)\) such that for any \(f\in A_0\):
\[\intop_{X} f\: d\nu_{j}^{(n)}=\mathcal{P}_{X}(f)_{n}(j,e)=\mathcal{P}_{B}(\langle f,\Phi_X\rangle)_n(j,e)=\intop_{B} \langle f,\Phi_X(b)\rangle\: dm_{j}^{(n)}(b)=\intop_{X} f d\Big(\intop_{B} \Phi_X(b)\: dm_{j}^{(n)}(b)\Big)\]
Thus this holds for any \(f\in A=C(X)\) which yields the integral factorization:
\[\nu^{(n)}_{j}=\intop_{B} \Phi_{X}\: d m^{(n)}_{j}\]
For the moreover part, we use Proposition \ref{Proposition Poisson transform for the boundary} to conclude that for any \(f\in A_0\) we have a co-null set \(\Omega_{f}\subset\Omega\) so that for \(\omega\in \Omega_f\):
\begin{multline*}
\langle f,\Phi_X\circ \tau(\omega)\rangle=\mathcal{P}_{B}^{-1}(\mathcal{P}_{X}(f))(\tau(\omega))=\mathcal{P}_{\Omega}^{-1}(\mathcal{P}_{X}(f))(\omega)=\\
\lim_{n\to\infty} \mathcal{P}_{X}(f)_{n}\big(I_{n}(\omega),Y_n(\omega)\big)=\lim_{n\to\infty} \intop_{X} f\: d\Big(Y_n(\omega)\nu_{I_{n}(\omega)}^{(n)}\Big)
\end{multline*}
The set \(\Omega^{\prime}=\bigcap_{f\in A_0} \Omega_f\) is co-null. Using Lemma \ref{lemma functional analysis limit of functionals} we see that for \(\omega\in \Omega^{\prime}\) we have the following convergence in \(w^{*}\)-topology: \(\lim_{n\to\infty } Y_{n}(\omega)\nu^{(n)}_{I_{n}(\omega)}=\Phi_X(\tau(\omega))\), concluding the result.
\end{proof}

\begin{definition}
Let \(\mathcal{X}=(X,\boldsymbol{\nu})\) be a topological \(\boldsymbol{\sigma}\)-stationary space. 
The mapping \(\Phi_{X}\) is called the conditional measures (sometimes we will write \(\Phi_{\mathcal{X}}\) when we want to be specific).
\end{definition}
Our next goal is to show functoriality of the conditional measures. First, we prove the following lemma:
\begin{lemma}\label{lemma continous model for morphisms}
Let \(\pi:\mathcal{X}\to\mathcal{Y}\) be a factor between 
two topological \(\boldsymbol{\sigma}\)-systems. Then, there is a topological \(\boldsymbol{\sigma}\)-systems \(\mathcal{Z}\) and factors \(\tau: \mathcal{Z}\to \mathcal{X}\) and \(\theta:\mathcal{Z}\to\mathcal{Y}\) such that \(\tau,\theta\) are continuous and \(\theta=\pi\circ\tau\). Moreover one can take \(\tau\) to be an isomorphism (up to null sets).
\end{lemma}
\begin{proof}
Consider the \(C^{*}\)-algebra \(A=\overline{C(X)+\pi^{*}C(Y)}\subseteq L^\infty(X)\) and let \(Z\) be the Gelfand dual to \(A\). Since \(A\) is \(G\)-invariant and separable we conclude that \(Z\) has the structure of a topological \(\boldsymbol{\sigma}\)-system. Since \(C(X),\pi^{*}C(Y)\subseteq C(Z)\) we get the continuous factors \(\tau,\theta\) as needed. As \(C(Z)\subseteq L^\infty(X)\) we conclude \(\tau\) is an isomorphism.
\end{proof}

\begin{prop}\label{proposition functoriality of conditional measures}
Let \(\pi:(X,\boldsymbol{\nu})\to (Y,\boldsymbol{\eta})\) be a factor (of \(\boldsymbol{\sigma}\)-systems) between topological \(\boldsymbol{\sigma}\)-systems, then (a.e.): \[\pi_{*}\Phi_{X}=\Phi_{Y}\]
\end{prop}
\begin{proof}
We first consider the case where \(\pi\) is continuous. By the uniqueness in Proposition \ref{Proposition conditional measures for stationary spaces} we only need to exhibit the integral decomposition for \(\pi_{*}\Phi_{X}\). Note that \(\pi_{*}: M(X)\to M(Y)\) is the restriction of the \(w^{*}\)-continuous linear operator \((\pi^{*})^{*}:C(X)^{*}\to C(Y)^{*}\) and thus preserves integration. Thus:
\[\eta_{j}^{(n)}=\pi_{*}\nu_{j}^{(n)}=\pi_{*}\Big(\intop_{B} \Phi_{X} \:dm_{j}^{(n)}\Big)=\intop_{B} \pi_{*}\Phi_{X}\:dm_{j}^{(n)}\]
In the general case, take \(\mathcal{Z}=(Z,\lambda)\) as in Lemma \ref{lemma continous model for morphisms}, then since \(\tau,\theta\) are continuous, we have:
\[\pi_*\Phi_{X}=\pi_*(\tau_*\Phi_{Z})=(\pi\circ\tau)_{*}\Phi_{Z}=\theta_{*}\Phi_{Z}=\Phi_{Y}\]
\end{proof}
\begin{definition}\label{definition regular sigma-system}
We say that a \(\boldsymbol{\sigma}\)-system \(\mathcal{X}=(X,\Sigma,\boldsymbol{\nu})\) is \emph{regular} if the underlying Borel space \((X,\Sigma)\) is a standard Borel space.\\
We will denote by \(M(X,\Sigma)\) the collection of all probability measures on the Borel space \((X,\Sigma)\).
\end{definition}
\begin{remark}
Above we identified two mappings \(\pi_1,\pi_2:X\to Y\) if they are equal \(\nu^{(-1)}\)-a.e. Note that this does not imply that the mappings \((\pi_1)_{*},(\pi_2)_{*}: M(X,\Sigma_{X})\to M(Y,\Sigma_{Y})\) are equal.\\
However, if \((X,\Sigma_{X})\) is standard and we have given a measurable mapping \(T\to M(X,\Sigma_{X}),\:t\mapsto\nu_{t}\), where \(T\) posses a measure \(m\) with \(\intop_{T} \nu_{t}\:dm\ll \nu^{(-1)}\), we have \((\pi_1)_{*}\nu_{t}=(\pi_2)_{*}\nu_{t}\) for \(m\)-a.e. \(t\in T\).
\end{remark}
The following lemma is clear by taking a topological model:
\begin{lemma}
If a \(\boldsymbol{\sigma}\)-stationary space is regular then it is isomorphic to a topological \(\boldsymbol{\sigma}\)-system.
\end{lemma}

\begin{example}
For any \(\boldsymbol{\sigma}\)-stationary space \(\mathcal{X}\), the proof of Lemma \ref{lemma RN factor of sigma-system} shows that the Radon-Nikodym factor \(\mathcal{X}_{RN}\) can be given as a regular \(\boldsymbol{\sigma}\)-stationary space.
\end{example}

We now summarize the previous propositions and lemmas in a theorem regarding conditional measures:
\begin{theorem}\label{Theorem conditional measures for regular systems}
Let \(\mathcal{X}=(X,\boldsymbol{\nu})\) be a regular \(\boldsymbol{\sigma}\)-stationary space.\\
Then there is a unique (up to a.e. equality) measurable \(G\)-mapping \(\Phi_{X}: B\to M(X,\Sigma_{X})\) from the Poisson boundary with the following integral factorization for any \(n\geq-1,\:j\in [\ell_n],\:g\in G\):
\[g\nu^{(n)}_{j}=\intop_{B} \Phi_{X}\: d gm^{(n)}_{j}\] 
Moreover, for every \(f\in L^{\infty}(X)\) we have for \(\mathbb{P}\)-a.e. \(\omega\in\Omega\)
\[\lim_{n\to\infty} \langle Y_n(\omega)\cdot \nu^{(n)}_{I_n(\omega)},f\rangle=\langle\Phi_X(\tau(\omega)),f\rangle\]
The \(\Phi\) is natural: given a factor \(\pi:\mathcal{X}\to \mathcal{Y}\) we have that \(\pi_{*}\Phi_{X}=\Phi_{Y}\).
\end{theorem}

\begin{remark}
In the theory of classical Poisson boundary (e.g. \cite{FurstenbergGlasner},\cite{BaderShalom}), the measure \(\Phi_{X}(\tau(\omega))\) is commonly denoted by \(\nu_{\omega}\).
\end{remark}

\begin{definition}
Let \(\mathcal{X}=(X,\boldsymbol{\nu})\) be a regular \(\boldsymbol{\sigma}\)-stationary space. 
The mapping \(\Phi_{X}\) is called the conditional measures.\\ 
We define a \(\boldsymbol{\sigma}\)-system \(\Pi(X):=(M(X,\Sigma_{X}),(\Phi_X)_* \boldsymbol{m})\).
\end{definition}

\begin{example}\label{Example conditional measures of Poisson boundary}
For the Poisson boundary one has \(\Phi_{B}(b)=\delta_{b}\).
\end{example}

\begin{definition}
A factor \(\pi:\mathcal{X}\to\mathcal{Y}\) between regular
\(\boldsymbol{\sigma}\)-systems is called \emph{proximal} if for a.e. \(b\in B\) we have that \(\pi:(X,\Phi_{X}(b))\to(Y,\Phi_{Y}(b))\) is an isomorphism (as usual, modulo null sets).\\
We say that \(\mathcal{X}\) is \emph{proximal} if \(\mathcal{X}\to pt\) is proximal.
\end{definition}
\begin{remark}
By definition, a system is proximal iff the conditional measures are a.e. delta measures.
\end{remark}

\begin{lemma}\label{lemma proximal systems}
A regular
\(\boldsymbol{\sigma}\)-system is proximal iff it is a factor of the Poisson boundary.\\
In particular, for any regular \(\boldsymbol{\sigma}\)-system \(\mathcal{X}\), the \(\boldsymbol{\sigma}\)-system \(\Pi(\mathcal{X})\) is proximal.
\end{lemma}

\begin{proof}
If a regular \(\boldsymbol{\sigma}\)-system is a factor of the Poisson boundary, then by the functionality of conditional measures and Example \ref{Example conditional measures of Poisson boundary} we conclude that the system is proximal.\\
On the other direction, if a regular system \(\mathcal{X}\) is proximal then for a.e. \(b\in B\) we have \(\Phi_{X}(b)=\delta_{\pi(b)}\). The mapping \(\pi:B\to X\) is measurable and \(G\)-equivariant and the integral decomposition in Theorem \ref{Theorem conditional measures for regular systems} yields that \(\pi\) is a factor.
\end{proof}
\begin{remark}
Note that the proof of Lemma \ref{lemma proximal systems} implies that being a factor of the Poisson boundary is a property of the system and not an extra structure. \\
It follows that proximal systems does not posses nontrivial endomorphisms.
\end{remark}

We need the following easy lemma in measure theory (no group action) relating disintegration and Radon-Nikodym derivative.
\begin{lemma}\label{Lemma disintegration and Radon-Nikodym derivative}
Suppose \(\pi: (X,\nu)\to (Y,m)\) is a factor between standard Borel spaces. Let \(\nu_{0}\ll \nu\) and \(m_{0}=\pi_{*}\nu_{0}\). Suppose there is a measurable mapping \(\Phi: Y\to M(X,\Sigma_{X})\) such that:
\begin{itemize}
    \item 
    For \(m\)-a.e. \(y\in Y\) we have \(\pi_{*}\Phi(y)=\delta_{y}\).
    \item
    \(\nu=\intop_{Y} \Phi \: dm\) and \(\nu_{0}=\intop_{Y} \Phi\: dm_{0}\).
\end{itemize}
Then \(\frac{d\nu_{0}}{d\nu}=\frac{dm_0}{dm}\circ \pi\).
\end{lemma}
\begin{proof}
We first show that for any \(f\in L^{\infty}(X,\nu)\:,\: g\in L^{1}(Y,m)\) we have for \(m\) a.e. \(y\in Y\) that
\(\langle\Phi(y),f\cdot (g\circ\pi)\rangle=\langle \Phi(y), f\rangle \cdot g(y) \). Indeed:
\begin{multline*}
|\big\langle\Phi(y),f\cdot (g\circ\pi)\big\rangle-\big\langle \Phi(y), f\big\rangle \cdot g(y)|=|\big\langle \Phi(y), f\cdot (\pi^{*}(g-g(y)))\big\rangle|\leq \big\langle \Phi(y), |f\cdot (\pi^{*}(g-g(y)))|\big\rangle \leq\\
||f||_{\infty} \cdot \big\langle \Phi(y), \pi^{*}(|g-g(y)|)\big\rangle= ||f||_{\infty}\cdot \big\langle \pi_{*}\Phi(y), |g-g(y)|\big\rangle=||f||_{\infty}\cdot \big\langle \delta_{y}, |g-g(y)|\big\rangle=0
\end{multline*}
Thus, for any \(f\in L^{\infty}(X,\nu)\):
\begin{multline*}
\intop_{X}f \:d\nu_{0}=\intop_{Y} \langle\Phi(y),f\rangle\:dm_{0}(y)=\intop_{Y} \langle\Phi(y),f\rangle (y)\cdot \frac{dm_0}{dm}(y)\:\: dm(y)=\\
\intop_{Y}\big\langle\Phi(y),f\cdot (\frac{dm_0}{dm}\circ \pi)\big\rangle \:dm(y)= \intop_{X} \: f\cdot (\frac{dm_0}{dm}\circ\pi)\:d\nu
\end{multline*}
proving the result.
\end{proof}
The following will be used in section \ref{section Realizations via ultralimit} in order to prove that a certain ultralimit construction produces the Poisson boundary.
\begin{prop}\label{proposition maximality of Poisson boundary}
If \(\pi:\mathcal{X}\to \mathfrak{B}(G,\boldsymbol{\sigma})\) is a factor, then \(\pi\) is a measure preserving extension.
\end{prop}
\begin{proof}
Replacing \(\mathcal{X}\) by a measure preserving regular factor (that still has the Poisson boundary as a factor), we may assume \(\mathcal{X}\) is regular.
Note that \(\Phi_{X}: B\to M(X,\Sigma_{X})\) is a measurable function such that a.e. \(\pi_{*}\Phi_{X}(b)=\delta_{b}\). Using the integral formula in Theorem \ref{Theorem conditional measures for regular systems} and  applying Lemma \ref{Lemma disintegration and Radon-Nikodym derivative} we conclude that \(\pi\) is measure preserving.
\end{proof}

We conclude this section with a construction for Joining and a Furstenberg-Glasner type theorem (strongly inspired from \cite{FurstenbergGlasner}).

\begin{definition}
Let \(\mathcal{X}=(X,\boldsymbol{\nu}),\mathcal{Y}=(Y,\boldsymbol{\kappa})\) be two regular \(\boldsymbol{\sigma}\)-systems. The \emph{Joining} of \(\mathcal{X},\mathcal{Y}\) is \(\mathcal{X}\curlyvee\mathcal{Y}=(X\times Y,\boldsymbol{\rho})\) where \(X\times Y\) is equipped with the diagonal action, and the vectors of measures \(\boldsymbol{\rho}=(\rho^{(n)})\) are defined by:
\[\rho^{(n)}=\intop_{B} \Phi_{X}(b)\times \Phi_{Y}(b) \: dm^{(n)}(b)\]
\end{definition}
Note that \[g\:\rho^{(n)}=\intop_{B} g\big(\Phi_{X}\times\Phi_{Y}\big) \: dm^{(n)}=\intop_{B} \Phi_{X}\times\Phi_{Y}\: d\big(g\:m^{(n)}\big)\]
Thus, \(\mathcal{X}\curlyvee\mathcal{Y}\) is a regular \(\boldsymbol{\sigma}\)-system. We sometimes denote \(\boldsymbol{\rho}=\boldsymbol{\nu}\curlyvee \boldsymbol{\kappa}\).\\
By definition, the conditional measures of \(\mathcal{X}\curlyvee \mathcal{Y}\) are \(\Phi_{\mathcal{X}\curlyvee \mathcal{Y}}(b)=\Phi_{\mathcal{X}}(b)\times \Phi_{\mathcal{Y}}(b)\).\\
Also, it is clear that the projections \(\mathcal{X}\curlyvee\mathcal{Y}\to\mathcal{X},\mathcal{X}\curlyvee\mathcal{Y}\to\mathcal{Y}\) are factors of \(\boldsymbol{\sigma}\)-systems.

\begin{lemma}\label{lemma joining with proximal}
Let \(\mathcal{X},\mathcal{Y}\) be regular \(\boldsymbol{\sigma}\)-systems. If \(\mathcal{X}\) proximal then the projection \(\pi: \mathcal{X}\curlyvee\mathcal{Y}\to\mathcal{Y}\) is proximal. 
\end{lemma}
\begin{proof}
We need to show that the projection \(\pi:(X\times Y, \delta_{p(b)}\times\Phi_{Y}(b))\to (Y,\Phi_{Y}(b))\) is an isomorphism. Indeed the inverse is \(y\mapsto (p(b),y)\).
\end{proof}

Now we can prove the corresponding Furstenberg-Glasner type theorem:

\begin{theorem}\label{Theorem Furstenberg Glasner}
Let \(\mathcal{X}\) be a regular \(\boldsymbol{\sigma}\)-system. Let \(\Lambda(\mathcal{X})\) be a proximal \(\boldsymbol{\sigma}\)-system, having \(\Pi(\mathcal{X})\) as a factor. Consider \(\mathcal{X}^*=\mathcal{X}\curlyvee \Lambda(\mathcal{X})\). Then we have:
\begin{enumerate}
    \item 
    \(\mathcal{X}^*\to\mathcal{X}\) is a proximal extension.
    \item
    \(\mathcal{X}^*\to \Lambda(\mathcal{X})\) is a measure preserving extension.
\end{enumerate}
\end{theorem}
\begin{proof}
Since \(\Lambda(\mathcal{X})\) is proximal, we conclude item 1 from Lemma \ref{lemma joining with proximal}.\\
For item 2: we write \(\mathcal{X}=(X,\boldsymbol{\nu}),\Pi(\mathcal{X})=(M(X,\Sigma),\boldsymbol{\kappa}),\Lambda(\mathcal{X})=(\Lambda,\boldsymbol{\lambda}),\mathcal{X}^*=(X\times \Lambda,\boldsymbol{\rho})\).\\
Let \(\varphi:\Lambda(\mathcal{X})\to\Pi(\mathcal{X})\) be the given factor and \(\pi:\mathcal{B}(G,\boldsymbol{\sigma})\to \Lambda(\mathcal{X})\) be the factor provided by the proximailty of \(\Lambda(\mathcal{X})\) and Lemma \ref{lemma proximal systems}. From the same lemma we know that the factor from the Poisson boundary to a proximal system is unique. Considering the composition
\[\mathcal{B}(G,\boldsymbol{\sigma})\stackrel{\pi}{\longrightarrow}\Lambda(\mathcal{X})\stackrel{\varphi}{\longrightarrow}\Pi(\mathcal{X})\] 
we conclude \(\varphi\circ\pi=\Phi_{\mathcal{X}}\) by the definition of \(\Pi(\mathcal{X})\).
Thus for any \(n\geq -1\:,\:i\in[\ell_{n}]\:,\: g\in G\):
\[g\rho_{i}^{(n)}=\intop_{B} \Phi_{\mathcal{X}}\times \Phi_{\Lambda(\mathcal{X})}\: d(gm_{i}^{(n)})=\intop_{B} \varphi(\pi(b))\times\delta_{\pi(b)}d(gm_{i}^{(n)})(b)=\intop_{\Lambda} \varphi(\xi) \times \delta_\xi \: d(g\lambda_{i}^{(n)})(\xi)\]
Using Lemma \ref{Lemma disintegration and Radon-Nikodym derivative} we conclude that the extension \((X\times \Lambda,\boldsymbol{\rho})\to (\Lambda,\boldsymbol{\lambda})\) is measure preserving.
\end{proof}


\printbibliography

\end{document}